\documentclass{TPmod}
\DeclareMathAlphabet{\mathpzc}{OT1}{pzc}{m}{it}
\usepackage{url}
\usepackage{setspace}
\usepackage{scrextend}
\usepackage{mathrsfs}
\usepackage{tocloft}
\usepackage{empheq}
\usepackage{pigpen}
\usepackage{tikz}

\newcommand{\po}{\ar@{}[dr]|{\text{\pigpenfont R}}}
\newcommand{\pb}{\ar@{}[dr]|{\text{\pigpenfont J}}}

\usetikzlibrary{decorations.markings,arrows}
\usepackage{xcolor}

\mathchardef\mhyphen="2D 

\makeatletter
\DeclareRobustCommand{\hvec}[1]{{\mathpalette\hvec@{#1}}}
\newcommand{\hvec@}[2]{
  \vbox{\offinterlineskip
    \ialign{
      \hfil##\hfil\cr
      $\m@th#1{}_{\rightharpoonup}$\kern-\scriptspace\cr
      $\m@th#1#2$\cr
    }
  }
}
\makeatother

\usepackage{graphicx,calc}
\newlength\myheight
\newlength\mydepth
\settototalheight\myheight{Xygp}
\usepackage{amsthm}
\usepackage{amssymb}
\usepackage[capitalize]{cleveref}
\usepackage{mathtools}
\usepackage{xy}
\usepackage{tikz}
\usepackage{tikz-cd}

\usepackage{cancel}

\usepackage[normalem]{ulem}
\usetikzlibrary{shapes.geometric,arrows}
\usetikzlibrary{calc, arrows,decorations.pathmorphing,backgrounds,fit,positioning,shapes.symbols,chains,snakes}
\usetikzlibrary{decorations.pathreplacing,decorations.markings}
\usetikzlibrary{decorations.pathreplacing,calligraphy}

\usetikzlibrary{arrows.meta}

\DeclareMathOperator{\Aut}{Aut}

\newcommand{\st}{\mathrm{st}}

\newcommand{\hocolim}{\mathrm{hocolim}}
\newcommand{\colim}{\mathrm{colim}}
\newcommand{\hofib}{\mathrm{hofib}}

\usepackage{xcolor}

\DeclareMathOperator{\loc}{loc}

\DeclareMathOperator{\Diff}{Diff}

\newcommand{\pt}{\operatorname{pt}}

\def\bC{\mathbb{C}}
\def\bD{\mathbb{D}}
\def\bE{\mathbb{E}}
\def\bF{\mathbb{F}}

\def\bH{\mathbb{H}}

\def\bO{\mathbb{O}}
\def\bP{\mathbb{P}}

\def\bR{\mathbb{R}}
\def\bS{\mathbb{S}}

\def\bU{\mathbb{U}}
\def\bV{\mathbb{V}}

\def\bZ{\mathbb{Z}}

\def\fg{\mathfrak{g}}

\DeclareMathOperator{\Thom}{\mathsf{Thom}}
\DeclareMathOperator{\Base}{\mathsf{Base}}

\newcommand{\cA}{\mathcal{A}}
\newcommand{\cB}{\mathcal{B}}
\newcommand{\cC}{\mathcal{C}}
\newcommand{\cD}{\mathcal{D}}
\newcommand{\cE}{\mathcal{E}}
\newcommand{\cF}{\mathcal{F}}
\newcommand{\cG}{\mathcal{G}}
\newcommand{\cH}{\mathcal{H}}
\newcommand{\cI}{\mathcal{I}}

\newcommand{\cL}{\mathcal{L}}
\newcommand{\cM}{\mathcal{M}}
\newcommand{\cN}{\mathcal{N}}
\newcommand{\cO}{\mathcal{O}}
\newcommand{\cP}{\mathcal{P}}

\newcommand{\cR}{\mathcal{R}}
\newcommand{\cS}{\mathcal{S}}

\newcommand{\cV}{\mathcal{V}}
\newcommand{\cW}{\mathcal{W}}
\newcommand{\cX}{\mathcal{X}}

\def\bZ{\mathbb{Z}}

\tikzset{
    ncbar angle/.initial=90,
    ncbar/.style={
        to path=(\tikztostart)
        -- ($(\tikztostart)!#1!\pgfkeysvalueof{/tikz/ncbar angle}:(\tikztotarget)$)
        -- ($(\tikztotarget)!($(\tikztostart)!#1!\pgfkeysvalueof{/tikz/ncbar angle}:(\tikztotarget)$)!\pgfkeysvalueof{/tikz/ncbar angle}:(\tikztostart)$)
        -- (\tikztotarget)
    },
    ncbar/.default=0.5cm,
}

\tikzset{square left brace/.style={ncbar=0.5cm}}
\tikzset{square right brace/.style={ncbar=-0.5cm}}

\tikzset{round left paren/.style={ncbar=0.5cm,out=120,in=-120}}
\tikzset{round right paren/.style={ncbar=0.5cm,out=60,in=-60}}

\renewcommand{\cong}{\simeq}

\numberwithin{equation}{section} 
\renewcommand{\theequation}{\thesection.\arabic{equation}}

\begin{document}

\title{Bordism from quasi-isomorphism}

\author[]{Noah Porcelli\footnote{MPIM Bonn, Germany\\ E-mail: porcelli@mpim-bonn.mpg.de}, Ivan Smith\footnote{University of Cambridge, United Kingdom\\ E-mail: is200@cam.ac.uk}}
\thanks{The Clynelish distillery.}
\date{September 2024. Unauthorised readers will be pestered by pufflings.}

\address{Noah Porcelli, Max Planck Institute for Mathematics, Vivatsgasse 7, 53111 Bonn, Germany}
\address{Ivan Smith, Centre for Mathematical Sciences, University of Cambridge, Wilberforce Road, Cambridge CB3 0WB, U.K.}


\begin{abstract} {\sc Abstract:}
    Let $X$ be a graded Liouville domain. Fix a pair of infinite loop spaces $\Psi = (\Theta \to \Phi)$ living over $(BO \to BU)$. This determines a spectral Fukaya category $\scrF(X;\Psi)$ whenever $TX$ lifts to $\Phi$, containing closed exact Lagrangians $L$ for which $TL$ lifts compatibly to $\Theta$; and by Bott periodicity and index theory, a Thom spectrum $R$ with bordism theory $R_*$.

Suppose that $L$ and $K$ are quasi-isomorphic in the Fukaya category over $\bZ$. We prove that: 

(a) if both lift to $\scrF(X;\Psi)$, then there is a rank one $R$-local system $\xi: L \to BGL_1(R)$ over $L$ so that $(L,\xi)$ and $K$ are quasi-isomorphic in the spectral Fukaya category; \newline
(b) when $X$ is polarised and $\Psi = (BO \times F \to BO)$, if only $K$ lifts to $\scrF(X;\Psi)$, then the composition $L \to B^2GL_1(R)$ of the stable Gauss map of $L$ and the delooped $J$-homomorphism is nullhomotopic. 

Combined with the computation of the open-closed fundamental class associated to $(L,\xi)$  in \cite{PS3}, these results have applications to bordism and stable homotopy types of quasi-isomorphic Lagrangians, to Hamiltonian monodromy groups, and to  smooth structures on nearby Lagrangians.  

A key ingredient in the proofs  is a new form of obstruction theory for flow categories `lying over' a manifold $L$, closely related to a `spectral Viterbo restriction functor' also introduced here.

\end{abstract}

\maketitle
\thispagestyle{empty}
\setlength{\cftbeforesubsecskip}{-2pt}
\tableofcontents

\section{Introduction}

\subsection{Context}

A fundamental question in symplectic topology is to constrain the Lagrangian submanifolds $L \subset X$ of a symplectic manifold $X$; both their intrinsic topology (homotopy type, smooth structure,...) and the way they lie in $X$ (their homology class, knot type,...). This paper resolves the long-standing question of determining the bordism-theoretic constraints on an exact Lagrangian submanifold determined by its quasi-isomorphism class in the classical Fukaya category. This question was one of the motivating goals for the development of Floer homotopy theory.

Indeed, constraints on the homology class of a Lagrangian often arise from Floer theory: two closed exact Lagrangians which are equivalent in the Fukaya category realise the same fundamental homology class. When $L$ carries some further tangential structure, like a $Spin$ structure or framing, it is natural to seek to constrain the associated $Spin$ or framed bordism class.  This paper, when combined with its prequel \cite{PS3}, explains how to derive such bordism-theoretic constraints by combining classical Floer theory with Floer homotopy theory.

In \cite{PS}, building on work of \cite{Large,FO3:smoothness}, we constructed a spectral Fukaya category for a stably framed Liouville manifold, with objects being compatibly stably framed compact exact Lagrangian submanifolds.  The sequel \cite{PS2} constructed an analogous spectral Fukaya category for any tangential structure $\Psi = (\Theta \to \Phi)$ living over $BO \to BU$, thus allowing us to access information contained in more general bordism theories. Here the symplectic manifold $X$ is assumed to have a lift of the classifying map for $TX$ to $\Phi$, and objects of the category are compact exact Lagrangians equipped with lifts of their tangent bundles to $\Theta$.   

In \cite{PS3} we revisited the constructions of \cite{PS,PS2} in the special case where the pair $(\Theta \to \Phi)$ is commutative, so forms a diagram of infinite loop spaces over $(BO \to BU)$. In that case, we showed that one can incorporate rank one spectral local systems into the category $\scrF(X;\Psi)$, and one furthermore retains control over the bordism class associated to a Lagrangian equipped with such a local system. It is this refined version of the spectral Fukaya category which is required essentially in this paper.

In both \cite{PS,PS2} there is an obstruction theory constraining when a quasi-isomorphism of branes $L$ and $K$ in $\scrF(X;\bZ)$ can be lifted to $\scrF(X;\Psi)$. The obstructions are typically non-vanishing; \cite{PS} discussed cases when some obstructions could be killed by passing to a quotient of the sphere spectrum, and \cite{PS2} worked in settings where obstructions vanished for degree reasons. In both cases, our applications were were limited to the cases where we could bypass said obstructions in one of these two ways. The underlying obstruction theory was in some sense classical in nature, related to the truncations of a spectrum arising from the choice of tangential structure $\Psi$.

The primary goal in this paper is to show that one can \emph{always} lift quasi-isomorphisms to spectra after incorporating (rank one) spectral local systems into the Fukaya category. This relies on a fundamentally new flavour of obstruction theory, not related to the truncations of an auxiliary spectrum, but coming from a geometric handle decomposition of the Lagrangian $L$, and direct arguments with Morse models for certain Floer moduli spaces.

One now encounters two obstructions at each stage of an argument -- an obstruction to extending a spectral local system over a handle, and then an obstruction to extending the desired quasi-isomorphism.  A crucial feature of our approach is that one can control both obstructions simultaneously at each stage of the induction. It is also worth emphasising that the new obstruction theory relies essentially on having access to spectral local systems, which in turn requires commutativity of the tangential pair, and hence relies on the particular model for the spectral Fukaya category developed in the prequel \cite{PS3}.

Together with the computation of the open-closed fundamental class from \cite{PS3}, this has numerous applications to symplectic topology which are detailed in Section \ref{sec: applications}.

\subsection{Results}

We work with the version of the Fukaya category constructed in \cite{PS3} and associated to a \emph{commutative} tangential pair $\Psi$, meaning the data of a commutative diagram of infinite loop spaces:\footnote{We impose some additional constraints, corresponding to requiring the ordinary Fukaya category to be $\bZ$-graded and defined over $\bZ$.}
\[
\xymatrix{
\Theta \ar[r] \ar[d] & \Phi \ar[d] \\
BO \ar[r] & BU
}
\]

Let $F$ be the homotopy fibre of the map $\Theta \to \Phi$. Using real Bott periodicity, we may define a commutative ring spectrum $R_\Psi$, as the Thom spectrum of an index bundle defined by the composition $\Omega F \to \Omega\widetilde{U/O} \xrightarrow{\mathrm{Bott}} BO$, see Section \ref{sec:tp and ad}. Its base, which we write as $E_\Psi$, is homotopy equivalent to $\Omega F$.

\begin{ex}
    Taking $\Theta=\Phi$ to be contractible corresponds to a Fukaya category of a framed symplectic manifold and its compatibly framed Lagrangians Lagrangians. Taking $\Theta=(BO, BO)$ corresponds to a compatible polarisation on symplectic manifold and Lagrangians.
    In both cases, $R_\Psi$ is the sphere spectrum.
\end{ex}

Analogously to ordinary rings, $R_\Psi$ has a `group of units' $GL_1(R_\Psi)$, and we define a \emph{$\Psi$-local system} over a base $B$ to be  a map $B \to BGL_1(R_\Psi)$. As with ordinary rings, this serves as a model for rank-1 $R_\Psi$-local systems.

Given a spectral lift of the Fukaya category, a natural question to ask is whether this can distinguish between any Lagrangians that the ordinary Fukaya category could not. We answer this in the negative: we prove that, essentially whenever it makes sense, quasi-isomorphisms in the ordinary Fukaya category (i.e. over $\bZ$) can be `lifted' to the spectral setting, at the expense of turning on a $\Psi$-local system.

    \begin{thm} \label{thm:main}
        Let $X$ be a graded Liouville domain, and $L, K \subseteq X$ closed exact graded Lagrangians.

        Let $\Psi=(\Theta,\Phi)$ be an oriented commutative tangential pair, and assume that $X$ admits a $\Phi$-orientation $\phi$, and $L$ and $K$ admit compatible $\Theta$-orientations. 
    
        Assume that $L$ and $K$ are quasi-isomorphic in the ordinary Fukaya category $\scrF(X, b^{can}; \bZ)$. Then there is a  $\Psi$-local system $\xi$ on $L$, such that $(L, \xi)$ and $K$ are isomorphic in the spectral Fukaya category $\scrF(X; \Psi)$.
    \end{thm}

\begin{rmk}
    Recall that the ordinary Fukaya category depends on a choice of \emph{background class} $b \in H^2(X; \bZ/2)$, and for a Lagrangian to define an object it must be relatively Spin relative to $b$; we write $\scrF(X, b; \bZ)$ for the resulting category.

    Because $\Psi$ is oriented, a $\Phi$-orientation on $X$ picks out a distinguished background class $b^{can}$, such that $\Theta$-oriented Lagrangians are relatively Spin relative to this class $b^{can}$ (cf. Section \ref{sec:Fuk and OC}). The ordinary Fukaya category in Theorem \ref{thm:main} is taken with respect to $b^{can}$.
\end{rmk}

    \begin{rmk}
        A version of this theorem in the case $X=T^*Q$ is a cotangent bundle, $L = Q$ is the zero-section and $K$ is a nearby Lagrangian was obtained in \cite{ADP}. They  use a significantly different set-up and proof, which appears to be fundamentally tied to  working with cotangent bundles. 
    \end{rmk}

This paper contains one other general structural result, constraining the stable Gauss maps of quasi-isomorphic Lagrangians. Recall that if $X$ is stably polarised (meaning $TX \oplus \bC^k \cong V \otimes \bC$ for some real vector bundle $V \to X$), the \emph{stable Gauss map} of a Lagrangian $L$, $g_L: L \to U/O$, determines whether $TL$ agrees (stably) with $V|_L$, as totally real subbundles of $TX|_L$. It is the stable Gauss map that determines which versions of the spectral Fukaya category $L$ determines an object in. In this setting, the canonical background class is $b^{can} = w_2(V)$. 

\begin{defn}\label{defn: rj}
    Let $\bS$ be the sphere spectrum. The \emph{$J$-homomorphism} $J: O \to GL_1(\bS)$ is defined roughly to send an orthogonal transformation $f: \bR^n \to \bR^n$ to the map induced by 1-point compactifying $(f: S^n \to S^n) \in \mathrm{hAut}(S^n)$; composing with the unit defines a map $J_R: O \to GL_1(R)$ for any commutative ring spectrum $R$. This is a map of infinite loop spaces, and so can be delooped.
\end{defn}

\begin{thm}\label{thm:main3} 
   Let $X$ be a stably polarised Liouville domain and $L, K \subseteq X$ closed exact relatively Spin graded Lagrangians. Let $\Psi = (\Theta = BO \times F, \Phi = BO)$ (cf. Example \ref{qi ex: tang str pol}), for some infinite loop space space $F$ admitting a map to $\widetilde{U/O}$.

   Assume $K$ admits a compatible $\Theta$-orientation, and that $L$ and $K$ are isomorphic in the ordinary Fukaya category $\scrF(X, b; \bZ)$. Then the following composition is nullhomotopic:
   \begin{equation}\label{eq: irwjgoergnieovn}
       L \xrightarrow{g_L} U/O \xrightarrow[\simeq]{\mathrm{Bott}} B^2 O \xrightarrow{B^2 J_{R_\Psi}} B^2 GL_1(R_\Psi)
   \end{equation}
\end{thm}
Theorem \ref{thm:main3} implies that if $X$ is Weinstein then $L$ determines an object of the $R_\Psi$-linear microlocal sheaf category defined by Nadler-Shende \cite{Nadler-Shende}.

Theorem \ref{thm:main3} should be viewed as a form of rigidity for `higher' Maslov classes; compare with the rigidity result \cite{Jack:Tiny} on the ordinary Maslov class of Lagrangians.

\begin{rmk}
    Cotangent bundles $T^*Q$ admit a canonical polarisation. In the case $F=\pt$ (so $R_\Psi \simeq \bS$), Jin \cite{Xin} and Abouzaid-Courte-Guillermou-Kragh \cite{ACGK} showed that for a nearby Lagrangian $L \subseteq T^*Q$, the composition $L \to B^2GL_1(\bS)$ of the stable Gauss map with the delooped $J$ homomorphism is nullhomotopic; applying Theorem \ref{thm:main3} to $L$ gives another proof of this fact\footnote{Though note that the composition (\ref{eq: irwjgoergnieovn}) depends on an a choice of Bott periodicity equivalence $\widetilde{U/O} \simeq B^2 O$. The results in \cite{Xin}, \cite{ACGK} and Theorem \ref{thm:main3} all use different models for this equivalence, which to our knowledge have not yet been shown to agree (though it seems likely that they do). }. They use different methods (microlocal sheaves and generating functions, respectively), which appear less well suited to working with general Liouville domains.
\end{rmk}
\begin{rmk}
    Following the ideas of \cite{AK}, from Theorem \ref{thm:main3} and the Gromov-Lees $h$-principle one can obtain constraints on immersion classes of Lagrangians in plumbings (e.g. the $A_2$-Milnor fibre or the affine flag $3$-fold, in each of which we have quasi-isomorphism classifications in the integral Fukaya category). These  are likely some of the first such constraints away from cotangent bundles; since this is a variation on themes covered by our other applications, we leave details to the interested reader.
\end{rmk}

\subsection{Methods of proof}

In the proofs of Theorems \ref{thm:main} and \ref{thm:main3}, we make systematic use of flow categories \emph{over} a manifold $L$ (possibly $E_\Psi$-orientated, see Section \ref{sec:E-orientations recap}). This is a flow category $\cF$ equipped with maps $ev: \cF_{xy} \to \Omega L$, which are compatible with concatenation. 

Lagrangian Floer flow categories naturally have this extra structure: the flow category $\cM^{LK}$ of a pair of Lagrangians $L$ and $K$ has such maps by restricting a holomorphic strip to one of its boundary components. This notion was also used in \cite{Barraud-Cornea,Abouzaid:based_loops, BDHO}, and is used both in the set-up and proof of main theorem in \cite{PS3}. 

By counting the 0-dimensional manifolds in $\cF$ and keeping track of the class they determine in $\pi_1 L$, we obtain a chain complex $CM_*(\cF; \bZ[\pi_1 L])$ over $\bZ[\pi_1 L]$. There is a category of ($E_{\Psi}$-oriented) flow categories over $L$, and our construction (cf. Section \ref{sec: Vit}) yield a Viterbo-style restriction functor
    \[
    \scrF(X;\Psi) \longrightarrow {\Flow}^{E_{\Psi}}_{/L}.
    \]

\begin{rmk}
   There is also a truncation of $\Flow_{/L}$ which recovers $C_*(\Omega L)\mhyphen$mod, making contact with well known ideas of  Floer theory with `universal local systems', cf. for instance \cite{BDHO, BDHO2}.
\end{rmk}

    The starting point for the proof of Theorem \ref{thm:main} in Section \ref{sec:ELS} is that if $L$ and $K$ are quasiisomorphic over $\bZ$, then $CM_*(\cM^{LK}; \bZ[\pi_1 L])$ is quasiisomorphic to the Morse chain complex of some $\bZ$-local system over $L$ with $\bZ[\pi_1 L]$-coefficients. We call such a flow category over $L$ an \emph{evaluation local system (ELS)} (to contrast with this, we often refer to a spectral local system  $L \to BGL_1(R)$ as a \emph{monodromy local system}). We classify evaluation local systems over $L$, showing that they are all isomorphic to a Morse flow category twisted by a monodromy local system $\xi_L \in [L, BGL_1(R_\Psi)]$ (cf. Theorem \ref{thm: WELS}). This involves an induction over a handle decomposition of $L$. At each stage of the induction, existence of such an ELS forces the relevant obstruction to vanish; this involves a different kind of obstruction theory to that of \cite{PS, PS2} (although that latter obstruction theory does also play a minor role in the proof). 
    
    This classification gives us strong control over $\cM^{LK}$, and we use this to show that the obstructions to lifting a quasiisomorphism over $\bZ$ to the spectral Fukaya identified in \cite{PS,PS2} vanish after twisting by the relevant local system.

    \begin{rmk}
        At one point in the inductive argument we use a Morse-Bott model, and so the smoothness of moduli spaces required in this paper requires deeper estimates from \cite{FO3:smoothness} than were used in our earlier work. This is reviewed and explained in Section \ref{sec:morse-floer}.
    \end{rmk}

    We now switch to the setting of Theorem \ref{thm:main3}. Again $\cM^{LK}$ is a flow category over $L$, but if only $K$ has a $\Theta$-orientation, $\cM^{LK}$ is not $E_\Psi$-oriented but instead satisfies a related condition, `twisted' by the stable Gauss map $g_L$ of $L$ (cf. Section \ref{sec: gaus twis eval loca syst}); we call such flow categories \emph{Gauss-twisted}. Again the hypothesis $L \cong K$ over $\bZ$ forces $\cM^{LK}$ to be a Gauss-twisted ELS, and inducting over a handle decomposition of $L$, we show existence of such a Gauss-twisted ELS forces the composition of $g_L$ with the delooped $J$-homomorphism $B^2J_{R_\Psi}$ to be null. Though structurally similar to the above proof, killing the obstruction at each stage is different.

    \begin{rmk}
        In all results in this paper, commutativity of the tangential pair is essential. In Theorem \ref{thm:main}, this is required in the obstruction theory, see Remark \ref{rmk: epughrdouighrdpsghdisrg}. In Theorem \ref{thm:main3}, commutativity is required to be able to deloop $BGL_1(R_\Psi)$ and hence to even state the result.
    \end{rmk}

\subsection{Sample applications}

We state some consequences of Theorems \ref{thm:main},  \ref{thm:main3} in tandem with \cite[Theorem 1.3]{PS3}, which we recall later in Theorem \ref{thm: gour main theorem}.

The proofs of these (and a number of other applications not listed here, including to symplectic mapping class groups and immersed Lagrangian cobordism) occupy Section \ref{sec: applications}.  

\subsubsection*{Bordism: Lagrangian, complex, framed}

Let $L, K \subseteq X$ be closed exact Lagrangians in a graded Liouville domain $X$, which define isomorphic objects in $\scrF(X; \bZ)$.
\begin{cor}
    Assume $L$ and $K$ are stably framed. Then there is a stably complex cobordism between $L$ and $K$, lying over $X$.

\end{cor}
\begin{rmk}
    Note the conclusion here is that the bordism class $[L]-[K] \in MU_d(X)$ vanishes, not that it is 2-torsion: for $MU$, we use chromatic homotopy theory and  work of Bousfield-Kuhn and Hovey \cite{Bousfield,Kuhn,Hovey} to show that $\eta$ is nullhomotopic ($p$-locally for all $p$) rather than just 2-torsion (see Section \ref{sec:Fuk and OC} for a recap of the role of $\eta$).
\end{rmk}
\begin{rmk}
    Any closed framed manifold bounds a stably complex manifold; the content here is that the cobordism lies over $X$.
\end{rmk}

Even in the case of the tangential pair $\Psi = (pt \to pt)$ corresponding to stable framings, and framed bordism,  our results have new implications.  For instance, we show:

\begin{cor}
 Suppose $L$ and $K$ are stably framed compatibly with an ambient stable framing of $X$. If $L$ and $K$ are quasi-isomorphic in $\scrF(X;\bZ)$, then they are stably homotopy equivalent.  
\end{cor}

In particular, the isomorphism $H^*(L;\bF_p) \to H^*(K;\bF_p)$ induced by the quasi-isomorphism entwines Steenrod operations.

We emphasise that, unlike in the case of a `nearby Lagrangian', here there is no natural map between $L$ and $K$ in either direction.

\subsubsection*{Normal invariants of nearby Lagrangians}

    Let $d\geq 5$. For a closed smooth $d$-manifold $Q$, the simple structure set $\cS^s(Q)$ is the set of pairs $\{(N,f)\}$ where $N$is a compact smooth $d$-manifold and $f: N \to Q$ is a simple homotopy equivalence, up to the equivalence relation $(N,f) \sim (N',f')$ if there is a diffeomorphism $h: N \to N'$ with $f \sim f'\circ h$. This sits in the `surgery exact sequence'\footnote{An exact sequence of pointed sets.} 
\[
\cdots \to L^s_{d+1}(\pi_1(Q)) \to \cS^s(Q) \stackrel{\sigma}{\longrightarrow} [Q,G/O] \to \cdots
\]
Here we write $G := GL_1(\bS) \simeq \colim_{k \to \infty} \mathrm{hAut}(S^n)$, and $G/O$ is the homotopy quotient of the $J$ homomorphism as in Definition \ref{defn: rj}. The map $\sigma$ sends $(N, f)$ to its \emph{normal invariant}. If $\pi_1(Q) = 1$ then the action of $L^s_{d+1}(1)$  on $\cS^s(Q)$ factors through an action of the group $bP_{d+1}$ of framed-bounding homotopy spheres, via connect sum.  In particular, if $Q$ is simply-connected and $d$ is even (so $bP_{d+1}=0$), then $\sigma(f) = 0$ if and only if $f$ is homotopic to a diffeomorphism. This underscores the importance of the normal invariant in the diffeomorphism classification of manifolds.

\begin{rmk}
    Since $O$ and $G$ are infinite loop spaces (and $J$ is compatible with this structure), $G/O$ is too. By delooping, this allows us to obtain operations on it induced by elements of the stable homotopy groups of spheres.
    
    Explicitly, for $\alpha \in \pi_k\bS$ represented by some map $a: S^{i+k} \to S^i$, we take the composition:
    \begin{equation}
        \alpha_*: B^k(G/O) \simeq \Omega^i B^{i+k}(G/O) \xrightarrow{a^*} \Omega^{i+k} B^{i+k} (G/O) \simeq G/O
    \end{equation}
    where the middle map is pullback along $a$.
\end{rmk}

Let $Q$ be a closed manifold and $L \subseteq T^*Q$ a nearby Lagrangian submanifold (meaning it is closed and exact). Abouzaid and Kragh \cite{Abouzaid:homotopy,Kragh,Abouzaid-Kragh:Simple} proved that the map $j: L \to Q$ is a simple homotopy equivalence, and hence that it has a well-defined normal invariant $[(L, j)] \in [Q, G/O]$.

\begin{thm}\label{thm: norm}
    Let $L \subseteq T^*Q$ be a nearby Lagrangian submanifold. Then the normal invariant $[(L, j)]$ factors through $\eta_*: B(G/O) \to G/O$.
\end{thm}
Since $\eta$ is two-torsion, this implies the normal invariant must be too. A similar statement was proved recently in work of Abouzaid--Álvarez-Gavela--Courte--Kragh \cite{AAGCKarxiv}, using generating functions: they proved factorisation through some 2-torsion map $w: B(G/O) \to G/O$ (though did not identify $w$ with $\eta$, nor is it clear $w$ is a loop map). 

In the special case that $L$ is a homotopy sphere, Theorem \ref{thm: norm} was obtained earlier in an earlier version of \cite{AAGCKarxiv} as well as \cite{PS}.

Theorem \ref{thm:main} implies that $L$ is isomorphic to the zero-section $Q$ equipped with some spectral local system $\xi_Q$, and \cite[Theorem 1.3]{PS3} shows the difference between the open-closed images of $L$ and $(Q, \xi_Q)$ is in the image of $\eta$. To prove the theorem, we then show the open-closed map sends $L$ to its normal invariant, in a suitable bordism-theoretic model for normal invariants, see Section \ref{sec:normal}. Here we work in a spectral Fukaya category with $\Psi=(BO \times \hofib(B^2 J), BO)$, which $L$ lives in by Theorem \ref{thm:main3}.

We record some concrete consequences for nearby Lagrangians in cotangent bundles of manifolds including $T^n$, $S^4 \times S^4$, $\bC\bP^4$ in Section \ref{sec: conc norm}.

\begin{rmk}
   In principle our technology would yield constraints on normal invariants of homotopy equivalent quasi-isomorphic Lagrangians away from the case of cotangent bundles. The obstacle is that we have very few sufficient criteria to guarantee quasi-isomorphism in the \emph{integral} Fukaya category!  
\end{rmk}

\noindent \textbf{\emph{Acknowledgements.}} N.P. is supported by EPSRC standard grant EP/W015889/1.  I.S. is partially supported by UKRI Frontier Research grant EP/X030660/1 (in lieu of an ERC Advanced grant).  We are grateful to Mohammed Abouzaid, Daniel Álvarez-Gavela, Kenneth Blakey, Sylvain Courte, Jeremy Hahn, Alice Hedenlund, Thomas Kragh, Alex Oancea, Oscar Randal-Williams, John Rognes for helpful discussions. N.P. is grateful to the Max Planck Institute for Mathematics in Bonn for its hospitality and financial support. 

\section{Applications}\label{sec: applications}

In this section, we deduce various concrete symplectic consequences of Theorem \ref{thm:main}, \cite[Theorem 1.3]{PS3} (recalled as Theorem \ref{thm: gour main theorem} in Section \ref{sec:Fuk and OC} below) and Theorem \ref{thm:main3}. 

\subsection{Stable homotopy type}

Consider the tangential pair $\Psi = fr = (\pt, \pt)$; then $\scrF(X;\Psi) = \scrF(X;\bS)$ is a version of the Fukaya category over the sphere spectrum from \cite{PS}, but incorporating $\bS$-local systems. Geometrically, this means $X$ is stably framed and spectral branes are compatibly stably framed. 
 \begin{cor}\label{cor:stable homotopy}
        Let $\Psi = fr$, and that $X$ is $\Phi$-oriented and $L, K$ are $\Theta$-oriented.
        
        If $L$ and $K$ are quasi-isomorphic over $\bZ$, then they are stable homotopy equivalent. In particular, the equivalence $H^*(L;\bZ/p)\to H^*(K;\bZ/p)$ defined by the quasi-isomorphism entwines Steenrod operations.
    \end{cor}

    \begin{proof}[Proof of Corollary \ref{cor:stable homotopy}]
        Theorem \ref{thm:main} implies there is a local system $\xi_L$ over $L$ for which $(L,\xi_L) \simeq K$ are equivalent in $\scrF(X;\Psi)$.  It follows that the categories $\cM^{\xi_L,\xi_L}$ and $\cM^{KK}$ are equivalent in the category $\Flow^{fr}$. Passing to $\bS$-mod and using Lemma \ref{lem:22}, the result follows.
    \end{proof}

    Note that, in contrast to Abouzaid's classical theorem that an exact $L\subset T^*Q$ is homotopy equivalent to the zero-section, in Corollary \ref{cor:stable homotopy} there is no obvious map beteen $L$ and $K$ in either direction.

\subsection{Hamiltonian monodromy}

We again take the tangential pair $\Psi = fr$.

Whenever $L^d\subset X^{2d}$ is a Lagrangian submanifold, there are maps
\[
\xymatrix{
\pi_0\mathrm{Ham}(X,L) \ar[r]_{\lambda} & \pi_0\Diff(L) \\ & \Theta_{d+1} \ar[u]_{\iota}
}
\]
where $\mathrm{Ham}(X,L)$ denotes those Hamiltonian symplectomorphisms which preserve $L$ setwise, and $\iota$ includes diffeomorphisms supported in a disc (recall that the mapping class group of the disc $D^n$ is naturally isomorphic to the group of $(n+1)$-dimensional exotic spheres up to $h$-cobordism $\Theta_{n+1}$, via a clutching construction). A special case of the `Hamiltonian monodromy' question asks which elements of $\Theta_{d+1}$ lift to $\mathrm{Ham}(X,L)$.

Suppose now $L$ is a spectral brane with respect to $\Psi$, in particular exact and stably framed. For any such $L$, any element of $\mathrm{Ham}(X,L)$ acts on $HF(L,L) = H^*(L;\bZ)$ trivially, by \cite{Hu-Lalonge-Leclercq,P}.

Recall that there is a map
\[
\Theta_n / bP_{n+1} \longrightarrow \pi_n^{st} / \mathrm{im}(J) = \mathrm{cok}(J)
\]
from the quotient of the group of homotopy spheres by the subgroup of those which framed bound, to the cokernel of the $J$-homomorphism. In the terminology of Section \ref{sec: topo gene}, this sends (the equivalence class of) $\Sigma$ to the normal invariant of the degree-1 map $\Sigma \to S^n$. (The map is injective, and either onto or has index two image, depending on `Kervaire invariant one' considerations.)

\begin{cor}\label{cor: teoihgreiopgsdivn}
    Assume $d \in \{3,5,6,7\}$ mod $8$. If $\alpha \in \Theta_{d+1}$ and $\iota(\alpha)$ lifts to $\pi_0\mathrm{Ham}(X,L)$, then the image of $\alpha$ in $\mathrm{cok}(J)$ vanishes after inverting the prime 2. 

    If $L$ is a homotopy sphere, the image of $\alpha$ in $\mathrm{cok}(J)$ is 2-torsion.
\end{cor}
\begin{proof}
    Let $\phi \in \mathrm{Ham}(X,L)$ be a representative of the lift, so $\phi$ is the time-1 flow of some Hamiltonian. Suspending the exact Lagrangian  isotopy $\phi_t(L)$ gives rise to a `Lagrangian mapping torus', a Lagrangian submanifold $L_\phi \subset X \times T^*S^1$. If $\lambda(\phi) = \iota(f)$ then $L_\phi$ is diffeomorphic to $(L\times S^1)\# \Sigma_\alpha$ where $\Sigma_\alpha$ is the homotopy $(d+1)$-sphere associated to $\alpha$.  
    
    Since $\iota(\alpha)$ is represented by a diffeomorphism that is the identity outside of a small ball, by a Weinstein neighbourhood argument we may assume $\phi|_L$ is the identity outside of a small ball. Therefore $\phi|_L$ is compatible with the polarisation on $L$ outside a small ball, and so $L_\phi$ is polarised (with respect to the product polarisation on $L \times T^*S^1$) outside of a small ball (now of dimension $(d+1)$). By our dimension assumption, $\pi_{d+1} U/O=0$, so we may extend this to the entirety of $L_\phi$. 

    $L \times S^1$ is naturally framed and the above discussion implies $L_\phi$ is too. Moreover, they are diffeomorphic outside of a small ball, and this diffeomorphism preserves framings. It follows that the difference $[L \times S^1]-[L_\phi] \in \Omega^{fr}_d(\pt)$ is exactly the framed bordism class of $[\Sigma_\alpha]$, for some framing on $\Sigma_\alpha$. 
    
    $L \times S^1$ clearly defines an object in $\scrF(X; \Psi)$; the above discussion implies that $L_\phi$ does too. Lemma \ref{lem: eirughrepig} below implies they are quasi-isomorphic in the ordinary Fukaya category $\scrF(X \times T^*S^1; \bZ)$. By Theorem \ref{thm:main}, there is a spectral local system $\xi$ on $L \times S^1$ such that $(L \times S^1, \xi) \cong L_\phi$ in $\scrF(X; \Psi)$; by Theorem \ref{thm: gour main theorem} (and pushing forward along the map $X \times T^*S^1 \to \pt$), the framed bordism classes are related via $[L\times S^1]\cap[\eta\xi] = [L_\phi]$. Following Remark \ref{rmk: recap rmk 1.4 ps3}, we may conclude.
\end{proof}

\begin{lem}\label{lem: eirughrepig} 

    Assume $L$ and $\phi$ are as above, $L\times S^1$ and $L_\phi$ are quasi-isomorphic in $\scrF(X\times T^*S^1; \bZ)$.
\end{lem}
This lemma does not require any constraints on the dimension $d$.

\begin{proof}[Sketch]
    Take a small Morse exact   perturbation of $S^1\subset T^*S^1$ which meets the zero-section in two points. 
    We can suppose that $L_\phi$ maps down to $T^*S^1$ with image contained in the union of the perturbed zero-section and a small disc disjoint from the zero-section (where the Hamiltonian suspension is supported).  In this model, the Lagrangians $L\times S^1$ and $L_\phi$ meet cleanly in two copies of $L$; fixing a Morse function on $L$, the Floer complex 
    \[
    CF(L\times S^1, L_\phi) = C^*_{Morse}(L) \oplus C^{*+1}_{Morse}(L)
    \]  The unit in $C^0_{Morse}(L)$ is a cycle which defines the required quasi-isomorphism.
\end{proof}

\begin{rmk}
    From the perspective of Fukaya categories, in parallel with $\mathrm{Ham}(X, L)$ it is natural to consider the group $\mathrm{Ham}^\Psi(X,L)$ of pairs $(\phi, h)$, where $L$ is $\Theta$-oriented, $\phi \in \mathrm{Ham}(X,L)$ and $h$ is a homotopy between the given $\Theta$-orientation and its pullback under $h$. This is the class of automorphisms to which one expects to be able to associate a spectral Lagrangian Seidel element.

    The homotopy $h$ in such a pair $(\phi, h)$ provides a natural $\Theta$-orientation on the mapping torus $L_\phi$. The same argument as Corollary \ref{cor: teoihgreiopgsdivn} then tells us that without the dimension assumption, in the same setting if $\iota(\alpha)$ lifts to $\pi_0\mathrm{Ham}^\Psi(X,L)$ the image of $\alpha$ in $\mathrm{cok}(J)$ is 2-torsion.
\end{rmk}

There is a pretty variation on this theme in the case when $L = T^n$ is a Lagrangian torus.  For each decomposition $T^n = T^k \times T^{n-k}$ one can consider diffeomorphisms $f\times \id$ where $f\in\Theta_{k+1}$ is supported in a disc on the first factor. Taking the obvious co-ordinate decompositions, and varying $k$, this defines a map
    \begin{equation} \label{eqn:MCG torus}
    \oplus \iota_k: \oplus_{k=0}^n\, (\Theta_{k+1})^{{n \choose k}} \longrightarrow \pi_0\Diff(T^n)
    \end{equation}
    This is close to being injective in large dimensions $n\geq 5$, see \cite[Theorem 4.1]{Hatcher:concordance} and \cite{Hsiang-Sharpe} (for the detailed relationship between the group $\Gamma_i$ that occurs in \cite{Hatcher:concordance} and $bP$, see e.g. \cite{Hsiang-Shaneson}). Moreover, the kernel of the map $\pi_0\Diff(T^n) \to GL(n;\bZ)$ coming from the action on $H_1(T^n;\bZ)$ is $(\bZ/2)^{\infty} \oplus \mathrm{image}(\iota)  \oplus (\bZ/2)^{n \choose 2}$. The first factor is related to pseudo-isotopy theory; the final factor is comprised of diffeomorphisms whose mapping tori are not $PL$-homeomorphic to the standard torus. 

    \begin{cor}\label{cor:torus case} 
        Suppose $n\geq 5$ and $L\subset X$ is an exact Lagrangian torus, stably framed compatibly with an ambient stable framing. Assume there is a retraction $r: X \to L$.

        If an element in $\mathrm{image}(\oplus \iota_k)$ from \eqref{eqn:MCG torus}  lifts to $[\phi] \in \pi_0\mathrm{Ham}(X,T^n)$, then it is $2$-torsion.
    \end{cor}
    \begin{proof} 
        We use the same suspension trick as above. Consider the Lagrangian mapping torus $L_\phi \subseteq X \times T^*S^1$. As before, it is isomorphic to $L \times S^1$ in $\scrF(X \times T^*S^1; \bZ)$. The map $L_\phi \to L \times S^1$ is a homology equivalence; though they are not simply connected, they are both homotopy equivalent to tori, so this map is a homotopy equivalence.
        The result then follows from Proposition \ref{prop: fin}.
    \end{proof}

     The pseudo-isotopy factor $(\bZ/2)^{\infty}$ in $\pi_0\Diff(T^n)$ was studied in \cite{CP}; they showed that  Hamiltonian monodromies which are local (supported in $T^*T^n$) can never give elements in this factor.
    Together with Corollary \ref{cor:torus case}, this puts a significant constraint on which diffeomorphisms can be realised by local exact isotopies of a Lagrangian torus.

\subsubsection*{Monodromy action on tangential structures}

We end the section with a slightly different flavour of application. A pretty consequence of mod two gradings in Floer theory  is that if $L$ is an oriented exact Lagrangian and $\phi \in \mathrm{Ham}(X,L)$, then the induced action of $\phi$ on $L$ preserves its orientation, see \cite{Jack:Tiny}.  One can use Lagrangian suspensions, and Theorem \ref{thm:main3}, to derive similar constraints on how monodromy can act on a tangential structure more complicated than an orientation.

\begin{lem}\label{lem:iergjhoiuerhg}
   Let $\phi \in \mathrm{Ham}(X,L)$, and assume that $X$ is polarised and the stable Gauss map of $L$ is nullhomotopic. The stable Gauss map of the suspension $L_{\phi} \subset X \times T^*S^1$ is determined by the derivative $D\phi$ viewed as an element of $[L,\Omega(\widetilde{U/O})]$.
\end{lem}

\begin{proof}
We deform $\phi$ to preserve a point $p\in L$ and act trivially on $T_pL$. It follows that the (stable) vertical tangent bundle of $L_{\phi}$ is trivialised over a fibre and a section, so is pulled back from the reduced suspension $\Sigma L$, and the stable Gauss map of $L_{\phi}$ is then pulled back along $[L_{\phi},U/O] \to [\Sigma L, U/O]$. The fact that we have trivialised the mapping torus of $\phi$ on $X$ means that $D\phi: L \to O$ lifts to the homotopy fibre of $O \to U$, giving a class in $[L, \Omega (U/O)] = [\Sigma L, U/O]$. Since the vertical tangent bundle of $L_{\phi}$ is clutched by $D\phi$, this class is identified  with the stable Gauss map of the suspension.  Finally, \cite{Jack:Tiny} implies that this stable Gauss map lifts to $\Omega\,\widetilde{U/O}$.
\end{proof}

    \begin{cor}\label{cor: gaus mono}
        Let $L \subseteq X$ be a closed exact Lagrangian in a polarised Liouville domain, and assume the stable Gauss map $g_L$ of $L$ is nullhomotopic.

        Then $\phi$ preserves the nullhomotopy of the composition of $g_L$ with the delooped $J$ homomorphism $\widetilde{U/O} \simeq B^2O \to B^2GL_1(\bS)$.
    \end{cor}
    Under Lemma \ref{lem:iergjhoiuerhg}, this is equivalent to $D\phi$ lying in the kernel of the delooped $J$ homomorphism $\Omega(\widetilde{U/O}) \simeq BO \to BGL_1(\bS)$.
    \begin{proof}
        The stable Gauss map of $L \times S^1$ is nullhomotopic, so Theorem \ref{thm:main3} implies the composition of the stable Gauss map of $L_\phi$ with the delooped $J$ homomorphism is nullhomotopic; the result then follows.
    \end{proof}

\subsection{Non-existence of twists}

For $X^{2d}$ a graded Liouville domain and $L \in \scrF(X; \bZ)$ an object of its derived Fukaya category whose endomorphism algebra $HF_*(L,L)$ is isomorphic to $H^*(S^d)$, there is an autoequivalence $T_L$ of $\scrF(X; \bZ)$, called the \emph{spherical twist} autoequivalence. If instead $d$ is even and $HF_*(L,L) \cong H^*(\bC\bP^{d/2})$, there is similarly a \emph{projective space twist} $T_L$ of $\scrF(X; \bZ)$. If $L$ is realised by a (non-exotic) sphere or complex projective space, this autoequivalence is induced by a compactly supported symplectomorphism of $X$, see \cite{Seidel:book, Mak-Wu}.

In \cite{PS} we showed that, if the spherical twist autoequivalence $T_{\Sigma}$ on a homotopy sphere $\Sigma$ is realised by a compactly supported symplectomorphism of $T^*\Sigma$, then $\Sigma$ has order two in the group $\Theta_d / bP_{d+1}$.  This paper allows one to make similar non-existence statements for twists in homotopy projective spaces.

There is a map $\iota: \Theta_{2k} \to \cS^s(\bC\bP^k)$, sending $\Sigma$ to $\bC\bP^k\#\Sigma$. The \emph{inertia group} of $\bC\bP$, $I(\bC\bP^k)$, is defined to be the preimage of the basepoint; this is a subgroup of $\Theta_{2k}$. These inertia groups have been computed in low degrees, for example it is trivial for $3 \leq k \leq 8$ \cite[Theorem 2.5]{Kasilingam-cpn58}, but non-trivial for $k=9$ \cite[Theorem 3.11]{Basu-Kasilingam:cpn}

\begin{rmk}
    Since $\bC\bP^k$ is even-dimensional, the normal invariant map $\cS^s(\bC\bP^k) \to [\bC\bP^k, G/O]$ is injective. The composition with $\iota$ sends $\Sigma \in \Theta_{2k}$ to the normal invariant of the homotopy equivalence $\bC\bP^k \#\Sigma \to \bC\bP^k$. Since the normal invariant map for $S^{2k}$, $\Theta_{2k} \to [S^{2k}, G/O]$, is also injective, it follows that the inertia group $I(\bC\bP^k)$ naturally embeds into the kernel of the map $[S^{2k}, G/O]\to [\bC\bP^k, G/O]$ obtained by pullback along the degree 1 map $\bC\bP^{2k} \to S^{2k}$. In particular, $I(\bC\bP^k) = I(Y)$ for $Y$ any closed manifold homotopy equivalent to $\bC\bP^k$.
\end{rmk}

Let $\Phi = BO$ with $\Phi \to BU$ the complexification map; $X$ admits a $\Phi$-structure exactly when it admits a stable polarisation. We let $\Theta = BO$; in this case, we can work over $\bS$ and over framed bordism.

\begin{cor} \label{cor:exotic non-twist}
    Let $L = \bC\bP^k \# \Sigma$ for a homotopy $2k$-sphere $\Sigma$. If the projective space twist $T_L$ is realised by a compactly supported symplectomorphism of the cotangent bundle $T^*L$, then $4[\Sigma] \in I(\bC\bP^k)$ lies in the inertia group.
\end{cor}

\begin{rmk}
    Note that it is not true that $\Sigma$ must be diffeomorphic to $S^{2k}$. The best we could hope for is that $[\Sigma] \in I(\bC\bP^k)$, since if this holds $\bC\bP^k\#\Sigma$ is diffeomorphic to $\bC\bP^k$ and hence such a symplectomorphism does exist.
\end{rmk}

\begin{proof}[Proof of Corollary \ref{cor:exotic non-twist}]

    Form the plumbing $X = T^*L \#_{pt} T^*S^{2k}$ by attaching a handle to the boundary of one fibre in $T^*L$ along a standard Legendrian unknot. Let $F = S^{2k}\subset X$ be the second core component, which is the compactification of a fibre of $T^*L$. $X$ has a natural polarisation by a vector bundle $V \to X$ which restricts to the tangent bundle over the two cores $L$ and $F$, and both their stable Gauss maps are trivial.
    
    According to Mak and Wu, see also \cite{Huybrechts-Thomas, HarrisR}, there is an autoequivalence $T_L$ of $\scrF(X;\bZ)$ which acts on objects by a double mapping cone
    \[
    T_L(F) := \mathrm{Cone}\left( \mathrm{Cone}(\xymatrix{L \ar[rr]_-{1\otimes h - h\otimes 1}&& L}) \to F \right)
    \]
    where $h \in CF^2(L,L)$ is a chain-level lift of a generator for $HF^2(L,L) = H^2(L;\bZ) = \bZ$. (When $\Sigma = S^{2k}$ and $L$ is a projective space, this is the categorical action of the geometric symplectomorphism given by the projective space twist; the functor is invertible whether or not it comes from geometry.) 
    
    If $T_L$ is represented by a compactly supported symplectomorphism $\tau_L$ then the LHS is quasi-isomorphic to $\tau_L(F)$, which is diffeomorphic to $F$, whilst the RHS is represented by a Bott Lagrangian surgery diffeomorphic to $F \# \Sigma \# \Sigma$, cf. \cite[Lemma 3.3]{Mak-Wu}. Note that since the stable Gauss map of $F$ is trivial, so is that of its image under a symplectomorphism, so both $F$, $F\#\Sigma\#\Sigma$ define objects in $\scrF(X; fr)$. By Theorem \ref{thm:main},  at the cost of turning on a spectral local system, these two Lagrangians are then quasi-isomorphic in $\scrF(X; fr)$. 
    
    Now consider the open-closed fundamental classes $[F], [F \#\Sigma\#\Sigma] \in \Omega^{fr}_{2k}(\Thom(-V \to X))$. There is a map of spaces $g: X \to L$ collapsing the second plumbing factor; since $T^* S^{2k} \oplus \bR$ is trivial this extends to a map of Thom spaces $g: (X, V) \oplus \bR^1 \to (L, TL) \oplus \bR^1$. Pushing forward along this, we obtain classes $g_*[F], g_*[F\#\Sigma\#\Sigma] \in \Omega^{fr}_{2k}(\Thom(-TL \to L))$. By Lemma \ref{lem: oreigrpieoghdirpgsgr}, these classes represent the fundamental classes of $F$ and $F\#\Sigma\#\Sigma$ over $L$. Connect summing (just as framed manifolds, here we do not use the Lagrange surgery) with $L$ (equipped with its tautological framing relative to $TL$) gives us classes $[L\#F]=[L]+[F] = [L]$ and $[L\#F\#\Sigma\#\Sigma]=[L\#\Sigma\#\Sigma]$ in $\Omega^{fr}_{2k}(\Thom(-TL \to L))$ (using that $F \cong S^{2k}$ with the trivial stable framing). But in the model for normal invariants $\cN\cM(L)$ discussed in Section \ref{sec: topo gene}, $[L]$ represents the trivial normal invariant and $[L\#\Sigma\#\Sigma]$ represents the normal invariant of $L\#\Sigma\#\Sigma$.

    By Theorem \ref{thm: gour main theorem}, $[F]=[F\#\Sigma\#\Sigma] \cap [\eta\xi]$ for some class $[\eta\xi]: F \simeq S^{2k} \to G$ in the image of $\eta$. This implies that under the splitting $\pi^0_{st}(S^{2k})=\bZ \oplus \pi_{2k}^{st}$, $[\eta\xi]$ corresponds to $[\pm 1, \zeta]$ where $\zeta \in \pi_{2k}^{st}$ is some 2-torsion class. Pulling back to $\bC\bP^k$ along the degree-1 map, we find that the normal invariant of $[L\#\Sigma\#\Sigma]$ is 2-torsion.
\end{proof}

\subsection{Complex cobordism}

 For a general commutative ring spectrum $R$, it seems hard to tell when $\eta: BGL_1(R) \to GL_1(R)$ is nullhomotopic, for instance this property is not functorial in any reasonable sense under maps of rings.  For Morava-local spectra, however, the property can be checked readily.  This can be combined with Hovey's splitting of $MU$ from the product of its localisations to obtain the following result.

Let $\Psi$ be the tangential pair where $\Psi=(\Theta=\pt, \Phi=BS_\pm U)$.

\begin{cor}\label{cor:complex cobordism}
    If $L$ and $K$ are $\Psi$-oriented and quasi-isomorphic over $\bZ$, then $[L] = [K] \in MU_*(X)$, i.e. they are complex cobordant over $X$.
\end{cor}

Fix a prime $p$, and write $L_{K(n)}$ for localisation with respect to the height $n$ Morava $K$-theory $K_p(n)$. For $R$ a connective $\bE_{\infty}$-ring, write $gl_1(R)$ for the connective spectrum with $\Omega^{\infty}(gl_1(R)) = GL_1(R)$. We learned the following from Oscar Randal-Williams. 

\begin{lem} \label{lem:BK}
    Suppose $\eta = 0 \in \pi_1(R)$. Then multiplication 
    \[
    \eta: \Sigma gl_1(R) \to gl_1(R)
    \] vanishes after any localisation $L_{K(n)}$. 
\end{lem}

   \begin{proof}
       The Bousfield-Kuhn functor $\Phi_{BK}$ from pointed homotopy types to spectra has the property that $\Phi_{BK}(\Omega^{\infty}(Y)) \simeq L_{K(n)}(Y)$, see \cite{Bousfield,Kuhn}. In particular the localisation $L_{K(n)}(Y)$ of a spectrum $Y$ is completely determined by the identity component $\Omega^{\infty}_0$ of $\Omega^{\infty}(Y)$ \emph{as a space}. The spaces $\Omega^{\infty}_0(R)$ and $\Omega^{\infty}_0(gl_1(R))$ are equivalent.  It follows that whether the localisation of $\eta$ vanishes depends entirely on whether $\eta$ vanishes as a map on $L_{K(n)}(R)$. Since this map is induced from the module structure (over $\bS$ and hence over $R$), it vanishes exactly when $\eta = 0 \in \pi_1 L_{K(n)}(R)$. But we have assumed $\eta = 0 \in \pi_1(R)$.
   \end{proof}

  \begin{proof}[Proof of Corollary \ref{cor:complex cobordism}]
      Theorem \ref{thm:main} says that there is some local system $\xi_L \to L$ and an equivalence $(L,\xi_L) \simeq K$ in $\scrF(X;\Psi)$. Theorem \ref{thm: gour main theorem} then says that the difference $[L] - [K] \in MU_*(X)$ lies in the image of multiplication by $\eta$. Now work at a fixed prime $p$, and with the localisation $MU_{(p)}$ of $MU$. This splits as a sum of shifts of copies of a spectrum $BP$, and a theorem of Hovey \cite{Hovey} says that the natural map
      \[
      BP \to \prod_n L_{K(n)}(BP)
      \]
      splits. But after any localisation, $\eta$ vanishes by Lemma \ref{lem:BK}. It follows that the difference of $[L]$ and $[K]$ vanishes in the $p$-localised theory $(MU_{(p)})_*(X)$,  for any $p$, and hence vanishes in $MU_*(X)$ by the arithmetic fracture square \cite{Sullivan:Galois}. 
  \end{proof}
  \subsection{Immersed Lagrangian cobordism}

Two closed Lagrangian submanifolds $L, K \subset X$ are \emph{Lagrangian cobordant} if there is a Lagrangian $P \subset X\times \bC$, which maps properly to $\bC$, and which outside a compact set fibres over real half-lines and agrees with $L \times (-\infty,-R] \sqcup K \times [R,\infty)$ for some real $R \gg 0$. One can allow all of $L,K$ and $P$ to be immersed, and one can also insist that all of them be oriented (with the orientation on $P$ compatible with those on the ends); they need not be connected, so up to Hamiltonian isotopy this incorporates the case of cobordisms with many ends, which fibre over finite collections of half-lines near $\pm\infty \subset \bR \subset \bC$. 

\begin{defn}
    The \emph{immersed oriented Lagrangian cobordism} group of $X$ is the abelian group generated by closed immersed oriented Lagrangian submanifolds, with relations $[L] = [K]$ if $L,K$ are immersed oriented Lagrangian cobordant. 
\end{defn}
\begin{rmk}
    It is natural to impose orientation requirements to rule out trivial cobordisms given by taking the product of $L$ and a $\subset$-shaped curve in $\bC$. Such cobordisms show that the unoriented immersed Lagrangian cobordism group is $2$-torsion.
\end{rmk}

The infinite complex Lagrangian Grassmannian $Sp/U$ admits a natural map to $U/O$ by viewing a complex Lagrangian subspace of $\bH^N$ as a real subspace. Since $Sp/U$ is simply connected, this canonically lifts to $\widetilde{U/O}$.

We take the tangential pair $\Psi = (\Theta, \Phi)$ with $\Theta = hofib(Sp/U \to \widetilde{U/O})$ and $\Phi = \{pt\}$. Note that a $\Theta$-structure on $L$ equips it with an orientation. 

\begin{cor} \label{cor:immersed general}
    Suppose $X$ is stably framed and $L,K$ admit $\Theta$-brane structures. If $L$ and $K$ are quasi-isomorphic in $\scrF(X;\bZ)$, then the classes of $L$ and $K$ agree in the oriented immersed Lagrangian cobordism group, after inverting the prime 2.

    If in addition $L$ and $K$ are homotopy spheres, then $L \sqcup L$ and $K \sqcup K$ are oriented immersed Lagrangian cobordant.
\end{cor}

\begin{proof} 
    By Theorems \ref{thm:main} and \ref{thm: gour main theorem} (cf. Remark \ref{rmk: recap rmk 1.4 ps3}), $L$ and $K$ have the same fundamental class in $\Omega^{\Psi}_*(X)$ up to inverting 2, so for some $l>0$ the disjoint unions $\sqcup_{i=1}^{2^l} [L\sqcup \ldots \sqcup L]$ and $\sqcup_{i=1}^{2^l} [K\sqcup \ldots \sqcup K]$ have the same class. By \cite{Audin, R-F} this bordism theory governs immersed oriented Lagrangian cobordism, so there is an immersed oriented Lagrangian $P \to X \times \bC$ which agrees near infinity with the immersions $\sqcup_j L_j  \times (-\infty,-R] \to L \times (-\infty,-R]$, and correspondingly for $K$ near $+\infty$, where the maps are the trivial covers indexed by $j\in \{1,\ldots,2^l\}$. We now use a (non-compactly-supported) Lagrangian isotopy to ``separate the ends'' so the projections lie over different real half-lines near $\pm\infty$. 

    The final statement uses the 2-torsion part of Remark \ref{rmk: recap rmk 1.4 ps3}.
\end{proof}

\begin{cor}
    If $Q$ is stably framed and $L \subset T^*Q$ is a nearby Lagrangian submanifold, then for some $l>0$ the union of $2^l$ copies of $L$ is immersed oriented Lagrangian cobordant to the union of $2^l$ copies of the zero-section.
\end{cor}

\subsection{Concrete consequences from normal invariant constraints}\label{sec: conc norm} 

    In this section we record some concrete consequences of Theorem \ref{thm: norm} for nearby Lagrangians in specific cotangent bundles, particularly of 8-manifolds. Most of them also follow from \cite{AAGCKarxiv}, though not the full statement of Example \ref{eq; ewrouhgruodhveudhv}, nor Corollary \ref{cor: wrigjeifvh}. 

    \begin{cor}
        $L \subseteq T^* S^8$ be a nearby Lagrangian. Then $L$ is diffeomorphic to $S^8$.
    \end{cor}
    This was shown earlier in \cite{AAGCKarxiv, PS}. There is a unique exotic 8-sphere \cite{Kervaire-Milnor}, and its normal invariant in $\pi_8 G/O$ is not in the image of $\eta$.
    \begin{ex}\label{eq; ewrouhgruodhveudhv}
        $T^n$ is stably homotopy equivalent to a wedge of spheres so $[T^n, G/O] \cong \oplus_k(\pi_k G/O)^{n \choose k}$, along with a similar direct sum decomposition of $[T^n, B(G/O)]$. $\eta$ preserves these decompositions since it is an infinite loop map. From Theorem \ref{thm: norm} it follows that for a nearby Lagrangian in $T^*T^n$, under this decomposition each component of its normal invariant must be in the image of $\eta$, and in particular be 2-torsion. 
        
        There is much odd torsion and non-torsion in $[T^n, G/O]$ that does arise from normal invariants of clases in $\cS^{\Diff}(T^n)$, as can be seen from \cite[(18.70) \& Theorem 17.6]{Luck-Macko}.
    \end{ex}

    \begin{cor}
        \begin{enumerate}
            \item Let $L \subseteq T^*(S^3 \times S^4)$. Then $L$ is diffeomorphic to $(S^3 \times S^4) \# \Sigma^7$, for $\Sigma^7$ some 7-dimensional homotopy sphere (of which there are 28, in particular finitely many).
            \item Let $L \subseteq T^*(S^4 \times S^4)$. Then $L$ is diffeomorphic to $S^4 \times S^4$.
        \end{enumerate}
    \end{cor}
    In both cases, the smooth structure set of the base of the cotangent bundle is infinite \cite{Crowley}. 
    \begin{proof}
        In the setting of (1), \cite[Theorem 1.1]{Crowley}  implies that if the normal invariant of $L$ is 2-torsion, it must lie in the image of the map $\omega^{bP}: \Theta_7 \to \cS^{\Diff}(S^3 \times S^4)$ which sends a homotopy 7-sphere $\Sigma$ to $(S^3 \times S^4)\#\Sigma$.

        In the setting of (2), $[S^4 \times S^4, G/O] \cong \pi_4 G/O \oplus \pi_4 G/O \oplus \pi_8 G/O \cong \bZ \oplus \bZ \oplus \bZ/2$. Similarly $[S^4 \times S^4, BG/O] \cong \pi_3 G/O \oplus \pi_3 G/O \oplus \pi_7 G/O$ which vanishes (cf. the proof of Corollary \ref{cor:hptwo}), so the image of $\eta$ is 0 (even though the group of normal invariants does have 2-torsion). Since $S^4 \times S^4$ is simply-connected and even-dimensional, the normal invariant map $\cS^{\Diff}(S^4 \times S^4) \to [S^4 \times S^4, G/O]$ is injective.
    \end{proof}
    \begin{cor}\label{cor:1251}
        Let $L \subseteq T^*\bC\bP^4$ be a nearby Lagrangian. Then $L$ is diffeomorphic to $\bC\bP^4$.
    \end{cor}
    \begin{rmk}
        Many exotic $\bC\bP^4$s can easily be shown not to embed in $T^*\bC\bP^4$ since they have the wrong Pontrjagin classes, such as those constructed in \cite[Theorem 2]{Hsiang}. Two other exotic $\bC\bP^4$s with the same $p_1$ as $\bC\bP^4$ were shown not to Lagrangian embed in $T^*\bC\bP^4$ in \cite[Theorem 7.11]{Torricelli}. However there are infinitely many homotopy-$\bC\bP^4$s that had not been ruled out: there are infinitely many topological manifolds which are homeomorphic to $\bC\bP^4$ and hence have the same Pontrjagin classes (such as $\bC\bP^4 \# \Sigma$, for $\Sigma$ the exotic 8-sphere) but are not diffeomorphic to $\bC\bP^4$. This can be seen from \cite[Theorem 18.15]{Luck-Macko}, combined with the exact sequence $[\bC\bP^4, G/O] \to [\bC\bP^4, G/\mathrm{Top}] \to [\bC\bP^4, B(\mathrm{Top}/O)]$ and finiteness of the final of these groups \cite[Theorem V.5.5(II)]{Kirby-Siebenmann}.
    \end{rmk}
    \begin{proof}[{Proof of Corollary \ref{cor:1251}}]
        Since $\bC\bP^4$ is simply-connected and even-dimensional, it suffices to show the normal invariant of $L$ vanishes; by Theorem \ref{thm: norm} it suffices to show $[\bC\bP^4, B(G/O)]$ vanishes. 
        
        $B(G/O)$ is an infinite loop space and so is the $0^{th}$ space in a connective spectrum $bgo$ with the same homotopy groups, and then $[\bC\bP^4, B(G/O)] = bgo^0(\bC\bP^4)$. There is an Atiyah-Hirzebruch spectral sequence: $H^*(\bC\bP^4, \pi_{-*}bgo) \Rightarrow bgo^*(\bC\bP^4)$. Therefore it suffices to show $\pi_i B(G/O) = 0$ for $i \in \{0,2,4,6,8\}$. The case $i=0$ is clear. From \cite{Adams}, we see that for $i \in \{1,3,5,7\}$, the $J$ homomorphism $\pi_i O \to \pi_i G$ is surjective, and for $i \in \{0,2,4,6\}$ it is injective. Therefore by the long exact sequence of a fibration $\pi_{i-1} O \to \pi_{i-1} G \to \pi_i B(G/O) \to \pi_{i-2} O \to \pi_{i-2} G$ we may conclude.
    \end{proof}
    
    \begin{cor}\label{cor:hptwo}
        Let $Q = \bH\bP^2$ or $\bH\bP^3$. Then any nearby Lagrangian $L \subseteq T^*Q$ is diffeomorphic to $Q$.
    \end{cor}
    There exist exotic $\bH\bP^2$s, such as $\bH\bP^2 \# \Sigma$ (where $\Sigma$ is the (unique) exotic 8-sphere).
    \begin{proof}
        Since in both cases $Q$ is simply connected and even-dimensional, it suffices to show the normal invariant of $L$ vanishes; by Theorem \ref{thm: norm} it suffices to show $[Q, B(G/O)] = 0$. By the same Atiyah-Hirzebruch spectral sequence as Corollary \ref{cor:1251}, this follows from the computation $\pi_i B(G/O)=0$ for $i \in \{4,8,12\}$, which follows from the fact $\pi_jO=0$ for $j \equiv 2\, (\mathrm{mod }\,4)$ and that $\pi_j O \to \pi_j G$ is surjective for $j \in \{3,7,11\}$.
    \end{proof}
    \begin{cor}\label{cor: wrigjeifvh}
        Let $Q = \bH\bP^4$ or $\bO\bP^2$. Any nearby Lagrangian $L$ in $T^*Q$ is diffeomorphic to $Q$.
    \end{cor}
    \begin{proof}
        In either case, the only contributions to the Atiyah-Hirzebruch spectral sequences for $[Q, B(G/O)]$ and $[Q, G/O]$ is $\pi_{15} G/O \cong \pi_{16} G/O \cong \bZ/2$. However, the map induced by $\eta$ between them is 0 \cite{IWX}.
    \end{proof}
    Concrete restrictions on nearby Lagrangians in cotangent bundles of other manifolds should similarly be extractable from \cite{Kasilingam-inertia, Crowley,Basu-Kasilingam, Kasilingam-CP}  and other existing computations.
    \begin{rmk}
    
        It is reasonable to ask how much smaller the image of $\eta$ is than the group of 2-torsion in $[Q, G/O]$. This of course depends on $Q$, but we record as an illustrative example the case $Q \simeq S^n$, taking a rolling average (using \cite[Figure 1]{IWX}): $50\%$ of the 2-torsion elements in $\pi_{10 \leq * \leq 19}\bS$ are in the image of $\eta$, around $54\%$ of those in $\pi_{20\leq * \leq 29}$, followed by $17\%$, $36\%$, $31\%$, $7\%$, $16\%$, ...\footnote{More precisely, this sequence is $n_k = \left(\sum_{i=10k}^{10k+9} \#(\eta\cdot\pi_{i-1}\bS)\right)/\left(\sum_{i=10k}^{10k+9} \#(\textrm{$2$-torsion in }\pi_{i}\bS)\right) \times 100\%$.}. For tori $T^n$, since $[T^n, G/O] \cong \oplus_k(\pi_k G/O)^{n \choose k}$ for $k\leq n$, the proportion will be much smaller. 
        
        In fact, Burklund's study of the asymptotic size of the stable homotopy groups of spheres \cite[Theorem 1.10]{Burklund:Big} shows that (assuming some conjectures related to the now-resolved \cite{Telescope} telescope conjecture), the cumulative average does indeed converge to 0.

    \end{rmk}

\section{Spectral Fukaya category recap}

In this section we recall some of the notation, construction and features of the spectral Fukaya category as constructed in \cite{PS3}. This r\'esum\'e is intended to fix notation and help keep this paper self-contained, but we encourage the reader to consult \emph{op.cit.} for additional context, details and examples.

\subsection{$\cI$-monoids and Thom spaces}

 Let $[n] := \{1, \ldots, n\}$. Let $\cI$ be the category with objects nonnegative integers $n$, and morphisms $n \to m$ given by (not necessarily order-preserving) injections $[n] \to [m]$.

    \begin{defn}
        A (convergent) \emph{$\cI$-space} is a functor $n \mapsto X(n)$ from $\cI$ to the category of spaces, with the property that the connectivity of $X(n) \to X(n')$ goes to infinity as $n\to \infty$. 
    \end{defn}

    An \emph{$\cI$-monoid} is a monoidal object in the category of $\cI$-spaces. Explicitly, this consists of an $\cI$-space $X$ equipped with maps $\mu_{mn}: X(m) \times X(n) \to X(m+n)$, which form an associative natural transformation $X(\cdot) \times X(\cdot) \to X(\cdot+\cdot)$, where we view both sides as functors $\cI^2$ to spaces. We also ask that this natural transformation be unital, i.e. that the basepoint acts as a two-sided identity element.

        An $\cI$-monoid $X$ is \emph{commutative} if the following diagram commutes for all $m,n$:
        \begin{equation}
            \xymatrix{
                X(m) \times X(n)
                \ar[r]_\mu
                \ar[d]_{\operatorname{swap}}
                &
                X(m+n)
                \ar[d]_{\sigma_{mn}}
                \\
                X(n) \times X(m)
                \ar[r]_\mu
                &
                X(n+m)
            }
        \end{equation}  
        where $\sigma_{mn}$ swaps the first $m$ factors of $[m+n]$ with the final $n$ factors.

We define $X_{h\cI}=\operatorname{hocolim}_{n \in \cI} X(n)$ to be the homotopy colimit of the $X(n)$.   An $\cI$-monoid $X$ is \emph{group-like} if $\pi_0 X_{h\cI}$ is a group. Group-like commutative $\cI$-monoids model infinite loop spaces (in the sense that there is a Quillen equivalence between appropriate model categories).

\begin{ex}
    Each of $n \mapsto BU(n)$, $n \mapsto U/O (n)$ and $n \mapsto \widetilde{U/O} (n)$ (where the universal covering is taken levelwise) define commutative $\cI$-monoids.
\end{ex}

\begin{lem}\label{lem: eta acts}
    Let $X$ be a commutative $\cI$-monoid, and $\alpha \in \pi_1\bS$. Then $\alpha$  induces a map $BX_{h\cI} \to X_{h\cI}$, well-defined up to homotopy.
\end{lem}

\begin{proof}
    See \cite[Remark 2.58]{PS3}.
\end{proof}

 Let $R$ be an associative ring spectrum. Following \cite{Schlichtkrull}, there is a commutative $\cI$-monoid $\Omega^\infty R$ sending $n$ to $\Omega^n R_n$, whose $\cI$-space structure maps are described in \cite[Section 2.8]{PS3}. 
        By abuse of notation, we often write $\Omega^\infty R$ for $\Omega^\infty R_{h\cI}$.

  If $R$ is commutative, there is a corresponding group-like commutative $\cI$-monoid $GL_1(R)$. Each $GL_1(R)(i)$ is defined to be the union of components of $\Omega^\infty R(i)$, which admit an element $c$ such that there is some $j$ and $c' \in \Omega^\infty R(j)$ such that $\mu(c,c') \in \Omega^\infty R(i+j)$ lies in the trivial path component. The product structure is inherited from $\Omega^\infty R$.  The assignment $R \mapsto GL_1(R)$ is functorial in $R$.
    \begin{ex}
        When $R=\bS$ is the sphere spectrum, $GL_1(\bS)_{h\cI} \simeq \hocolim_{n \to \infty} h\Aut(S^n)$ can be described as the homotopy autoequivalences of $S^n$ as $n \to \infty$, cf. \cite[Example 2.14]{PS3}.
    \end{ex}

       A Thom space is a pair $(X,E)$ where $E$ is a vector bundle on $X$, with the obvious notion of morphism.  A \emph{Thom $\cI$-space} $X$ consists of Thom spaces $X(n)$ for each $n \in \cI$, and for each morphism $f: n \to m$ in $\cI$, a map of Thom spaces $f(n): X(n) \oplus \bR^{m-n} \to X(m)$, where $m-n$ denotes the finite set $[m] \setminus f[n]$.

        We require that, for morphisms $n \to m \to l$, the following diagram commutes:
        \begin{equation}
            \xymatrix{
                X(n) \oplus \bR^{m-n} \oplus \bR^{l-m} 
                \ar[r]_-=
                \ar[d]
                &
                X(n) \oplus \bR^{l-n}
                \ar[d]
                \\
                X(m) \oplus \bR^{l-m}
                \ar[r]
                &
                X(l)
            }
        \end{equation}

\subsection{Tangential pairs and abstract discs\label{sec:tp and ad}}

\begin{defn}
     A \emph{(graded) commutative tangential pair} $\Psi$ consists of a pair $(\Theta \to \Phi)$ of group-like connected commutative $\cI$-monoids, equipped with maps to $BO$ and $BS_\pm U$ fitting into the following commutative diagram:
        \begin{equation}
            \xymatrix{
                \Theta 
                \ar[r]
                \ar[d]
                &
                \Phi
                \ar[d]
                \\
                BO
                \ar[r]
                &
                BS_{\pm} U
            }
        \end{equation}
        where $BS_{\pm} U$ is the $\cI$-space with level  $BS_\pm U(N)$ given by the homotopy fibre of the map $BU(N) \to K(\bZ,2)$ classifying $2c_1$. 

        $\Psi$ is \emph{orientable} if the induced map $\hofib(\Theta \to \Psi) \to \hofib(BO \to BS_\pm U)$ is zero on $\pi_2$.
\end{defn}
\begin{rmk}
    The orientability condition guarantees that objects of a spectral Fukaya category associated to $\Psi$ determine objects of a corresponding $\bZ$-linear Fukaya category.
\end{rmk}

We will later make use of the following (note that whilst $\eta$ appears here, its appearance is not obviously tied to its appearance in the formula for the bordism class of a brane with local system as established in \cite{PS3}):

 \begin{prop}\label{prop: uo eta}   
    
        Let $\Psi = (\Theta \to \Phi)$ be a commutative tangential pair and $F \to \widetilde{U/O}$ its homotopy fibre. Let $\eta^*:F_{h\cI} \to \Omega F_{h\cI}$ be the map induced by the Hopf element $\eta \in \pi_1 \bS$ as in Lemma \ref{lem: eta acts}. Then the following diagram of spaces commutes up to homotopy:
        \begin{equation}
            \xymatrix{
                F_{h\cI} 
                \ar[d]_{\eta_*}
                \ar[rr]
                &&
                \widetilde{U/O}_{h\cI}
                \ar[d]_{Re}
                \\
                \Omega F_{h\cI} 
                &
                (E_F)_{h\cI}
                \ar[l]_\simeq 
                \ar[r]
                &
                BO_{h\cI}
            }
        \end{equation}
    \end{prop}
    \begin{proof}
        The left hand side of the diagram is functorial in $F$, so it suffices to check the final case $\Psi = (BO \to BS_\pm U)$ (in which case $F = \widetilde{U/O}$); in this case, the result is standard.
    \end{proof}

One source of commutative Thom $\cI$-monoids is spaces of abstract discs, and index bundles over such spaces.   For each $n \in \cI$, there is a Thom space $(\bU(n), \bV(n))$, defined as the geometric realisation of a simplicial set, whose $0$-simplices comprise: 
        \begin{itemize}
            \item A complex vector bundle $E$ over $D^2\setminus\{1\}$ of rank $n$. 
            
            \item A real subbundle over the boundary $F \subseteq E|_{\partial D}$; $E$ is called the \emph{complex part} and $F$ the \emph{real part}.

            We call $n$ the \emph{rank}.

            \item A classifying map $D^2\setminus \{1\} \to BS_\pm U(n)$ for $E$.

            \item A classifying map $\partial D^2 \setminus 1 \to BO(n)$ for $F$, compatible with the classifying map for $E$. 

            We require that the classifying maps agree with (the $n$-fold direct sum of) a certain standard model puncture datum over the strip-like end.

            \item A 1-form $Y \in \Omega^{0,1}(E)$, a metric (standard in the strip-like ends), and a Hermitian connection  $\nabla$ on $E$.
            
           We define a \emph{Cauchy-Riemann operator}:

            \begin{equation}
                D^{CR} = \overline\partial^\nabla + Y: \Gamma_{W^{2,\kappa}}\left(D^2\setminus \{1\}, E, F\right) \to \Omega^{0,1}_{W^{2,\kappa-1}}\left(D^2\setminus \{1\}, E\right)
            \end{equation}
            where the subscripts $\cdot_{W^{2,\cdot}}$ indicate that we are passing to suitable Sobolev completions. This operator is  Fredholm, and in particular its cokernel is finite-dimensional.

            \item A linear map $f: \bR^n \to \Omega^{0,1}_{W^{2,\kappa-1}}(D^2\setminus \{1\}, E)$, such that $D^{CR} + f$ is surjective. 
        \end{itemize}

        (One could more generally allow the stabilising vector space and the bundle pair to have different ranks; we have passed to a diagonal $\cI$-space inside a natural $\cI^2$-space.)
        
There is a simplicial set in which a $k$-simplex consists of a smoothly-varying family $(E_t, F_t, Y_t, \ldots)_{t \in \Delta^k}$ of all of the above data, parametrised by the standard $k$-simplex $\Delta^k$.

        The vector bundle $\bV(n) \to \bU(n)$ is defined so that for a simplex $(\sigma_t)_{t \in \Delta^k} = (E_t, F_t, \ldots)_{t \in \Delta^k}$ as above, the restriction of $\bV(n)$ to the realisation of $(\sigma_t)_t$ has fibre over $t \in \Delta^k$ given by the kernel of the associated Cauchy-Riemann operator  $\ker(D^{CR}_t)$. The particular choice of model puncture data that we make at the strip-like end ensures that 
the vector bundle $\bV(n)$ has rank $n$. 

\begin{lem}
    There is a weak equivalence of commutative $\cI$-monoids:
        \begin{equation}
            \bU \to \hofib(\Omega BO \to \Omega BS_\pm U)
        \end{equation}
\end{lem}

 \begin{defn}\label{qi def: E Psi}
        For $\Psi = (\Theta, \Phi)$ a commutative tangential pair, we define $\bU^\Psi$ to be the homotopy pullback of 
        \begin{equation}\label{qi eq: hopullb}
            \xymatrix{
                \bU^\Psi 
                \ar[r]
                \ar[d]
                &
                \bU
                \ar[d]
                \\
                \hofib(\Omega \Theta \to \Omega \Phi) 
                \ar[r]
                &
                \hofib(\Omega BO \to \Omega BS_\pm U)
            }
        \end{equation}
        and $\bV^\Psi$ to be pullback of $\bV$ (defined levelwise). The pair $(\bU^\Psi(\cdot), \bV^\Psi)$ forms a commutative Thom $\cI$-monoid, which we write as $E_\Psi$. We will write $R_{\Psi}$ for the associated symmetric spectrum, and $R_{\Psi}^*$ for the associated generalised cohomology /  bordism theory.
        \end{defn}

        The spaces of abstract discs $\bU^{\Psi}$ are built from discs with a single outgoing boundary puncture. There are more general spaces of discs $\bU^{\Psi}_{ij}$ with $i\geq 0$ and $j\geq 0$ incoming and outgoing boundary punctures, discussed at length in \cite{PS3}; it is often helpful to fix a particular choice $\bE$ of puncture data (which in  particular prescribes Maslov indices at punctures). For current purposes, it suffices to note that these more general spaces can be equipped with index bundles by gluing on fixed caps and co-caps to reduce to the case $i=0,j=1$ discussed above. There is an addition-type operation arising from a 
        map of Thom $\cI$-spaces (strictly of the underlying $\cI$-simplicial sets) 
        \[
        \bU_{\bE}^\Psi(n) \times_{\cR\cS_{ij}} \bU^\Psi_{\bE'}(n')  \to \bU^\Psi_{\bE\oplus\bE'}(n + n') 
        \]
                which is commutative, and if we consider the same construction applied to iterated fibre products, it is associative.

        \begin{rmk}
             A variation on these spaces of abstract discs, allowing boundary punctures with non-transverse totally real boundary conditions, will appear in Section \ref{sec:morse-floer} below.
        \end{rmk}

\subsection{$E$-oriented flow categories\label{sec:E-orientations recap}}

Fix a monoid $E$ in Thom $\cI$-spaces. Thus we have base spaces $\Base(E(n))$ and bundles $E(n) \to \Base(E(n))$, with suitably compatible addition maps. 

To define an $E$-orientation on a flow category $\cF$, we will use two auxiliary pieces of data:
\begin{enumerate}
    \item \emph{stabilisation data} $\{\nu_{xx'}\}$ for $x,x' \in \cF$; these are non-negative integers with injections $[\nu_{xx'}] \sqcup [\nu_{x'x''}] \to [\nu_{xx''}]$ which are associative for quadruples;
    \item \emph{index data} $\{d_{xx'}\}$ for $x,x' \in \cF$; these are finite sets $d_{xx'}$ with injections $d_{xx'} \sqcup d_{x'x''} \to d_{xx''}$ which again make the corresponding diagram for a quadruple of elements commute (but these injections need not preserve order);
\end{enumerate}
cf.  \cite[Section 4.1]{PS3}.

\begin{defn}
    An $E$-oriented flow category  comprises a flow category $\cF$ with stabilisation data $\{\nu_{xx'}\}$ and index data $\{d_{xx'}\}$,  together with maps of Thom spaces
   \begin{equation} 
   \rho: (\cF_{xx'},I^\cF_{xx'})(\nu_{xx'}) \to E(\nu_{xx'}+d_{xx'})
   \end{equation}
    for any $x,x' \in \cF$, which are compatible with and associative under composition. Note that $\rho$ incorporates a map of base spaces $\rho_{xx'}: \cF_{xx'} \to \Base(E(\nu_{xx'}+ d_{xx'}))$ and a covering isomorphism of bundles $\st_{xx'}: I^{\cF}_{xx'} \oplus \bR^{\nu_{xx'}} \to E(\nu_{xx'} + d_{xx'})$.
\end{defn}

One can similarly define $E$-orientations on morphisms, bordisms, associators, and left and right modules over flow categories, and there is a category $\Flow^E$ of $E$-oriented flow categories. An $E$-oriented right module of degree $i$ is exactly a morphism $\ast[i] \to \cF$ of $E$-oriented flow categories.  The bordism group $\Omega_i^E(\cF)$ is the morphism group $[\ast[i],\cF]$ in the category $\Flow^E$.

We defer to \cite{PS3} for most of this, but a couple of features will be particularly important in this paper.

 For a $k$-dimensional manifold with corners $X$, \cite[Lemma 6.3 \& Definition 6.4]{AB2} construct a $(k+1)$-dimensional manifold with corners $\bD(X)$, the `conic degeneration of $X$'. If in a local chart $X$ is modelled on $[0,\infty)^k$, then $\bD X$ maps to $[0,\infty)^k$ with fibres a union of $i+1$-intervals over the locus where $i$ co-ordinates vanish. 
    
    \begin{lem}\label{lem:flow unit}
        For a flow category $\cF$:
        \begin{enumerate}
            \item there is a morphism $\bD\cF: \cF \to \cF$ with spaces
    \[
    (\bD\cF)_{xy} = \begin{cases} \bD(\cF_{xy}) & x\neq y \\ \{pt\} &  x=y\end{cases}
    \]
    \item  $\bD\cF$ defines the unit for $\cF$ in a unital structure on the category $\Flow$;
    \item if $\cF$ is $E$-oriented, then $\bD\cF$ is canonically $E$-oriented, and becomes the unit in $\Flow^E$. 
        \end{enumerate}
    \end{lem}

    \begin{proof}
        See \cite[Section 4.5]{PS3}.
    \end{proof}

\subsection{Flow categories over a target\label{sec:flow over target}}

Let $L$ be a finite cell complex; we will primarily be interested in the case where $L$ is a manifold or manifold with boundary, or possibly the quotient thereof by a contractible subset. We write $\cP L$ for the \emph{Moore path space} of $L$, comprising pairs $(r,\alpha)$ where $r\in \bR_{\geq 0}$ and $\alpha: [0,r] \to L$.  This admits a natural concatenation operation for paths whose endpoints agree. This gives a category that we also denote by $\cP L$, with objects points in $L$ and morphisms $x$ to $y$ given by $\cP_{xy}L$, Moore paths from $x$ to $y$.

\begin{defn}
A \emph{flow category over $L$} comprises a flow category $\cF$, along with a topological functor $\cF \to \cP L$.  
\end{defn}

  Explicitly, this consists of a point (abusively denoted) $x$ in $L$ for each $x \in \cF$ and maps of spaces $\cF_{xy} \to \cP_{xy}L$, compatible with concatenation.  We can define morphisms and bordisms of flow categories over $L$ in the obvious way, and an inspection of the definition of composition of morphisms from \cite{PS,PS2} shows that a composition of morphisms over $L$ inherits the structure of a morphism over $L$.

        An \emph{$E$-oriented flow category over $L$} consists of an $E$-oriented flow category $\cF$ which lives over $L$.  Explicitly, this means that there are maps  $\cF_{xy} \to \Base(E)\times \cP_{xy} L$ covered by maps of vector bundles $(\cF_{xy},I^{\cF}_{xy})(\nu_{xy}) \to (E\times \cP_{xy} L)(\nu_{xy}+d_{xy})$ where $(E\times \cP_{xy} L)(n) := E(n) \times \cP_{xy} L$ and the vector bundle on the target is pulled back from $E(n)$.

    There is a category $\Flow^{E}_{/L}$ of $E$-oriented flow categories over $L$. Morphisms in this category are given by bordism-over-$L$ classes of $E$-oriented morphisms over $L$.

The construction of the unit morphism $\bD\cF$ in \cite[Section 4.5]{PS3} applies unchanged when $\cF$ lives over $L$, and makes $\Flow^{E}_{/L}$ a unital category.

    \begin{ex}\label{ex: ev hom iso flow}
        Suppose $\cF',\cF'' \in\Flow_{/L}^E$ have the same underlying $E$-oriented flow category $\cF$, but are equipped with different evaluation maps $Z'_{xy},Z''_{xy}: \cF_{xy} \to \cP_{xy}L$.

        If $\{Z^t_{xy}\}^{t \in [0,1]}_{xy}$ is a collection of homotopies between $\{Z'_{xy}\}$ and $\{Z''_{xy}\}$, the identity morphism $\cF \to \cF$ in $\Flow^E$ equipped with the evaluation maps $\{Z^t_{xy}\}$ defines an isomorphism between $\cF'$ and $\cF''$ in $\Flow_{/L}^E$.
    \end{ex}

    \begin{defn}
        To define $*[i]$ as an $E$-oriented flow cat over $L$, the only data we need to pick is a base-point $\ell_0 \in L$.
    \end{defn}

Let $\cF$ be a flow category over $L$, with evaluation maps
\[
\Gamma_{xy}: \cF_{xy} \to \cP_{xy}L
\]
We fix a `spider'
\begin{tikzpicture}[scale=0.05]
    \draw [fill=black] (0,0) circle(.5cm);
    \draw [fill=black] (0,1) circle(.8cm);
    \draw[thick] (0,.1) parabola (2,3);
    \draw[thick] (0,.05) parabola (2,1);
    \draw[thick] (0,0) parabola (2,-1);
    \draw[thick] (0,0) parabola (2,-3);
    \draw[thick] (0,-.3) parabola (.3,-.8);
    
    \draw[thick] (0,.1) parabola (-2,3);
    \draw[thick] (0,.05) parabola (-2,1);
    \draw[thick] (0,0) parabola (-2,-1);
    \draw[thick] (0,0) parabola (-2,-3);
    \draw[thick] (0,-.3) parabola (-.3,-.8);
\end{tikzpicture}
$:= Y\subseteq L$ as in \cite[{Remark 4.35}]{PS3}, i.e. a contractible set of arcs which connect the points $x\in \mathrm{Ob}\,\cF \subset L$ to a basepoint $\ell \in L$. The collapse map $L \to L/Y$ is a homotopy equivalence; we will write $\theta: L/Y \to L$ for a homotopy inverse, and also 
\[
\theta: \Omega_{[\ell]}(L/Y) \to \Omega_{\ell}L
\]
for the induced homotopy inverse on based loops.

The choice of $(Y,\theta)$ means that the evaluation of any one-dimensional moduli space $\cF_{xy} \to L \to L/Y \dashrightarrow L$ yields an element of $\Omega_{[\ell]}L$; both morphisms and bordisms only change this up to homotopy. 
    One can then define a functor to chain complexes of local systems over $L$, as follows. We recall that a flow category $\cF$  admits a truncation $\tau_{\leq 0}\cF$, in which we keep track of only moduli spaces of dimension $\leq 1$, and both $E$-orientations and evaluation maps to $L$ are inherited by such a truncation.

Suppose that $E$ is oriented in the sense of \cite[Definition 4.27]{PS3}, i.e. that we have compatible trivialisations of the determinant lines of the bundles $E(n)$ (this enables us to work over $\bZ$ below). In this case, $\pi_0\Thom(E) = \bZ$.

    \begin{defn}
        Let $M$ be a $\bZ[\pi_1 L]$-module, and $\cF \in \tau_{\leq 0}\Flow^E_{/L}$. The \emph{Morse chains of $L$ with coefficients in $M$} is the chain complex $CM_*(\cF; M)$ which as a graded abelian group is given by

        \begin{equation}
            CM_*(\cF; M) := \bigoplus\limits_{x \in \cF} M \cdot x
        \end{equation}
        with $x \in \cF$ a formal generator in degree $|x|$, and with differential 
        \begin{equation}
            \partial m \cdot x := \sum\limits_{|x|-|y|=1} \sum\limits_{p \in \cF_{xy}} [p]m \cdot y.
        \end{equation}
        Here $[p] \in \bZ[\pi_1 L]$ is the homotopy class $\pm [\Gamma_{xy}(p)] \in \pi_1 L$, with sign given by the image of the oriented bordism class of $\{p\} \in \cF_{xy}$ determined by the orientation of $E$.

    \end{defn}
    It is easy to check:
    \begin{lem}
        $\partial^2 = 0$ so $CM_*(\cF; A)$ is a chain complex.
    \end{lem}
Note that this construction is functorial under morphisms. If $E$ is not oriented, both statements hold over $\bZ/2$ instead of $\bZ$.

Let $hoCh(\bZ[\pi_1 L])$ denote the homotopy category of chain complexes of $\bZ[\pi_1 L]$-modules. 
    
    \begin{lem}\label{lem:morse chains over group rings}
        If $E$ is oriented and $\pi_0\Thom(E) \cong \bZ$, the functor 
        \begin{equation} 
            CM_*(\cdot; \bZ[\pi_1 L]): \tau_{\leq 0} {\Flow}^{E} _{/L} \to hoCh(\bZ[\pi_1 L])
        \end{equation}
        is an isomorphism onto the full subcategory of finitely generated free complexes.  The corresponding statement also holds when $\pi_0\Thom(E) \cong \bZ/2$.
    \end{lem}
   
    \begin{proof}
        Identical to \cite[Lemma 5.41]{PS} and \cite[Lemma 5.25]{PS2}.
    \end{proof}

  \begin{lem}\label{lem: 0-modification}
        Let $\cA, \cB \in \Flow^{E}_{/L}$. Assume that $\tau_{\leq 0} \cA$ and $\tau_{\leq 0}\cB$ are isomorphic in $\tau_{\leq 0}\Flow^{E}_{/L}$. 

        Then there is some $\cB' \in \Flow^{E}_{/L}$ which is isomorphic to $\cA$ in $\Flow^{E}_{/L}$, with $\tau_{\leq 0} \cB' = \tau_{\leq 0}\cB$.
    \end{lem}
    \begin{proof}
        Same as \cite[Proposition 4.30]{PS3}, incorporating the evaluation map to $L$  in the natural way. In particular, the closed manifolds $P = \{P_{yy'}\}$ introduced in the proof of \emph{op.cit.}~in constructing the category $\cC(P)$ should now be equipped with  classifying maps to $E \times \cP L$ rather than to $E$.  
    \end{proof}

      \begin{prop}\label{prop: cone prop over L}
        Let $\cW, \cW': \cA \to \cB$ be $E$-oriented morphisms over $L$ of $E$-oriented flow categories over $L$. 
        \begin{enumerate}
            \item If $\cW$ and $\cW'$ are bordant (via a $E$-oriented bordism $\cR$ which itself lives over $L$), then $\Cone(\cW) \cong \Cone(\cW') \in \Flow^{E}_{/L}$.
            \item If $\cV': \cA' \to \cA$ and $\cV'': \cB \to \cB'$ are isomorphisms, then $\Cone(\cW) \cong \Cone(\cV'' \circ \cW \circ \cV')$ are isomorphic in $\Flow^{E}_{/L}$.
        \end{enumerate}
    \end{prop}
           \begin{proof}
        Same as \cite[Proposition 4.26]{PS3}, incorporating the evaluation to path spaces.
    \end{proof}

\begin{rmk}\label{rmk:loop space coefficients}
    We briefly comment on the connection between flow categories over a target $L$, and the development of `Floer theory with universal local coefficients' (in the based loop space) as in \cite{BDHO,BDHO2}.

 Let $\cF$ be a flow category over $L$. There are chains $\sigma_{xy} \in C_{|x|-|y|-1}(\cF_{xy})$ which represent the relative fundamental class, and which satisfy 
    \begin{equation} \label{eqn:fundamental chains}
    \partial \sigma_{xy} = \sum_z (-1)^{|x|-|z|} \sigma_{xz} \times \sigma_{zy}
    \end{equation}
    For the flow category associated to a Morse function on a compact manifold,  this is established over $\bZ/2$ by \cite{Barraud-Cornea} and over $\bZ$ by \cite{BDHO}. The proof constructs the required fundamental chains inductively in $|x|-|y|-1$, using only that the boundary of $\cF_{xy}$ is covered by the facets $\cF_{xz} \times \cF_{zy}$ (compatibly with evaluation to $L$); it therefore generalises immediately to any flow category over $L$.  Let $\frak{m}_{xy} \in C_*(\Omega_{\ell}L)$ be given by the pushforward of $\sigma_{xy}$ under the map
\[
\cF_{xy} \to \cP_{xy}L \to \Omega_{[\ell]}(L/Y) \stackrel{\theta}{\longrightarrow} \Omega_{\ell}L.
\]
The Morse complex $CM_*(\cF;\Omega L)$ of $\cF$ over $L$ is the complex
    \begin{equation} \label{eqn:morse complex over L}
    CM_*(\cF) \otimes C_*(\Omega_{\ell} L), \quad \partial (x \otimes \alpha) = x \otimes \partial \alpha + \sum_y (-1)^{|\alpha|} y \otimes \frak{m}_{xy}\cdot \alpha
    \end{equation}
     A morphism $\cB: \cF \to \cG$ of flow categories over $L$ and which itself lives over $L$ defines a chain map
    \[
    \cB_*: CM_*(\cF;\Omega L) \to CM_*(\cG;\Omega L).
    \]
    If $\cB$ and $\cB'$ are bordant over $L$, the corresponding chain maps are chain homotopic. Thus, there is a functor
\[
{\Flow}_ {/L} \longrightarrow C_*(\Omega L)\mhyphen\textrm{mod}
\]
taking a flow category $\cF$ over $L$ to $CM_*(\cF;\Omega L)$.
\end{rmk}

\subsection{Twisting flow categories by local systems\label{sec:twist me}}

A \emph{monodromy local system} (MLS) over $L$, also called a `spectral local system', is a map $L \to BGL_1(R)$, where $R$ is a commutative ring spectrum. We will be interested in the case where $R = R_{\Psi}$ is associated to a commutative tangential pair as in Definition \ref{qi def: E Psi}.

Suppose the flow category $\cF$ lives over $L$. Fix a spider $Y \subset L$ as in Section \ref{sec:flow over target}, and to simplify notation identify $L$ and $L/Y$. There are then  maps $\cF_{xy} \to \Omega L$.  Combined with $\xi: L \to BGL_1(R)$, we obtain maps $\cF_{xy} \to GL_1(R)$. \cite[Section 7]{PS3} explains how to refine these to obtain maps
\[
\Xi_{\cM}: \cM \to GL_1(n_{xy})
\]
for every moduli space $\cM$ in the flow category, or indeed related moduli space (underlying a morphism, bilinear map, associator, unit \ldots) 
which are compatible with composition, and where the values $n_{xy}$ are determined by the chosen stabilisation and index data for the flow category.  By definition, $GL_1(n)$ is a union of components of $\Omega^n \Thom(\bV^\Psi(n) \to \bU^\Psi(n))$. Hence the maps $\Xi_\cM$ give maps, for each  moduli space $\cM$:
    \begin{equation}
        \hat\Xi_\cM: \cM \times S^{(\sum_b n_b)} \to  \Thom\left(\bigoplus_b\bV^\Psi(n_b) \to \prod_b\bU^\Psi(n_b)\right)
    \end{equation}
    where, in the geometric case, $b$ runs over the set of boundary components of the domains of maps in the particular moduli space $\cM$.

    One can assume that all these maps are transverse to the zero-sections of the image Thom spaces, and those preimages $\hat\cM = \hat\Xi_{\cM}^{-1}(0)$ are then smooth manifolds with faces (of the same dimension as $\cM$, since the bundle $\bV(n)$ has rank $n$). The compatibility of the $\Xi_{\cM}$ with flow category compositions and other algebraic structures mean that the same properties are inherited by the $\hat{\cM}$, so from $\cF$ over $L$ and $\xi: L \to BGL_1(R)$ we obtain a twisted flow category $\cF^{\xi}$.  Starting from the fact that there are canonical maps $\cF^{\xi}_{xy} \to \cF_{xy}$ from the construction, \cite{PS3} shows that:
    
    \begin{lem}\label{lem:twists are oriented}
        When $\cF$ is $E$-oriented, the twist $\cF^{\xi}$ inherits a canonical $E$-orientation.
    \end{lem}

    \begin{rmk}
        In \cite{PS3} the fact that this twisting can be done coherently is crucial. In this paper, that coherence plays a much more minor role.
    \end{rmk}

\subsection{The Fukaya category and $\cO\cC$-bordism\label{sec:Fuk and OC}}

We fix a commutative $\cI$-monoidal tangential structure $\Psi = (\Theta \to \Phi)$.  
We pick a finite set of Lagrangians $\cL$, and for each one pick some finite collection of $\Psi$-monodromy local systems $\{\xi_L \, | \, L \in \cL\}$. In \cite{PS3} we defined a spectral (Donaldson-)Fukaya category $\scrF^{loc}(X; \Psi)$ with objects $(L,\xi_L)$, where $\xi_L \to L$ is one of this finite set of distinguished monodromy local systems. We will typically abbreviate the pair ($L,\xi_L)$ to just $\xi_L$, and we call such pairs \emph{$\Psi$-branes}.  

The appearance of the Lagrangian data $\cL$, in particular the restriction to finitely many Lagrangians and local systems, appears because of the use of inductive arguments in dimension in setting up coherent $E_{\Psi}$-orientations and twistings of moduli spaces in \cite{PS3}.

We will not give a general review of Floer flow categories, which are discussed extensively in \cite{PS,PS2}. The new aspect of the story begins with the incorporation of spectral local systems.  For any pair $\xi_L, \xi_K$ with $L,K \in \cL$ we have a twist $\cM^{\xi_L,\xi_K}$ of the Floer flow category $\cM^{LK}$, as sketched in Section \ref{sec:twist me},  and $E_\Psi$-orientations on all the twisted moduli spaces, as above. Thus the categories $
\cM^{\xi_L,\xi_K}$ 
are $E_\Psi$-oriented flow categories, see \cite[Section 7.4]{PS3} for the detailed constructions. 

\begin{defn}
    The morphism group $\scrF^{loc}_i((L,\xi_L), (K,\xi_K); \Psi)$ is the group of bordism classes of $E_{\Psi}$-oriented  right flow modules $\Omega_i^{E_\Psi}(\cM^{\xi_L,\xi_K})$.
\end{defn}

Note that this is the morphism group $[\ast[i],\cM^{\xi_L\xi_K}]$ in $\Flow^{E_{\Psi}}$; we are not currently working over a target Lagrangian, so are not working in $\Flow^E_{/L}$ (or indeed $\Flow^E_{/K}$, since moduli spaces of strips evaluate to both Lagrangians). 

All required algebraic properties are established in \cite{PS3}, so the objects $\xi_L$ and morphism groups above do define a category $\scrF^{\loc}(X;\Psi)$.

To relate this category to the `usual' Fukaya category, we recall first that an exact symplectic manifold $X$ gives rise to a collection of integral Fukaya categories $\scrF(X,b;\bZ)$ indexed by a choice of $b \in H^2(X;\bZ/2)$, where $b$ determines a particular coherent orientation scheme.

In \cite[Lemma 7.20]{PS3}, it is shown that if $\Psi=(\Theta, \Phi)$ is an oriented commutative tangential pair and $X$ is $\Phi$-oriented, then there is a canonical class $b^{can} \in H^2(X; \bZ/2)$, and any $\Theta$-oriented Lagrangian $L \subseteq X$ inherits a relative Spin structure relative to $b^{can}$.

\begin{prop}\label{prop:truncate to ordinary}
    Let $\Psi$ be an oriented graded tangential pair, and $X$ a $\Phi$-oriented Liouville domain. Let $b^{can} \in H^2(X; \bZ/2)$ be as defined above. Fix a set of $\Psi$-oriented Lagrangian data $\cL$. 
    
    Then there is a fully faithful functor
    \[
    \tau_{\leq 0} \scrF^{loc,\cL}(X;\Psi) \to \scrF^{loc,\cL}(X,\phi^*w_2;\bZ).
    \]
\end{prop}

\begin{proof}
    See \cite[Proposition 7.25]{PS3}. (The superscript indicates that also in the target we are only seeing the subcategory of the Fukaya category comprising the finite set of chosen branes $\cL$.)
\end{proof}

\begin{rmk}
    In \cite{PS,PS2} we developed an obstruction theory for lifting quasi-isomorphisms from $\tau_{\leq i}\Flow$ to $\tau_{\leq (i+1)}\Flow$. However, the obstructions are often non-vanishing. Despite Proposition \ref{prop:truncate to ordinary}, the crucial `lifting quasi-isomorphisms from integral to spectral categories' theorems in this paper will work fundamentally differently, over a handle decomposition of the target Lagrangian $L$.
\end{rmk}

 Let $F = \hofib(\Theta \to \Phi)$ be the homotopy fibre and $E_\Psi$ the commutative Thom $\cI$-monoid constructed from spaces of abstract discs in Section \ref{sec:tp and ad}.  As with any Thom spectrum, there is an associated bordism theory $R^{\Psi}_*$, but the `fundamental class' of a Lagrangian brane does not  live most naturally in $R^{\Psi}_*(X)$, but in a cousin called `open-closed bordism'. 
 
 \begin{defn}\label{defn: oc-bordism}
     The groups $\Omega^{E_\Psi; \cO\cC}_i(X, \phi)$ are defined to be the set of cobordism classes of tuples $(M, f, k, \rho, h)$ where:
        \begin{itemize}
            \item $f: M \to X$ is a closed $i$-manifold over $X$.
            \item $k \in \cI$.
            \item $\rho: (M, TM)(k+N) \to (\Theta(N), \Gamma) \times E(i+k)$ is a map of Thom spaces, where $\Gamma \to \Theta(N)$ is the pullback of the universal bundle over $BO(N)$.
            \item $h$ is a homotopy between the two maps $M \to \Phi(N)$ which factor through $f$ and $\Base(\rho)$ respectively.
        \end{itemize}
 \end{defn}

\begin{rmk}
    A somewhat more intrinsic description of the $\cO\cC$-bordism group is given at the start of \cite[Section 5.3]{PS3}, but this concrete description will be most useful in this paper, in particular in the discussion of normal invariants in Section \ref{sec:normal}.
\end{rmk}

The definition of the groups $\Omega_*^{E_{\Psi};\cO\cC}(\bullet)$ makes sense for any pair $(Y,\phi)$ where $Y$ is a cell complex with a complex vector bundle $\phi: Y \to \Phi \to BU$ equipped with a  $\Phi$-lift $\phi$. 

\begin{lem}\label{lem: oreigrpieoghdirpgsgr}
    Any $\Psi$-oriented Lagrangian $L \subset X$ has a `fundamental class' $[L] \in  \Omega_d^{E_{\Psi};\cO\cC}(X)$ in $\cO\cC$-bordism.
\end{lem}

\begin{proof}
    We define $\Omega^{E_\Psi,\cO\cC}_*(L,\phi|_L)$ using the restriction of $\phi$ to $L$, here $L$ defines a canonical element (taking the trivial $E$-orientation), and then we push forward under inclusion.
\end{proof}

\begin{lem}\label{lem:module}
    The groups $\Omega_*^{E_{\Psi};\cO\cC}(Y,\phi)$ are modules over $R_{\Psi}^{-*}(Y)$.
\end{lem}

\begin{proof}
    See \cite{PS3}.
\end{proof}

The $\cO\cC$-fundamental class of $L$ is defined by an evaluation map from moduli spaces of perturbed holomorphic discs with boundary on $L$, themselves arising from the moduli spaces underlying the unit and counit. The `twisting' construction for flow categories in the presence of a spectral local system $\xi: L \to BGL_1(R)$ can also be applied to the moduli spaces of holomorphic half-planes defining the unit and co-unit of $L$, so give rise to a twisted version of its $\cO\cC$-fundamental class, which is by definition the class $[L,\xi] \in \Omega^{E_\Psi,\cO\cC}_*(X,\phi)$ associated to the pair $(L,\xi)$.

There is a degree $0$ $R_{\Psi}^*$-cohomology class $[\eta \xi] \in R_{\Psi}^0(L)$ given by the following composition:
        \begin{equation}
            L \xrightarrow{\xi} B(GL_1^\Psi)_{h\cI} \xrightarrow{\eta_*} (GL_1^\Psi)_{h\cI} \subseteq \Omega^\infty R 
        \end{equation}
        Here $\eta_*$ is the action of the Hopf map $S^3 \to S^2$ on the infinite loop space $GL_1^\Psi$ as arising from Lemma \ref{lem: eta acts}, cf. also \cite[Section 2.6]{PS3}

Then the main result of \cite{PS3} asserts that the twisted fundamental class is given by the module action of Lemma \ref{lem:module}, using this particular class:

\begin{thm}\label{thm: gour main theorem}
    Let $(L, \xi)$ be an object of $\scrF^{loc}(X; \Psi)$, i.e. $L$ is a closed exact $\Psi$-oriented Lagrangian in $X$, and $\xi$ a $\Psi$-local system on $L$.  Then:
        \begin{equation}
            [L,\xi] = [L] \cap [\eta\xi] \in \Omega_d^{E_\Psi,\cO\cC}(L, \phi_L).
        \end{equation}
\end{thm}
\begin{rmk}\label{rmk: recap rmk 1.4 ps3}
    \cite[Corollary 1.4]{PS3} implies the difference between the usual bordism class $[L]-[L]\cap[\eta\xi]$ is always trivial after inverting the prime 2.

    If $L$ is a homotopy sphere, this difference is in fact 2-torsion.
\end{rmk}
\begin{rmk}\label{recap: OC bordism maps here}
    \cite[Example 5.22]{PS3} produces maps from $\cO\cC$-bordism groups to more familiar theories in some cases:
    \begin{enumerate}
        \item If $\Psi=(\pt,BS_\pm U)$, there is a map $\Omega^{E_\Psi;\cO\cC}_i(X, \phi) \to MU_i(X)$.
        \item If $\Psi=(BO \times F, BO)$ is as in Example \ref{qi ex: tang str pol}, the $\Psi$-structure on $X$ corresponds to a polarisation $TX \cong V \otimes \bC$. Then there is a map to bordism relative to $V$: $\Omega^{E_\Psi;\cO\cC}_i(X,\phi) \to \Omega^{E_\Psi}_i(X, V)$.
    \end{enumerate}
    In both cases, these are $R^*(X)$-module maps. 

    Our applications combine Theorem \ref{thm: gour main theorem} with this to land in a more familiar bordism theory.
\end{rmk}

Together with the lifting theorem for quasi-isomorphisms established below, this is the key result from \cite{PS3} that underpins the applications in Section \ref{sec: applications}.

\section{Morse models for Floer flow categories\label{sec:morse-floer}}

Let $L$ be a closed exact Lagrangian in a Liouville domain $X$. 

\subsection{Morse flow categories}

It will be helpful to have a Morse model for the flow category associated to the pair of branes given by $L$ and a (sufficiently small) transverse Hamiltonian push-off of $L$.

         Assume that $L$ is a smooth manifold of real dimension $d$ and let $f: L \to \bR$ be Morse. To define the Morse flow category $\cM(f)$, we set the grading of a critical point to be given by 
         \begin{equation} \label{eqn:damned conventions}
         |x| = \mathrm{ind}(x).
         \end{equation} 
         We work with \emph{negative} gradient flow, so the function $f$ decreases along flow-lines, and moduli spaces $\cM(f)_{xy}$ are spaces of negative flow-lines from $x$ to $y$.

         Assuming the Riemannian metric is chosen to be standard at the critical points, 
         $\cM(f)$ admits the structure of a smooth flow category \cite{Wehrheim, CJS}.  By using Moore paths it comes with a functor $\cM(f) \to \cP L$, so it admits the structure of a flow category over $L$.

         \begin{rmk}
             Working with negative gradient flow (which corresponds to a `positively adapted' gradient field in the terminology of \cite{Abbaspour-Laudenbach}) means that our Morse complex computes absolute homology. 
         \end{rmk}

         \begin{rmk}\label{rmk: boundary ok}
             If $L$ is a manifold with boundary, and we pick a Morse function $f: L \to \bR$ with only interior critical points and with $-\nabla f$ inwards pointing at the boundary, there is still a well-defined flow category $\cM(f)$. A typical example of interest is to take $L$ the sublevel set of a Morse function on a closed manifold. 
         \end{rmk}

We will write $\cM^{LL}$ for the flow category associated to the pair $(L, \phi_f^1(L))$ where $f$ is a Morse function on $L$. Then the set of Hamiltonian chords $\cX(L,L)$ between $L$ and $\phi_f^1(L)$ bijects with the critical points of the Morse function $f$. 
We record the following stronger Morse-Floer comparison, cf. \cite[Remark 5.3]{PS3}.

\begin{lem}\label{lem:morse-floer}
Let the Morse function $f: L \to \bR$ be sufficiently $C^2$-small. There is a taming almost complex structure $J$ with respect to which 
    there is an equivalence of \emph{unoriented} flow categories $\cM(f) \simeq \cM^{LL}[d]$.
\end{lem}

\begin{rmk}
    We will construct an equivalence (i.e. invertible bimodule) between the flow categories, rather than a diffeomorphism, to avoid having to directly compare the charts on the Morse and Floer sides, which are constructed somewhat differently. 
\end{rmk}
\begin{proof}
    We take the $C^2$-small autonomous Hamiltonian $H = f$ on a Weinstein neighbourhood of $L$, extended to the rest of $L$ via a cut-off function. The fact that, under the $C^2$-smallness condition, $f$ determines a time-independent almost complex structure $J_f$ for which the Morse and Floer moduli spaces are homeomorphic as stratified topological manifolds is standard. We remark that the proof has three ingredients: (i) all the time one $H$-chords from $L$ to itself are constant; (ii)  all $J_f$-Floer solutions are regular and $t$-independent and hence solve a Morse flowline  equation $u'(s) = J_fX_H = \grad(H)$; and (iii) the linearisations of the Morse and Floer equations along any such solution have the same kernel. 
    
The smooth structure on each side depends on a choice, e.g. of flat Riemannian metric near the critical points on the Morse side and of gluing profile on the Floer side. Indeed, the construction of smooth charts on Floer moduli spaces in \cite{Large,PS} uses a $J$ for which each Floer strip $u$ admits a dense set of `regular points' $z=(s,t) \in \bR\times [0,1]$, i.e points  where $u^{-1}(u(\bR\times\{t\})) = z$, $\partial_s u|_z \neq 0$ and $u(z)$ does not belong to any chord. By inspection, $J_f$ satisfies these conditions, so the construction of smooth charts on the Floer moduli spaces applies in this case.  The compatibilities of linearisations, (iii) above, amounts to saying that having fixed these  choices, the tangent microbundles of the Morse and Floer moduli spaces admit compatible vector bundle lifts (furthermore compatible under breaking).

It then follows from general smoothing theory  that the flow categories are \emph{stably diffeomorphic}, in the following sense.  For objects $x,y$ of the flow categories, we can find vector spaces $E_{xy}$, with embeddings $E_{xz} \oplus E_{zy} \to E_{xy}$, and smooth structures on the products $\cB_{xy} := \cM_{xy} \times [0,1] \times E_{xy}$, which interpolate between the given smooth structures on $\cM(f)_{xy}$ respectively $\cM^{LL}_{xy}$ at the ends.  (The need for the stabilising vector spaces $E_{xy}$ is to avoid subtleties in smoothing theory for manifolds of dimension $4$; in all other dimensions, concordance classes of smoothing of a topological manifold biject with homotopy classes of vector bundle lift of the tangent microbundle, whilst in dimension four the vector bundle lift only gives a smooth structure on the stabilisation of the manifold given by taking a product with a positive-dimensional Euclidean space.) The original smoothing theory is due to Lashof \cite{Lashof}, and was extended to the required relative setting and to systems of flow category moduli spaces by Bai and Xu \cite{Bai-Xu} and by Rezchikov \cite{Rezchikov}. 

We now have `derived manifold' presentations of the moduli spaces $\cM_{xy}$, i.e. we have bundles $E_{xy} \to \cM_{xy}$ which come with sections whose zero-set defines the original moduli space. (For comparison to \cite{Bai-Xu}, we have effected the `stabilisation move' on derived orbifold presentations from \cite{AMS,Bai-Xu} in the special case of trivial groups, so no orbifold points.) On the space $\cB_{xy}$, the tautological section is smooth over the ends $\{0,1\}$ of the interval factor -- because the smooth structure is constructed relative a given structure at the ends -- but only continuous globally. However, it can be replaced by a homotopic smooth section, whose zero-set $\cW_{xy}$ defines a cobordism between $\cM(f)_{xy}$ and $\cM^{LL}_{xy}$.  The fact that the stable moothings are constructed for all the flow category moduli spaces compatibly with breaking gives rise, upon constructing such smooth sections inductively in dimension, to a collection of cobordisms which defines a morphism $\cW$ from $\cM(f)$ to $\cM^{LL}$. This morphism is invertible when viewed as a morphism in $\mathrm{Flow}$, because the smooth structure on the concatenation of $\cB_{xy}$ with its reversal in the $[0,1]$-direction gives a smooth structure on $\cM(f) \times [0,2] \times E_{xy}$ which is concordant to the product smooth structure, and thus one can homotope the composition smooth section to the trivial tautological section. 

The existence of an invertible morphism gives the equivalence between $\cM(f)$ and $\cM^{LL}$ in the category of unoriented flow categories.
Finally, note that our Floer grading conventions are homological with unit in degree zero, cf.  \cite[Remark 1.4]{PS}, but our Morse gradings reproduce classical singular homology (with unit in degree $d$), so $|x|_{Floer} = |x|_{Morse} -d $.
\end{proof}

\begin{rmk}
    Replacing $f$ by $\varepsilon\cdot f$ for $\varepsilon>0$ sufficiently small, which does not change any of the Morse moduli spaces, shows that we can compute (a category equivalent to) $\cM^{LL}$ via any choice of Morse function on $L$.
\end{rmk}

\begin{rmk}
    The Morse flow category $\cM(f)$ admits the structure of a framed flow category, and \cite[Section 6]{Blakey} relates the index bundles on the two sides of the equivalence of Lemma \ref{lem:morse-floer} in that case. Given a $\Psi$-brane structure on $L$, along with stabilisation and index data, we may construct an $E_\Psi$-orientation on both the Morse flow category $\cM(f)$ and the invertible bimodule in Lemma \ref{lem:morse-floer}. Indeed one can push an $E_\Psi$-orientation through a stable diffeomorphism, and then the `thickenings' $\cB_{xy}$ are all $E_\Psi$-oriented, and so are the resulting bordisms.
\end{rmk}

 \begin{lem} \label{lem:change morse} Let $L$ be a $\Psi$-oriented Lagrangian.  
        Up to isomorphism in $\Flow^{E_\Psi}_{/L}$ the flow category $\cM(f)$ is independent of the choice of Morse function $f$ on $L$. When $\partial L \neq \emptyset$, the same result holds if one varies the Morse function through functions which satisfy that $-\nabla f$ is inwards-pointing along the boundary for all time.
        \end{lem}
        
        \begin{proof}
            This follows from the fact that the morphism of flow categories associated to a continuation of Morse functions (or $C^2$-small Floer data) naturally admits $E_{\Psi}$-orientation data and  lives over $L$.
        \end{proof}

\subsection{An endomorphism computation}

In a small number of cases, we can completely identify the endomorphisms of an object of the spectral Fukaya category.

Assume $\Psi=fr$. Abouzaid-Blumberg \cite{AB} construct an equivalence (of homotopy categories) $\Flow^{fr} \to \bS$-mod, and \cite{Blakey} shows (using the Morse-theoretic model discussed above) that under this equivalence $\cM^{LL}$ is taken to the suspension spectrum $\Sigma^{\infty}_+ L$. A variation on Blakey's argument shows:

\begin{lem}\label{lem:22}
    Let $(L, \xi)$ be an object in $\scrF^{loc}(X; fr)$. Under the equivalence $\Flow^{fr} \to \bS$-mod, the flow category $\cM^{\xi_L,\xi_L}$ is also taken to $\Sigma^{\infty}_+L$.
\end{lem}
\begin{proof}
    We can identify $\cM^{LL}$ and $\cM(f)$ as $E_{fr}$-oriented flow categories, since the $\Theta$-orientation on $L$ is induced from a stable framing.  For critical points $x,y$ there are maps
    \begin{equation}\label{eqn:twisting morse flow cats}
    \cF_{xy} \to \cP_{xy}(L) \times \cP_{yx}(L) \stackrel{\sim}{\longrightarrow} \Omega L \times \Omega L \to GL_1(\bS) \times GL_1(\bS)
    \end{equation}
    where the first arrow evaluates along the two boundary components. Such a map yields, given the boundary stabilisation data underlying the $E_{fr}$-orientations, maps
    \[
    \cF_{xy} \times (S^{N_{xy}} \times S^{N_{yx}}) \to S^{N_{xy}} \times S^{N_{yx}}
    \]
    (where the two $S^N$-factors arise from the two boundaries of the strip, which may have different rank boundary stabilisations). By definition the moduli spaces of the twisted flow category $\cM_{xy}^{\xi_L\xi_L}$ are given by preimages of $(0,0)$ under such.
    
    Since we start with a Morse flow category, the first map in \eqref{eqn:twisting morse flow cats} lands in the diagonal, and can be retracted back to a space of constant paths. A coherent choice of such retractions under breaking yields a (framed) bordism and thence an equivalence between $\cM^{LL}$ and $\cM^{\xi_L\xi_L}$.
\end{proof}

\subsection{Ascending and descending manifolds}\label{Sec:ascending descending}

  We recall some conventions on ascending (compactified unstable) and descending (compactified stable) manifolds of a Morse function, and their Floer analogues in the setting of Lemma \ref{lem:morse-floer}, where they are moduli spaces of perturbed holomorphic half-planes.

By definition,  the ascending manifold $\cA(x)$ is 
a compactification of a space of half flow-lines $\gamma: (-\infty,0] \to L$, which are asymptotic to $x$ as $s \to -\infty$ under \emph{positive} gradient flow. Thus $\cA(x)$ is a space of \emph{positive} half-flow-lines asymptotic \emph{to} $x$ from the viewpoint of negative gradient flow.  Thus, we have 
\begin{equation} \label{eqn:break ascending}
\partial \cA(x) = \cA(x') \times \cM(f)_{x'x}.
\end{equation}
Similarly, the descending manifold is a space of maps $[0,\infty) \to X$ which are positive gradient flow half-lines, asymptotic to $y$, equivalently a space of maps from the negative half-line originating from $y$ under negative gradient flow, so
\begin{equation} \label{eqn:break descending}
\partial \cD(y) = \cM(f)_{yy'} \times \cD(y').
\end{equation}
See Figure \ref{Fig:ascending_descending}.

\begin{figure}[ht]
\begin{center} 
\scalebox{0.75}{
\begin{tikzpicture}

\begin{scope}[decoration={
    markings,
    mark=at position 0.5 with {\arrow{>}}}
    ] 
\draw[semithick,postaction={decorate}] (-2,2) -- (-2,-2);
\draw[semithick,postaction={decorate}] (-3,2) -- (-4,0);
\draw[semithick,postaction={decorate}] (-4,0) -- (-3,-2);
\end{scope}

\draw[fill] (-2,2) circle (0.1);
\draw[fill,gray!30] (-2,-2) circle (0.1);
\draw (-2,2.5) node {$y$};
\draw (-1.25,0) node {$-\nabla f$};
\draw (1.25,0) node {$-\nabla f$};

\draw[fill] (-3,2) circle (0.1);
\draw[fill] (-4,0) circle (0.1);
\draw[fill,gray!30] (-3,-2) circle (0.1);
\draw (-3,2.5) node {$y$};
\draw (-4.5,0) node {$y'$};

\begin{scope}[decoration={
    markings,
    mark=at position 0.5 with {\arrow{>}}}
    ] 
\draw[semithick,postaction={decorate}] (2,2) -- (2,-2);
\draw[semithick,postaction={decorate}] (3,2) -- (4,0);
\draw[semithick,postaction={decorate}] (4,0) -- (3,-2);
\end{scope}

\draw[fill,gray!30] (2,2) circle (0.1);
\draw[fill] (2,-2) circle (0.1);
\draw (2,-2.5) node {$x$};
\draw[fill,gray!30]  (3,2) circle (0.1);
\draw[fill] (4,0) circle (0.1);
\draw[fill] (3,-2) circle (0.1);
\draw (3,-2.5) node {$x$};
\draw (4.5,0) node {$x'$};

\end{tikzpicture}
}
\end{center}
\caption{Descending manifold $\cD(y)$ of dimension $\ind(y)$ and ascending manifold $\cA(x)$ of dimension $d-\ind(x)$, with arrows labelling negative gradient flow}
\label{Fig:ascending_descending}
\end{figure}

There are canonical evaluation maps $ev$ from $\cA(x)$ and $\cD(y)$ to $L$, via evaluation at the end-point  $\gamma \mapsto \gamma(0)$. By taking the metric to be standard near the critical points, we always assume that the ascending and descending manifolds $\cA(x)$ and $\cD(y)$ admit the structure of smooth manifolds with corners, see \cite{Wehrheim} (see also \cite[Theorem 3.7]{Cip} for closely related discussions).

\begin{lem}
    The descending manifolds $\cD(y)$ define a morphism $\cM(f) \to \ast[0]$. The ascending manifolds $\cA(x)$ define a morphism $\ast[d] \to \cM(f)$.
\end{lem}

\begin{proof}
In our grading convention \eqref{eqn:damned conventions}, $|q| = \ind(q)$ and the moduli spaces are spaces of negative gradient flow-lines.  The fact that the descending manifolds form a left module and the ascending manifolds a right module follows from \eqref{eqn:break ascending} and \eqref{eqn:break descending}.   To check the claimed gradings, note that the open stratum of $\cD(y)$ is a cell of dimension $\ind(y) = |y|$ whilst that of $\cA(x)$ has dimension $\dim_{\bR}(L)-\ind(x) = d - |x|$. 
\end{proof}

  \begin{lem} 
          The fibre products $\cD(y) \times_L \cA(x)$ are smooth compact manifolds with corners for generic data.
      \end{lem}

\begin{proof}
     See e.g. \cite[Proposition 6.7]{Joyce}); in general  fibre products of manifolds with corners over closed manifolds (rather than over more general manifolds with corners) are generically transverse. If $L$ has boundary, note that $ev: \cA(x) \to L$ never hits $\partial L$.
\end{proof}

There are moduli spaces $\cE_{*x}$ of perturbed holomorphic half-planes with a single outgoing boundary puncture, asymptotic to $x$, and the spaces $\Upsilon_{x*}$ of half-planes with a single incoming boundary puncture, constructed for instance in \cite[Theorems 5.2(3) and 5.6(1)]{PS3}. We view these moduli spaces as maps of domains equipped with a boundary marked point as well as the incoming / outgoing puncture.

The proof of Lemma \ref{lem:morse-floer} adapts to relate moduli spaces of gradient half-lines with such spaces of perturbed holomorphic half-planes. It is helpful to add a further boundary puncture to these half-planes, so their domains are conformally strips.

\begin{defn}\label{defn:weighted puncture}
    An outgoing respectively incoming  \emph{weighted strip-like end} at a boundary puncture of an abstract disc $D$ is a strip-like end $[-2T,\infty) \times [0,1]$ respectively $(-\infty, 2T] \times [0,1]$ (with co-ordinate $s+it$) at outgoing respectively incoming punctures, equipped with a \emph{weight function} $e^{out/in}_{\delta}$. Writing `interior' for the complement of slightly larger strip-like ends, the weight is given by
    \[
    e^{out}_\delta = \begin{cases} e^{\delta (s + 2T)} & s > 1-2T \\ 1 & (s,t) \, \mathrm{interior} \end{cases}  \qquad e^{in}_{\delta} = \begin{cases} e^{\delta (-s + 2T)} & s < 2T-1 \\ 1 & (s,t) \, \mathrm{interior} \end{cases} 
    \]
    \end{defn}

Suppose now we have an abstract disc $\bD = (D,\bC^N,\Lambda)$ comprising a disc $D$, a complex vector bundle $\bC^N \to D$ and a totally real boundary condition $\Lambda \subset \bC^N|_{\partial D}$. When the boundary conditions either side of a puncture agree (as opposed to being transverse, as we usually assume), the Cauchy-Riemann operator on our usual Sobolev space $W^{2,\kappa}$  is not Fredholm. 
    Instead, for a $\bC^N$-valued $1$-form $Y$ and choice of $\delta >0$, there is an associated Fredholm operator 
    \begin{equation} \label{eqn:weighted dbar}
    \overline{\partial}_J + Y: W^{2,\kappa}_{\delta}(D,\bC^N,\Lambda) \oplus \oplus_{\xi} V_{\xi} \to \Omega^{0,1}_{W^{2,\kappa-1}_{\delta}}(D,\bC^N)
    \end{equation}
    on a  weighted Sobolev space $W^{2,\kappa}_{\delta}$ of functions which have  finite norm after multiplication by the exponentially increasing weight function in the end.  Since the weight forces all sections to vanish at the puncture, we incorporate further finite-dimensional vector spaces $V_{\xi}$ in the domain of the operator,  indexed by the weighted punctures $\xi$, which one can view as spanned by cut-off constant solutions or by basis vectors for the (unique!) totally real space $\Lambda_{\xi}$ at the given puncture.  Assuming for simplicity there is a unique weighted puncture, the associated norm on a pair 
    \[
    (s,v) \in W^{2,\kappa}_{\delta}(D,\bC^N, \Lambda) \oplus V_\xi
    \]
    is then given by
    \[
\sum_{j=0}^{\kappa} \| \nabla^j s\|^2_{\mathrm{interior}} + \sum_{j=0}^{\kappa} (e_\delta^{out/in})^2  \| \nabla^j(s-v^{\parallel})\|^2_{\mathrm{neck}} + \|v\|^2
    \]
   where $v^{\parallel} \in V_{\xi}$ denotes the extension of $v \in \Lambda_\xi$  to a local section through parallel transport (with respect to a fixed background metric).

\begin{rmk}
    Extending the discussion of abstract discs in Section \ref{sec:tp and ad}, one can introduce simplicial spaces $\bU^{\Psi}_{\dagger 1}$ and $\bU^{\Psi}_{1\dagger}$ of 
$\Psi$-oriented abstract discs $\bD$ of rank $0 \leq N \leq \infty$ with:

\begin{itemize}
    \item two boundary punctures, one incoming and one outgoing, exactly one of which $(\dagger)$ is labelled as weighted;
    \item a complex bundle $E \to \bD$ of rank $N$ with totally real subbundles $\Lambda$ over each boundary component, and lifts of the corresponding elements of $BU(N)$ respectively $BO(N)$ to $\Psi$; 
    \item $\Psi$-oriented puncture data, \emph{which at weighted punctures comprises paths to totally real subspaces $\Lambda$ of the given complex bundle $E$ which agree on either side of the puncture} (rather than being transverse as previously);
    \item strip-like ends $[-2T,\infty) \times [0,1]$ respectively $(-\infty, 2T] \times [0,1]$ (with co-ordinate $s+it$) at outgoing respectively incoming punctures, equipped with weight functions as in Definition \ref{defn:weighted puncture}; 
    \item a Hermitian connection $\nabla$ and an $E$-valued $1$-form $Y$, and a virtual perturbation of the associated Fredholm operator $\overline{\partial}^{\nabla}_J + Y$ from \eqref{eqn:weighted dbar}     (meaning a surjection to the cokernel from a finite-dimensional vector space of fields which we assume are supported in the interior).
    \end{itemize}

As in the case without weighted punctures, the kernels of the stabilised operators define an index bundle on the associated simplicial space. Note that the simplicial spaces $\bU^\Psi_{\dagger 1}$ and $\bU^\Psi_{1\dagger}$ are spaces over $\Theta \to BO$ by taking the (unique) totally real boundary condition at the weighted puncture. Gluing of discs with weighted punctures, which is briefly reviewed before Lemma \ref{lem:identify identity} below, constructs a map
\[
\bU^{\Psi}_{1\dagger} \times_{\Theta} \bU^{\Psi}_{\dagger 1} \to \bU^{\Psi}_{11}
\]
on the subspace of pairs of discs where the $\Theta$-lifts of the real boundary conditions on the two discs co-incide at the weighted punctures; this is needed to equip the glued strip with a totally real boundary condition and a $\Theta$-lift thereof.

 Given a Floer solution $u: (\Sigma,\partial \Sigma) \to (X,L)$  with a weighted puncture at $\xi \in \partial\Sigma$, the map $u$ extends smoothly across the puncture.  There is a canonical isomorphism between the kernel of the usual $\overline{\partial}$-operator on the domain with a boundary marked point at $\xi$ and the kernel of the corresponding operator with the weighted puncture, defined on a weighted Sobolev space augmented with a finite-dimensional space of cut-off solutions to span the tangent space of the Lagrangian at the puncture as in \eqref{eqn:weighted dbar}, cf. for instance \cite[Section 7.2]{Ekholm-Smith}. Thus, the forgetful map
 \[
 \bU_{\dagger 1}^\Psi \to \bU_{01}^\Psi
 \]
 which forgets the weighted puncture altogether induces an isomorphism on index bundles. 
 \end{rmk}

\begin{lem}\label{lem:morse-floer-ascending-descending}
    In the setting of Lemma \ref{lem:morse-floer}, there are bordisms (compatible with breaking) $\cA_{\ast x}\cong \cE_{\ast x}$ and $\cD_{y*} \cong \Upsilon_{y*}$. 
\end{lem}

\begin{proof}
    We follow the usual construction of the unit and counit moduli spaces from \cite[Theorem 5.13]{PS3}, but equip the domains with a boundary marked point, which we then replace by a puncture with a weighted strip-like end. This means that the domains for the unit and counit moduli spaces can now be viewed as strips $\bR\times [0,1]$ rather than half-planes. 
    
  Given a time-dependent Hamiltonian $H_t$ and a cut-off function $\rho(s)$ with $\rho(s) = 0$ for $s\leq 0$ and $\rho(s) = \epsilon$ for $s\geq 1$, there is a `PSS-type' Hamiltonian
\[
H^{PSS} = (\varepsilon-\rho(s)) + \rho(s)H_t(x)
\]
and a corresponding Hamiltonian vector field $X_{H^{PSS}} \otimes dt$ on the strip which vanishes for $s \leq 0$. 
    Taking $H_t = f$ to be our given $C^2$-small autonomous Hamiltonian $H=f$, Floer solutions to the equation are $t$-independent and co-incide with gradient flow lines, by the same argument as underlies Lemma \ref{lem:morse-floer}.
Moreover, all solutions extend smoothly over the weighted puncture, so the moduli spaces agree with the moduli spaces originally used to define the (co)unit. One then runs the same argument (passing via stable diffeomorphisms of flow modules, rather than flow categories, to cobordisms of such) as in the proof of Lemma \ref{lem:morse-floer}.
\end{proof}

Given that the $E_{\Psi}$-orientation on $\cM(f)$ is itself obtained by pullback under its isomorphism with $\cM^{LL}$, we may now equip the ascending and descending modules $\cA$ and $\cD$ with $E_{\Psi}$-orientations by pullback from the $E_\Psi$-orientations on $\cE$ and $\Upsilon$.  Summarising: 

\begin{cor} \label{cor:morse-floer-summary}
    There are isomorphisms \begin{equation} \label{eqn:morse floer bordism groups}
\Omega_*^{E_{\Psi}}(\cM(f)) \cong \Omega_*^{E_{\Psi}}(\cM^{LL}) \qquad \mathrm{and} \qquad \Omega^*_{E_{\Psi}}(\cM(f)) \cong \Omega^*_{E_{\Psi}}(\cM^{LL})
\end{equation}
taking the ascending manifolds module $\cA$ to the unit module $\cE$, and the descending manifolds module $\cD$ to the counit module $\Upsilon$.
\end{cor}

    One can glue an outgoing weighted puncture to an incoming weighted puncture, and the glued strip inherits a weight function.  Given a pair of Floer solutions $\textbf{u}=(u_1,u_2)$ with one outgoing respectively incoming weighted puncture $\xi_i$, $i=1,2$, and with a common value $u_i(\xi_i) = p \in L$ and any $T$ sufficiently large, there is a pregluing $u_1 \#_T u_2 = u_T^{(0)}$ defined on a domain which contains a long neck region $Z(T) :=[-2T,2T]\times [0,1]$. Fukaya, Oh, Ohta and Ono \cite{FO3:smoothness} make explicit the Newton-Picard iteration scheme which defines a convergent sequence $u_T^{(k)}$ with limit $u_T = \mathrm{exp}_{u_T^{(0)}}(\xi_{\textbf{u},T})$ given by the exponential of a vector field over the preglued solution.  They crucially prove the exponential decay estimate
\[
\| \nabla^j_{\textbf{u}} (\partial / \partial T)^{\ell} \xi_{\textbf{u},T} \|_{L^{2,\kappa-\ell}(Z(T))} \, < \, C\cdot e^{-\delta T}
\]
(where $C$ depends only on $\kappa$ and the ambient geometry).  Given this, the construction of smooth charts for moduli spaces of curves perhaps with weighted punctures can be carried out exactly as in \cite{Large, PS}.  In particular, there is a smooth structure on a moduli space of strips which contains fibre product boundary components $\Upsilon_{x*} \times_L \cE_{*y}$.

Recall the unit morphism $\bD\cF$ of a flow category $\cF$, introduced in  Lemma \ref{lem:flow unit}.

\begin{lem}\label{lem:identify identity}
    The identity morphism $\cM^{LL} \to \cM^{LL}$ is bordant to a morphism $\cW: \cM^{LL} \to \cM^{LL}$ with moduli spaces $\cW_{xy} = \Upsilon_{x*} \times_L \cE_{*y}$; equivalently the identity of $\cM(f)$ is bordant to a morphism with moduli spaces $\cD_{x*}  \times_L \cA_{*y}$.
\end{lem}

\begin{proof}
By Corollary \ref{cor:morse-floer-summary}, we can just work with $\cM(f)$. 
The interior (top open stratum) of the space $\cD(p) \times_L \cA(q)$ is exactly the space of parametrized negative gradient flow lines from $p$ to $q$.  The whole space  is the compactification of the space of gradient flow lines with a single marked point, the combinatorics of which exactly reproduces (and motivated) the definition of the conic degeneration, cf. \cite{AB2}.
\end{proof}

\begin{rmk}\label{rmk:PSS}
    In general, meaning for any not necessarily small or autonomous  Hamiltonian, the PSS construction gives rise to morphisms 
    \[
    PSS: \cM(f) \to \cM^{LL} \qquad \mathrm{and} \qquad SSP: \cM^{LL} \to \cM(f)
    \]
    with moduli spaces 
    \[
    PSS_{px} = \cD(p) \times_L \cE_{*x}, \qquad SSP_{yq} = \Upsilon_{y*} \times_L \cA(q)
    \]
    which define inverse equivalences, so define isomorphisms of groups of (unoriented) modules.  Consider 
    \[
    (PSS \circ \cA)_{*y} = \cup_p \, \cA(p) \times PSS_{py} = \cup_p \, \cA(p) \times (\cD(p) \times_L \cE_{*y})
    \]
    (where we have abusively written $\cup_p$ for a smoothing of the coequaliser which identifies certain faces in the disjoint union), and rewrite the last term as 
    \[
    \left( \bigcup_p \cA(p) \times \cD(p) \right) \times_L \cE_{*y}.
    \]
    The closure in $L \times L \times [0,\infty)$ of the forward graph $\{(x,y,t) \, | \, t\geq 0, y = \phi_{-\nabla f}^t(x)\}$ of the negative gradient field on $L \times L$ defines a manifold with corners in $L \times L \times [0,\infty]$ whose boundary strata are given by the  diagonal in $L \times L$ and the products $\cA(p) \times \cD(p)$, see \cite[Proposition 2.9]{Abbaspour-Laudenbach}, cf. also \cite{Harvey-Lawson} for an earlier result giving the relationship as currents. (Note that \cite{Abbaspour-Laudenbach} show that the closure has `conic singularities' but they enumerate all strata and in our case these are all manifolds with corners due to our stronger hypotheses on the Morse function.) The forward graph of the flow thus defines a bordism between the indicated module and the module with moduli spaces $\Delta_L \times_L \cE_{*y} = \cE_{*y}$. In other words, the PSS and SSP morphisms always relate the ascending / descending modules with the unit / counit modules.  However, our treatment of $E_{\Psi}$-orientations has been exclusively based on Floer theory, and one cannot straightforwardly pull back an $E_{\Psi}$-orientation on $\cE$ to one on $\cA$ when their moduli spaces are only bordant, but not necessarily diffeomorphic. 
\end{rmk}

\subsection{$E_\Psi$-orientations on fibre products}

The orientation on the identity morphism $\bD\cF$ of a flow category $\cF$ appeals to the composition $(\bD\cF)_{xy} \to \cF_{xy} \to \bU^{\Psi} = E_{\Psi}$, cf. \cite[Lemma 4.21]{PS3}. We wish to show that this $E_{\Psi}$-orientation on the identity morphism is compatible with its fibre product description in Lemma \ref{lem:identify identity}.

Let $\cM, \cN$ denote spaces of discs with respectively one incoming and one outgoing puncture, each with a boundary marked point, and with boundary on a Lagrangian $L$. (In the discussion of $E_{\Psi}$-orientations in the proof of \cite[Theorem 5.13]{PS3}, a moduli space $\cM$ of discs with no ouput puncture is equipped with an interior marked point with asymptotic marker. We now choose the boundary marked point to be the point on the boundary to which the asymptotic marker points.)   There are thus natural evaluation maps
\[
\cM \to L \qquad \mathrm{and} \qquad \cN \to L
\]
and we will assume the fibre product $\scrX :=\cM \times_L \cN$ is transverse. Note there is an induced map 
\[
\ev_{\scrX}: \scrX \to L.
\]

By \cite[Theorem 5.13]{PS3}  (appealing for $\cN$ to the special case for moduli spaces with no outputs), we have maps
\[
\cM \to \Base(E_{\Psi})\times\Theta \qquad \mathrm{and} \qquad \cN \to \Base(E_{\Psi})
\]
the latter of which factors through a map which forgets the boundary marked point; the $\Theta$-term comes from the (interior or boundary) marked point on $\cM$, which is oriented relative to the classifying map of the tangential structure $\phi$. 

The kernel of the linearised Cauchy-Riemann operator on a nodal disc is (by definition) the subspace of the kernel of the operator on the normalisation of pairs of vector fields which agree at the marked points arising from normalising the nodes. In this picture, the index bundle $I^{\scrX}$ is not directly related to the pullback of $E_{\Psi} \oplus E_{\Psi}$ under the map $\scrX \to \Base(E_{\Psi}) \times \Base(E_\Psi)$, but should be corrected by a factor of $TL$.  

Indeed, assume that we have stabilisation data $\nu_{\cM}, \nu_{\cN}$ for $\cM$ and $\cN$ and index data $d_{\cM},d_{\cN}$, giving rise to $E_{\Psi}$-orientations as in Section \ref{sec:E-orientations recap} and constructed in detail in \cite[Theorem 5.13]{PS3}. The classifying maps
$\cM \to \Base(E_{\Psi})$ and $\cN \to \Base(E_{\Psi})$ factor through $\Base(E^*_{\Psi})$ where $E^*_{\Psi}$ is the corresponding Thom $\cI$-space built from abstract discs with an additional boundary marked point. The induced map
\[
\scrX \to \Base(E^*_{\Psi}) \times \Base(E^*_{\Psi})
\]
lands in the fibre product over  evaluation maps $\ev: \Base\,E^*_{\Psi}(\nu) \to BO(\nu)$,  taking the value of the real boundary condition at the marked point, which on a Floer moduli space $\cM$ or $\cN$ becomes a map valued in a stabilisation of $TL$.

\begin{defn}
We will say that $\scrX$ is $E_{\Psi}$-oriented if the classifying map $\cM \times \cN \to \Base(E_{\Psi}) \times \Base(E_{\Psi})$ is covered by maps of Thom spaces 
\begin{equation} \label{eqn:bott gluing}
(\scrX,I^{\scrX})(\nu_{\cM} + \nu_{\cN}) + \ev_{\scrX}^*TL \to E(\nu_{\cM} + d_{\cM}) \times E(\nu_{\cN}+d_{\cN}).
\end{equation}
\end{defn}

Given a morphism $\cW$ of flow categories with one boundary component such a fibre product, an $E$-orientation on  $\cW$ will comprise maps $\cW_{xy} \to \Base(E_{\Psi})$ which on any fibre product boundary component satisfy \eqref{eqn:bott gluing}.

\begin{lem}\label{lem:identify identity 2}
     The identity morphism $\cM(f) \to \cM({f})$ is $E_{\Psi}$-oriented bordant to a morphism $\cW: \cM({f}) \to \cM(f)$ with moduli spaces $\cW_{xy} = \cD_{x*}  \times_L \cA_{*y}$.
\end{lem}

 \begin{proof} This is a variation, incorporating weighted punctures, of the steps underlying the rather involved proof of \cite[Theorem 5.13]{PS3}, so we will just give a sketch. 
In general, we have two constructions of an $E_{\Psi}$-orientation on a space  $\scrX = \cM \times_L \cN$.  First, viewing $\scrX$ as a fibre product whose $E_{\Psi}$-orientation is induced from that on the factors, we  have maps
\begin{equation} \label{eqn:orient as fibre product}
\xymatrix{
\cM \times_L \cN \ar[r]_-{\mathrm{incl}} & \cM\times \cN \ar[r]_-{\mathrm{taut}} & \bU^{\Psi}_{10} \times \bU^{\Psi}_{01} \ar[r]_-{(co)cap} & \bU^{\Psi}_{01} \times \bU^{\Psi}_{01} \ar[r]_-{\oplus} & \bU^{\Psi}_{01}
}
\end{equation}
where the arrow labelled $(co)cap$ acts by attaching a framed cap and then a framed `cocap' on the first factor and does nothing to the second, reflecting the proof of \cite[Theorem 5.13]{PS3}. 

Separately, viewing $\scrX$ as a boundary stratum of a space of strips (which is relevant to the $E_{\Psi}$-orientation arising from the composition $\bD\cF \to \cF \to \bU_{01}$ coming from the identity morphism), we have maps
\begin{equation}\label{eqn:orient as boundary}
\xymatrix{
\cM\times_L \cN \ar[r]_-{\mathrm{taut}} & \bU^{\Psi}_{1\dagger} \times_{\Theta} \bU^{\Psi}_{\dagger1} \ar[r]_-{\mathrm{glue}} & \bU^{\Psi}_{11} \ar[r]_-{cap} & \bU^{\Psi}_{01}
}
\end{equation}
where now the arrow labelled $cap$ attaches a framed cap to the strip. (In both these schematics, we have simplified the steps from the proof of \cite[Theorem 5.13]{PS3} by omitting the steps separating out framed components with trivial index bundle and steps related to the inclusion of spaces of discs with fixed domain in all abstract discs.) Note also that since all discs here have boundary marked points, we are dealing with cousins of the simplicial spaces which both live over $\Theta$, and the `tautological' domain arrow in the top line does naturally land in a fibre product over $\Theta$. In the discussion before Lemma \ref{lem:morse-floer-ascending-descending} we compared  the spaces of discs with a weighted puncture and with just a boundary marked point.
     
The key point is therefore that the two operations
\[
\mathrm{(attach\, cap)} \circ \mathrm{(glue)}  \qquad \mathrm{and} \qquad \mathrm{(attach \, cap \, and \, cocap, identity)} \circ \mathrm{(sum)}
\]
(where the attached caps and cocap are framed) induce the same isomorphism on index bundles.  In \eqref{eqn:orient as boundary}, the gluing step and attaching the framed cap commute, so one gets the same isomorphism from
\[
\xymatrix{
\bU^{\Psi}_{1\dagger} \times_{\Theta} \bU^{\Psi}_{\dagger1} \ar[r]_-{cap} & \bU^{\Psi}_{0\dagger} \times_{\Theta} \bU^{\Psi}_{\dagger 1} \ar[r]_-{\mathrm{glue}} &  \bU^{\Psi}_{01}
}
\]
whereas on \eqref{eqn:orient as fibre product} one can separate out the gluing on a framed cap from gluing on a framed cocap, to make the essential step come from a map
\[
\xymatrix{
\bU^{\Psi}_{10} \times \bU^{\Psi}_{01} \ar[r]_-{cap} & \bU^{\Psi}_{00} \times \bU^{\Psi}_{01} \ar[rr]_-{cocap\,  \mathrm{and} \, \oplus} && \bU^{\Psi}_{01}
}
\]
The first steps in the two compositions above both involve gluing a framed cap to the disc in $\cM$, so commute under the natural maps which replace weighted punctures by marked points.

On the image of a moduli space $\cM\times_L \cN$ we can homotope both maps to land in the subspace $\bU^{\Psi} \subset \bU_{01}^{\Psi}$ corresponding to fixing the domain. Compatibility with the maps on Thom spaces now follows from the fact that under gluing (for sufficiently large gluing parameter, which we always assume holds over our finite set of finite-dimensional moduli spaces), orthogonal projection induces a canonical isomorphism between the index of the glued disc and the direct sum of the indices of the components. Similarly,  direct summing abstract disc data over a fixed domain also induces an isomorphism of the index to the direct sum of the indices of the constituents. 
 \end{proof}

\section{Evaluation local systems\label{sec:ELS}}

Throughout this section we fix a commutative tangential pair $\Psi$. The hypothesis that $\Psi$ is commutative  ensures that there is a Fukaya category which contains $\Psi$-oriented branes equipped with monodromy local systems. 

We will introduce `evaluation local systems', namely flow categories $\cF \in \Flow^{E_{\Psi}}_{/L}$ whose $0$-th truncation is isomorphic to the flow category defined by a $\bZ$-local system on $L$.  The main goal of the section will be to prove that every such evaluation local system arises naturally from a $\Psi$-oriented monodromy local system on $L$, so the two notions of local system are essentially equivalent.  We emphasise that this section is entirely independent of the discussion of the open-closed map and the $\cO\cC$-fundamental class associated to a monodromy local system and twisted flow category from \cite[Section 8]{PS3}, as reviewed in Section \ref{sec:Fuk and OC}.

\begin{rmk}
    Throughout the section we work over $\bZ$, but the proofs carry over  with minimal change for the unoriented tangential structure $\Psi_{unor} = (BO \to BU)$ if we work over $\bZ/2$. In particular, the analogue of Theorem \ref{thm: WELS} holds in that setting.
\end{rmk}

\subsection{A Viterbo restriction functor} \label{sec: Vit}
Let $L\subset X$ be a $\Psi$-oriented compact exact Lagrangian.  We construct a `restriction functor'
\[
\cR: \scrF(X;\Psi) \longrightarrow {\Flow}^{E_\Psi} _{/L}
\]

\begin{rmk}
    If $L$ is graded and exact but does not  have a $\Psi$-orientation, there is still a restriction functor to $\Flow_{/L}$.  This reflects the fact that Viterbo restriction should exist to any convex open subset of $X$, and $\Flow_{/L}$ is playing the role of a `categorical open neighbourhood' of $L \subset D(T^*L) \subset X$.  We work in the $\Psi$-oriented case since that is most relevant to our intended applications.
\end{rmk}

Let $L$ be a manifold, perhaps with non-empty boundary.  If $\xi: L \to B(GL_1(R_{\Psi}))$ is a monodromy local system, we will write $\cM(f)^\xi$ for the twist of the flow category $\cM(f)$ by $\xi$. This is constructed with respect to a particular choice of  boundary stabilisation data, and its $E_{\Psi}$-orientation further invokes the given stabilisation and index data entering into the $E_{\Psi}$-orientation on $\cM(f)\sim \cM^{LL}$, however it is well-defined up to isomorphism.

\begin{lem}\label{lem:melon to flow}
    There is a function 
    \begin{equation} \label{eqn:restrict L to L}
    [L,BGL_1(R_\Psi)_{h\cI}] \to \mathrm{Ob}\, {\Flow}^{E_{\Psi}}_{/L}, \qquad \xi \mapsto \cM(f)^{\xi}
    \end{equation}
    from the set of (homotopy classes of) spectral local systems to the set of (isomorphism classes  of) $E_{\Psi}$-oriented flow categories over $L$.
\end{lem}

\begin{rmk}
    We note that \eqref{eqn:restrict L to L} depends only on a small Stein neighbourhood of $L$ in $X$, which we can take to be a sufficiently small disc cotangent bundle.  
\end{rmk}

\begin{proof}
    For a fixed Morse function, a homotopy of local systems induces a morphism  of flow categories over $L$, which one checks is an isomorphism by concatenating with the morphism associated to the inverse homotopy. Changing the Morse function is dealt with by Lemma \ref{lem:change morse}. More precisely, noting that our $E_{\Psi}$-orientation arises from identifying $\cM(f) \sim \cM^{LL}$ with that arising from Floer theory for the Hamiltonian $\varepsilon f$, we can compare the categories associated to different Morse functions by working inside a small disc cotangent bundle of $L$. This works equally well when $L$ has boundary and the cotangent bundle is naturally a Liouville sector (or one could reduce to the closed case by doubling).
\end{proof}

When $\Psi$ is oriented, a monodromy local system induces a classical $\bZ$-local system, by using the existence of a map $R_{\Psi} \to H\bZ$ and the composition
\[
L \to BGL_1(R_{\Psi}) \to BGL_1(H\bZ) \simeq K(\bZ/2,1)
\]
cf. \cite[Lemma 7.4]{PS3}. 
Either using that and Lemma \ref{lem:melon to flow}, or just directly by considering Floer theory with local coefficients, one has:

\begin{cor}\label{cor: efoguhr9ouhg}
    There is a function
    \[
    \left\{ \bZ\mhyphen\textrm{local  systems on\,} L \right\} \longrightarrow \mathrm{Ob}\,\tau_{\leq 0} {\Flow} _{/L}, \qquad \xi_{\bZ} \mapsto \cM(f)^{\xi_{\bZ}}
    \]
\end{cor}

Recall that if $(K,\xi_K)$ and $(K',\xi_{K'})$ are branes with local systems, we write $\cM^{\xi_K,\xi_{K'}}$ for the twisted flow category. If $\xi_K$ is a trivial local system, we continue to write $K$ for $\xi_K$. 

The graded Lagrangian $L$ defines, for any $\xi_K$ and $\xi_{K'}$, a bilinear map 
\[
\cM^{L\xi_K  \xi_{K'}}: \cM^{L,\xi_K} \times \cM^{\xi_K \xi_{K'}} \to \cM^{L,\xi_{K'}} 
\]
(obtained from moduli spaces of holomorphic triangles), 
which admits an $E_\Psi$-orientation (cf. \cite[Theorem 5.13]{PS3}. If $\beta \in \scrF_*(\xi_K,\xi_{K'})$ is a morphism, then by gluing a disc with a unique outgoing puncture labelled by $\beta$ to the 3-punctured disc underlying the bilinear map (which has two incoming and one outgoing punctures), we obtain a strip with one incoming and one outgoing puncture with associated boundary conditions $(L, \xi_K)$ and $(L,\xi_{K'})$ respectively. This defines a morphism $L \circ \beta$ of flow categories $\cM^{L,\xi_K} \to \cM^{L,\xi_{K'}}$.

\begin{lem}
    The morphism $L \circ \beta$ admits a canonical homotopy class of $E_\Psi$-orientations. 
\end{lem}

\begin{proof} This is a special case of \cite[Theorem 5.13]{PS3}.
\end{proof}

\begin{defn}
    Define $\cR$ to act as follows:
    \begin{itemize}
        \item on objects: $(K,\xi_K) \mapsto \cM^{L,\xi_K}$ (with the $E_\Psi$-orientation from Lemma \ref{lem:twists are oriented}, cf. \cite[Section 7.5]{PS3});
        \item on morphisms:  $\beta \in \scrF_i(\xi_K, \xi_{K'})$ is mapped to $L \circ \beta \in [\cM^{L,\xi_K}, \cM^{L,\xi_{K'}}]$, where $[\cdot,\cdot]$ denotes the morphism group in ${\Flow^{E_{\Psi}}_{/L}}$.
    \end{itemize}
\end{defn}

\begin{prop}\label{prop:viterbo functor}
    $\cR$ defines a functor.
\end{prop}

\begin{proof}
    This is a direct application of gluing, involving a moduli space of discs with four marked points which plays the role of an associator living over $L$. For inputs $\beta_1 \in \scrF_*(K,K';\Psi)$ and $\beta_2 \in \scrF_*(K',K'';\Psi)$, the moduli spaces defining $\cR(\beta_1 \circ \beta_2)$ and those defining $
    \cR(\beta_1) \circ \cR(\beta_2)$ are given by the top left respectively right of Figure \ref{Fig:Viterbo} (where we have suppressed labelling the local systems for simplicity), and both are bordant to the moduli space shown in the bottom layer of Figure \ref{Fig:Viterbo}. The various breakings are manifestly compatible with evaluation to the path spaces of $L$. On the other hand, the bordism admits an $E_{\Psi}$-orientation compatible with those already associated to the moduli spaces appearing in the top row as a particular instance of \cite[Theorem 5.13]{PS3}.
\end{proof}

\begin{figure}[ht]
\begin{center}
\begin{tikzpicture}

\draw[semithick] (-8,1) -- (-3,1);
\draw[semithick] (-8,2.5) -- (-3,2.5);
\draw[semithick] (-3.5,1.5) -- (-3,1.5);
\draw[semithick] (-3.5,2) -- (-3,2);
\draw[semithick] (-3.5,1.5) arc (270:90:0.25);
\node[left = 1mm of {(-3.75,1.75)}] {$K$};
\node[below = 1mm of {(-6,1)}] {$L$};
\node[above = 1mm of {(-6,2.5)}] {$K''$};

\draw[semithick] (-2.75,2.5) -- (-2.5,2.5);
\draw[semithick] (-2.5,2.5) arc (-90:0:0.25);
\draw[semithick] (-2.25,2.75) arc (180:90:0.25);
\draw[semithick] (-2,3) -- (-1,3);

\draw[semithick] (-2.75,2) -- (-2.5,2);
\draw[semithick] (-2.5,2) arc (90:0:0.25);
\draw[semithick] (-2.25,1.75) arc (180:270:0.25);
\draw[semithick] (-2,1.5) -- (-1,1.5);

\draw[semithick] (-1.25,2) -- (-1,2);
\draw[semithick] (-1.25,2.5) -- (-1,2.5);
\draw[semithick] (-1.25,2.5) arc (90:270:0.25);
\node[left = 1mm of {(-1.35,2.25)}] {$K'$};

\draw[semithick,dotted] (-0.75,3) arc (90:-90:0.25);
\draw[semithick,dotted] (-0.75,2) arc (90:-90:0.25);
\node[right = 0.25 mm of {(-0.5,2.75)}] {$\beta_2$};
\node[right = 0.25 mm of {(-0.5,1.75)}] {$\beta_1$};

\draw[semithick] (0.5,3) -- (2.5,3);
\draw[semithick,dotted] (2.75,3) arc (90:-90: 0.25);
\draw[semithick] (2,2.5) -- (2.5,2.5);
\draw[semithick] (2,2.5) arc (90:270:0.25);
\draw[semithick] (2,2) -- (3.25,2);
\draw[semithick] (0.5,0.5) -- (3.25,0.5);

\draw[semithick] (3.5,0.5) -- (6,0.5);
\draw[semithick] (3.5,2) -- (6,2);
\draw[semithick, dotted] (6.25,2) arc (90:-90:0.25);
\draw[semithick]  (5.75, 1.5) -- (6,1.5);
\draw[semithick] (5.75,1.5) arc (90:270:0.25);
\draw[semithick] (5.75,1) -- (6,1);

\node[below =1 mm of {(2,0.5)}] {$L$};
\node[above = 1 mm of {(1.5,3)}] {$K''$};
\node[above = 1 mm of {(5,2)}] {$K'$};
\node[left = 1 mm of {(5.5,1.25)}] {$K$};

\node[right = 0.25 mm of {(6.5,1.75)}] {$\beta_1$};
\node[right = 0.25 mm of {(3,2.75)}] {$\beta_2$};

\draw[semithick] (-3,-2) -- (3,-2);
\draw[semithick] (-3,-4.5) -- (3,-4.5);
\node[below = 1mm of {(0,-4.5)}] {$L$}; 
\node[above = 1mm of {(0,-2)}] {$K''$}; 
\node[left = 1mm of {(2,-2.75)}] {$K'$};
\node[left = 1mm of {(2,-3.75)}] {$K$};

\begin{scope}
\clip (2.5,-3) rectangle (0.5,0.5);
\draw[semithick] (2.5,-2.75) circle(0.25); 
\end{scope}

\draw[semithick] (2.5,-2.5) -- (3,-2.5);
\draw[semithick] (2.5,-3) -- (3,-3);
\draw[semithick] (2.5,-3.5) -- (3,-3.5);
\draw[semithick] (2.5,-4) -- (3,-4);

\begin{scope}
\clip (2.5,-4) rectangle (0.5,0.5);
\draw[semithick] (2.5,-3.75) circle(0.25); 
\end{scope}

\draw[semithick, dotted] (3.25,-2.5) arc (-90:90:0.25);
\node[right = 1mm of {(3.5,-2.25)}] {$\beta_2$};

\draw[semithick, dotted] (3.25,-3.5) arc (-90:90:0.25);
\node[right = 1mm of {(3.5,-3.25)}] {$\beta_1$};

\end{tikzpicture}
\end{center}
\caption{Functoriality of restriction to $\mathrm{Flow}^{\Psi} _{/L}$\label{Fig:Viterbo}}
\end{figure}

\subsection{Evaluation local systems from restriction}

Fix $L$ a compact manifold of dimension $n$. Let $f: L \to \bR$ be Morse (with $-\nabla f$ inwards-pointing along $\partial L$). 
 Lemma \ref{lem:morse-floer} says that the untwisted Morse and Floer flow categories $\cM(f)$ and $\cM^{LL}$ are unoriented equivalent, and the $E_{\Psi}$-orientation on $\cM(f)$ is by definition inherited from that equivalence.

    \begin{defn}\label{def: ELS}
        A ($\Psi$-oriented) \emph{evaluation local system} (ELS) is an $\cF \in \operatorname{Flow}^{E_\Psi}_{/L}$ such that $\tau_{\leq 0} \cF \in \tau_{\leq 0}\operatorname{Flow}^{E_\Psi}_{/L}$ is isomorphic to $\cM(f)^{\xi_{\bZ}}$ for some $\bZ$-local system $\xi_{\bZ}$ on $L$. 
    \end{defn}

\begin{lem} \label{lem:MLS to WELS}
    There is a canonical map
    \[
    \left\{\Psi\mhyphen \mathrm{MLS} \, \mathrm{on}\, L\right\} \longrightarrow \left\{\Psi\mhyphen \mathrm{ELS} \, \mathrm{on}\, L \right\}, \qquad \xi \mapsto \cM^{\xi_L,L} \simeq \cM(f)^{\xi}
    \]
\end{lem}

\begin{proof}
     This is the content of Lemma \ref{lem:melon to flow}  in this language. 
\end{proof}
    
    \begin{prop}\label{prop:quis over Z means WELS}
        Let $(K, \zeta)$ be an object in $\scrF(X; \Psi)$, and let $L \subseteq X$ be a $\Psi$-oriented Lagrangian, equipped with the trivial local system.

        If the images of $(K, \zeta)$ and $L$ in $\scrF(X; \bZ)$ are isomorphic, then $\cR(K, \zeta) \in \Flow^{E_\Psi}_{/L}$ is a ELS. 
    \end{prop}

    \begin{proof}
        There is a commutative diagram
        \[
        \xymatrix{
        \scrF(X;\Psi) \ar[r]^-{\cR} \ar[d]_{\tau_{\leq 0}} & \Flow^{E_{\Psi}}_{/L} \ar[d]_{\tau_{\leq 0}} \\ 
        \scrF(X;\bZ) \ar[r]^-{\cR} & \bZ[\pi_1(L)]\mhyphen\textrm{mod}
        }
        \]
        where $\cR$ denotes spectral Viterbo restriction on the upper line and classical Viterbo restriction on the lower line, and the vertical arrows are truncation functors. Since $(K,\zeta)$ and $L$ have quasi-isomorphic images in $\scrF(X;\bZ)$, they also have quasi-isomorphic images in $\bZ[\pi_1(L)]$-mod, which implies that $\cR(K,\zeta)$ satisfies the conditions to define an ELS. 
    \end{proof}

    \subsection{Classification of ELSs}\label{sec: clas els}

    In this section we prove a classification result for ELSs which amounts to saying that the map of Lemma \ref{lem:MLS to WELS} is surjective. Let $L \subset X$ be a $\Theta$-oriented Lagrangian submanifold; we allow the case in which $L$ has non-empty boundary.
    
    \begin{thm}\label{thm: WELS}
        Any $\Psi$-oriented evaluation local system $\cF$ is isomorphic in $\Flow^{E_\Psi}_{/L}$ to $\cM(f)^{\zeta}$ for some $\Psi$-monodromy local system $\zeta$ on $L$. 
    \end{thm}
    \begin{rmk}
        We allow $L$ to have boundary because the proof involves an induction over a handle decomposition of $L$.
    \end{rmk}
    \begin{rmk}\label{rmk:no uniqueness claimed}
        It is also true that $\zeta$ is unique up to homotopy, but we do not prove or use this.
    \end{rmk}

    \begin{proof}[Set-up for the proof of Theorem \ref{thm: WELS}]

        The proof will be carried out by an induction over the handles in a handle decomposition of $L$, and will occupy the next two sections.  We set-up the induction here. 
        Let $f: L \to \bR$ be a self-indexing Morse function, and let $L_i = f^{-1}(-\infty, i+\frac12]$ for all $i$. Let $\cF$ be a ELS on $L$; let $\xi$ be the corresponding $\bZ$-local system.

        Fix some $i$. Assume (for simplicity) that $L_{i+1} = L_i \cup H$ where $H$ is a single $(i+1)$-handle; the general case is identical (except where otherwise indicated). Let $P$ be the (unique) critical point of $f$ in $H$; note $P$ has index $i+1$. Write $f_j$ for $f|_{L_j}$ and $\xi_j$ for $\xi|_{L_j}$ for all $j$.

        By induction on $i$, we may assume Theorem \ref{thm: WELS} holds for $L_i$.
    \end{proof}

    \subsection{Simplifying the ELS}\label{sec: simp els}

        By Lemma \ref{lem: 0-modification}, replacing $\cF$ with a shift of an isomorphic object if necessary, we may assume that $\tau_{\leq 0} \cF = \cM(f_{i+1})^{\xi_{i+1}}$, for a $\bZ$-local system $\xi_{i+1} \to L_{i+1}$ with the property that $(\xi_{i+1})|_{L_i}$ is the $\bZ$-local system canonically obtained from a $\Psi$-MLS on $L_i$. 

        \begin{rmk}
            Now $Ob(\cF)=\Crit(f_{i+1})$, with the same gradings. In particular, for $x \in \cF$, either $x= P$ and $|x|=i+1$, or $0 \leq |x| \leq i$.
        \end{rmk}

        \begin{lem}\label{lem: wels simp 1}
            Replacing $\cF$ with an isomorphic object if necessary, we may assume that for $x, y \in \cF \setminus \{P\}$, the image of $\Gamma_{xy}: \cF_{xy} \times I \to L_{i+1}$ does not intersect the handle $H$.
        \end{lem}
        \begin{proof}
            The cocore of the handle $H$ has dimension $n-i-1$. Since $0 \leq |x|,|y| \leq i$, $(\cF_{xy} \times I)$ has dimension $\leq i$. Therefore by generically perturbing all $\Gamma_{xy}$, we may assume that the cocore of the handle (which has dimension $d-i-1$) is not in their image. By Example \ref{ex: ev hom iso flow} the isomorphism type of the resulting flow category is unchanged in $\Flow^{E_\Psi}_{/L_{i+1}}$.

            Then postcomposing all evaluation maps with the large-time flow of $-\nabla f_{i+1}$, we find that the images of all $\Gamma_{xy}$ (for $x,y \neq P$) are disjoint from all of $H$, not just $P$.
        \end{proof}
        \begin{rmk}
            If there are several $(i+1)$-handles and $f$ is chosen to be self-indexing, then this argument adapts to let one push the $\Gamma_{xy}$ off all the top index handles.
        \end{rmk}
        Let $\cH$ be the flow category with objects $Ob(\cF) \setminus \{P\}$, and $\cH_{xy} = \cF_{xy}$. By Lemma \ref{lem: wels simp 1}, the $\Gamma_{xy}$ make $\cH$ into an object of $\Flow^{E_\Psi}_{/L_i}$ (as opposed to $\Flow^{E_\Psi}_{/L_{i+1}}$).

        \begin{lem} \label{lem: cF is a cone}  
            $\cF$ is isomorphic to $\Cone(\cW: *[i] \to \cH)$, for some morphism $\cW$ in $\Flow^{E_\Psi}_{/L_{i+1}}$.
        \end{lem}
        \begin{proof}
        There is no data in an $E_{\Psi}$-orientation of the flow category $\ast[i]$ (the grading determines the value of $i$), so the result is essentially tautological.  

        \end{proof}

        By construction, $\tau_{\leq 0} \cH = \cM(f_i)^{\xi_i}$, and so $\cH$ is an ELS over $L_i$. Therefore by induction, $\cH$ is isomorphic in $\Flow_{/L_i}^{E_\Psi}$ to $\cM(f)^{\zeta_i}$, for $\zeta_i$ some monodromy local system over $L_i$. By Lemma \ref{lem: 0-modification}, we may assume that $\cH = \cM(f_i)^{ \zeta_i}$.

        Let $S \subset \partial L_i$ be the attaching sphere for the handle $H$, so $S \cong S^i$. Let $g: S \to \bR$ be a standard Morse function with two critical points $Q$ and $R$, of index $i$ and $0$ respectively. Choose a tubular neighbourhood $\iota': S \times D^{d-i-1} \to \partial L_i$ of $S$ (recall that the normal bundle of an attaching sphere is always trivial), and extend this to an embedding $\iota: S \times D^{d-i-1}_t \times [-1,1]_s \to L_i$ restricting to $\iota'$ along $(\ldots) \times \{s=0\}$. 
        \begin{lem}
            There is a Morse function $f': L_{i+1} \to \bR$, with $-\nabla f'$ inwards-pointing along $\partial L_i$, such that on $\operatorname{Im}(\iota)$, $f'$ is given by $g+|t|^2+s^2+c$ (where $c$ is some constant) and agreeing with $f$ over $H$.
        \end{lem}

        \begin{proof}
            We assume $f$ looks like the ``standard'' Morse function on $H = D^{i+1}_r \times D^{d-i-1}_t$; this is $h = |t|^2-|r|^2+c$. It's straightforward that one can choose $f'$ to agree with $h$ over $H$ and $f_i$ over $L_i$.

            Then we can choose any Morse function $f'$ on $L_i$ which has $-\nabla f'$ inwards-pointing on $\partial L_i$, and glue it to this local model, to get a Morse function on $L_{i+1}$.

            We choose the local model in the statement in a neighbourhood (in $L_i$) of the attaching sphere $S$ of $H$; this is inwards along where it touches $\partial L_i$. Combined with the local model on $H$, this gives a Morse function on the union of $H$ and a neighbourhood of $S$ in $L_i$. Then extend this arbitrarily to a Morse function (inwards on $\partial L_i$) on the rest of $L_i$.
        \end{proof}

        By Lemma \ref{lem:change morse}, we may assume henceforth that $f_{i+1} = f'$.

        Now by construction, $\cM(f_{i+1})
        _{QR} = \cM(g)_{QR} \cong S^{i-1}$, and 
         \begin{equation}
            \cM(f_{i+1})_{Pz} \cong \begin{cases}
                \operatorname{pt} & \textrm{ if } z=Q\\
                D^i & \textrm{ if } z=R\\
                \emptyset & \textrm{ otherwise}
            \end{cases}
        \end{equation}
        \begin{lem} \label{lem:slide bott gluing}
            $\cW$ is $E_{\Psi}$-oriented bordant to a morphism $\cW': *[i] \to \cH$ in $\Flow^{E_\Psi}_{/L_{i+1}}$ for which 
            \[
            \cW'_{Pz} = \emptyset \qquad \mathrm{for} \quad  z \not \in \{Q,R\}
            \]
            and with $\tau_{\leq 0}\cW = \tau_{\leq 0}\cW'$.
        \end{lem}
        \begin{proof}
           Let $\bD\cH$ be the identity morphism of $\cH$, which inherits an $E_{\Psi}$-orientation from that on $\cH$ by Lemma \ref{lem:flow unit}. Certainly $\cW$ is $\Psi$-oriented bordant over $L_{i+1}$ to $\cW^\# := \bD\cH \circ \cW$.  Then
           \begin{equation}\label{eq: reouhgeoudbg}
               \cW^{\#}_{Pz} = \bigcup_q \cW_{Pq} \times \bD\cH_{qz} = \, \bigcup_q \cW_{Pq} \times \cA(q) \times_{L_{i+1}} \cD(z)
           \end{equation}
           using the description of the identity morphism of a Morse flow category from Section \ref{Sec:ascending descending}. Write
           \[
           \cX_P = \bigcup_q \cW_{Pq} \times \cA(q).
           \]
           Since we are working with flow categories over a target, there is a map $\ev: \cX_P \to \cP_{P \to L_i}L$, where the codomain indicates the space of paths which begin at $P$ and end somewhere in $L_i$.  We now consider a bordism of the morphism $\cW^{\#}$ obtained by sliding the point at which we evaluate for forming the fibre product with $\cD(z)$ along the paths defined by $\ev$. Once we have deformed this back to the boundary $\partial L_i$, the evaluation lands in the attaching sphere $S$ for the handle containing $P$. It follows that the only descending manifolds which can be in the target of the deformed evaluation are those for $z \in \{Q,R\}$. This yields both the required morphism $\cW'$ and the bordism from $\cW'$ to $\cW^{\sharp}$. It follows from the construction that we do not change the $0$-th truncation in this process.
        \end{proof}
        \begin{rmk}\label{rmk: live over disc}
            Note that the morphism $\cW'$ constructed in Lemma \ref{lem:slide bott gluing} also satisfies the property that the map $ev: \cW_{PR} \to \cP_{PR} L$ land in the space of gradient flows $\cM(f)_{PR} \cong D^{i-1}$. 
            
            It sends the boundary $\partial \cW_{PR} \cong \cW_{PQ} \times \cF_{QR}$ to the boundary $S^{i-2}$, and is degree 1: on the boundary, it is degree 1 since $\cW_{PQ} \cong \pt$ and $\cF_{QR}$ is obtained from $\cM(f)_{QR}$ by twisting; for the map $\cW_{PR} \to D^i$ it follows from the fact that a boundary-preserving map between compact manifolds has the same degree as its restriction to boundaries.
        \end{rmk}

         \begin{rmk}\label{rmk:take one space to be just a point}
            The description of $\cF$ as a cone in Lemma \ref{lem: cF is a cone} determines $\tau_{\leq 0}\cW$. That in turn determines $\tau_{\leq 0}\cW'$, and in particular we may assume that $\cW'_{PQ} = \{pt\}$.
        \end{rmk}

\begin{rmk}\label{rmk:relabel cW}
    By Proposition \ref{prop: cone prop over L}, we can replace $\cW$ by $\cW'$.  From now on, we may therefore assume that $\cF$ is equal to the cone on a morphism $\cW: \ast[i] \to \cH$ which satisfies the conclusions of Lemma \ref{lem:slide bott gluing}. 
\end{rmk}

\subsection{Killing the obstruction}\label{sec: killing ob}

Recall that the inductive step has provided a monodromy local system $\zeta_i \to L_i$, whose homotopy class is given by an element of $[L_i, B(GL_1^{\Psi})]$, where we abbreviate $BGL_1(R_{\Psi}) =: B(GL_1^{\Psi})$. We have two goals:

\begin{itemize}
    \item Show that the local system $\zeta_i \to L_i$ extends to a local system $\zeta_{i+1} \to L_{i+1}$;
    \item show that $\cF \simeq \cM(f_{i+1})^{\zeta_{i+1}}$.
\end{itemize}

We first identify the obstruction to the first step. 

\begin{lem} \label{lem:obstruction to extension over top handle}
    The moduli space $\cM(f_{i+1})_{QR}$ defines a class in $\pi_{i-1}(GL_1^{\Psi})$, well-defined up to sign, which is the obstruction to extending $\zeta_i$ as a monodromy local system to $L_{i+1}$.
\end{lem}

\begin{proof}
    $L_{i+1}$ is homotopy equivalent to $L_i$ with an $(i+1)$-cell attached along $S = S^i$, so the obstruction to extending $\zeta_i$ can be identified with a class in (the set of unbased homotopy classes of maps) $[S, BGL_1^{\Psi}]$. In the special case $i=0$ (so $L_i$ is just a disc) the obstruction vanishes since $BGL_1^{\Psi}$ is connected, so henceforth assume $i\geq 1$.

    Using that $\pi_i(BG) = \pi_{i-1}(G)$ for any group-like topological monoid $G$, we have
    \begin{equation} \label{eqn:obstruction defined to sign}
    \pi_1(BGL_1^{\Psi}) = \pi_0(GL_1^{\Psi}) = \pi_0(R_{\Psi})^{\times} = \bZ/2
    \end{equation}
    (since by hypothesis $E_{\Psi}$ is oriented, cf. Section \ref{sec:flow over target}, which in particular implies that $\pi_0\Thom(E_{\Psi}) = \bZ$), so the corresponding based homotopy class and obstruction is well-defined up to sign.

The twisted flow category $\cM(f_i)^{\zeta_i}$ is defined, via a choice of twisting data for the monodromy local system $\zeta_i$, so that for all critical points $p,q$ of $f_i$,
\[
\cM(f_i)^{\zeta_i}_{pq} \subset \cM(f_i)_{pq} \times S^{n_{pq}}\times S^{n_{qp}}
\]
is the zero-locus of a map to the Thom space $\Thom(\bV(n_{pq}) \oplus \bV(n_{qp}))$. Under the isomorphism
\[
\pi_{i-1}(GL_1^{\Psi}) \stackrel{\sim}{\longrightarrow} \Omega_{i-1}^{E_\Psi}
\]
(strictly this is an injection rather than an isomorphism when $i=1$, but the argument is unchanged)  the class of Lemma \ref{lem:obstruction to extension over top handle} is exactly the $\Psi$-oriented bordism class of $\cM(f_i)^{\zeta_i}_{QR} = \cF_{QR}$. This is because the twisting data on $\cM(f_i)_{QR}$ is pulled back from $S$.
\end{proof}

\begin{rmk}\label{rmk: epughrdouighrdpsghdisrg}
    Lemma \ref{lem:obstruction to extension over top handle} makes essential use of commutativity of the tangential pair. If we did not assume commutativity, and worked in the spectral Fukaya category set up in \cite{PS2}, then the bordism class of $\cF$ would live in $\pi_*(R_\Psi \otimes R_\Psi^{op})$ by \cite[Proposition 5.39]{PS2}, which is not in general isomorphic to $\pi_*(GL_1^\Psi)$.
\end{rmk}

\begin{lem}
    $\cF_{xy}$ is $E_\Psi$-oriented nullbordant.
\end{lem}

\begin{proof}
    For the morphism $\cW: \ast[i]\to\cH$ arranged as in Remark \ref{rmk:relabel cW}, Lemma \ref{lem:slide bott gluing}, we have 
    \[
    \partial \cW_{PR} = \cW_{PQ} \times \cH_{QR} 
    \]
    and for critical points which are not $P$, $\cF$ and $\cH$ have exactly the same moduli spaces, with the same $E_\Psi$-orientations. We conclude $\cF_{QR} = \partial \cW_{PR}$, by Remark \ref{rmk:take one space to be just a point}.
\end{proof}

It follows that the obstruction identified in Lemma \ref{lem:obstruction to extension over top handle} vanishes.  That is, there is an extension of the map $S^{i-1} \to GL_1^{\Psi}$ arising from Lemma \ref{lem:obstruction to extension over top handle} to a map $D^i \to GL_1^{\Psi}$.  Let $\zeta_{i+1} \to L_{i+1}$ denote the resulting (homotopy class of) monodromy local system.

\begin{rmk}\label{rmk:characterise extension over handle from nullbordism}
    The previous proof furthermore shows that the extension defining $\zeta_{i+1}$ is uniquely defined up to homotopy by insisting that the associated bordism class represents $\cW'_{PR}$ viewed as a nullbordism of $\cF_{QR}$.   
\end{rmk}

\begin{cor}
    The MLS $\zeta_{i+1}$ satisfies that
    \[
    \cM(f_{i+1})^{\zeta_{i+1}}_{Pq} = \cW'_{Pq} \qquad \mathrm{when} \ q \in \{Q,R\}
    \]
\end{cor}

\begin{proof}
Since $\cM(f_{i+1})_{PQ}$ is a point, the twisted manifold $\cM(f_{i+1})^{\zeta_{i+1}}_{PQ}$ will be a compact zero-manifold which is algebraically a single point. This can be  identified with $\cW_{PQ}$ by the hypothesis that $\tau_{\leq 0} \cF$ is a $\bZ$-local system twist of $\cM(f_{i+1})$.

    The identification
\[
D^i = \cM(f_{i+1})_{PR}
\]
and Remark \ref{rmk:characterise extension over handle from nullbordism} shows that $\cW_{PR}$ arises from the construction of twisted manifolds, applied to the moduli space   $\cM(f_{i+1})_{PR} = D^i$.
\end{proof}

Summarising the constructions up to this point, we find:

    \begin{prop}\label{prop: obstruction killed}
        There is a monodromy local system $\zeta_{i+1}$ on $L_{i+1}$, such that:
        \begin{enumerate}
            \item $\zeta_{i+1}|_{L_i} = \zeta_i$
            \item The $\bZ$-local system on $L_{i+1}$ induced by $\zeta_{i+1}$ is $\xi_{i+1}$.
            \item $\cM(f)^{\zeta_{i+1}}_{QR} = \cW_{PR}$ 
            \item $\zeta_{i+1}$ sends $\cW_{PQ}$ (which is a point by Remark \ref{rmk:take one space to be just a point}) to the unit element in $GL_1^{\Psi}$.
        \end{enumerate}
        
    \end{prop}

    It then follows from Lemma \ref{lem: cF is a cone} that $\cM(f_{i+1})^{ \zeta_{i+1}} = \cF$, completing the proof of Theorem \ref{thm: WELS}.

\subsection{Lifting quasi-isomorphisms}

 We retain the notation and set-up of the previous section, in particular we fix a Morse function $f$ on the $\Psi$-oriented brane $L$. 

\begin{lem}
    If $(Q,\zeta_Q)$ and $(K,\zeta_K)$ are quasi-isomorphic in $\scrF^{\loc}(X;\Psi)$, then $\cM^{L,\zeta_Q} \sim \cM^{L,\zeta_K}$ in $\Flow^{\Psi}_{/L}$.
\end{lem}

\begin{proof}
    This follows immediately from the fact that the spectral restriction $\cR$ is a functor, cf. Proposition \ref{prop:viterbo functor}.
\end{proof}

A flow category $\cM^{LL}$ admits two different structures as a flow category over $L$, corresponding to evaluation along the two boundary components of the moduli spaces of strips (note that either choice is coherent under breaking).  It follows that given $\zeta_L\to L$ one can form two different twists, $\cM^{\zeta_L,L}$ and $\cM^{L,\zeta_L}$ (the `left' respectively `right' twists), corresponding in terms of boundary stabilisation data to setting $n_b = 0$ for one of the two possible choices of element $b$ of $\mathrm{Cpt}(\partial \Delta)$.

\begin{lem}\label{lem:common twist}
If $\cM^{LK} \sim \cM^{LQ}$ are equivalent in $\Flow^{\Psi}_{/L}$ for branes $K,Q$ (perhaps equipped with local systems), then for any MLS $\zeta_L \to L$ the left twists $\cM^{\zeta_L,K}$ and $\cM^{\zeta_L,Q}$ are also equivalent. 
\end{lem}

\begin{proof}
The hypothesis means there are morphisms $\cW_{KQ} \in [\cM^{LK},\cM^{LQ}]$ and $\cW_{QK} \in [\cM^{LQ},\cM^{LK}]$ whose compositions are bordant to the corresponding  unit morphisms
\[
\cW^{KQ} \circ \cW^{QK} \sim \bD(\cM^{LQ})
\]
(and where everything is $E_{\Psi}$-oriented, which we drop to spare the notation).

We now choose all data for constructing the twisted moduli spaces associated to $\xi$ coherently, including for the moduli spaces associated to $\cW^{KQ}, \cW^{QK}, \bD(\cM^{LK}), \bD(\cM^{LQ})$. For the unit morphisms we recall \cite{PS3} that the index data, stabilisation data and boundary stabilisation data are all inherited from the underlying flow category. The key claim is now that the twist of the unit morphism is the unit morphism for the twisted flow category.

The twist of the flow category morphism spaces, respectively the unit morphism moduli spaces, arise from coherent-under-breaking diagrams
\[
\xymatrix{
\bD\,\cM^{LK}_{xy} \times S^{n_{xy} + n_{yx}} \ar[r] \ar[d] & \Thom (\bV^{\Psi}(n_{xy} + n_{yx})) \ar@{=}[d] \\
\cM^{LK}_{xy} \times S^{n_{xy} + n_{yx}} \ar[r] & \Thom (\bV^{\Psi}(n_{xy} + n_{yx}))
}
\]
by taking the preimages of zero along the top and bottom lines.  In particular, writing $(\mhyphen)^{\xi}$ for the twist of a moduli space $(\mhyphen)$, we have that
\[
\cM^{\xi_L,K}_{xy} := (\cM^{LK})^{\xi} \subset \cM^{LK}_{xy} \times S^{n_{xy}+n_{yx}}
\]
is a submanifold with corners, so the inclusion is an immersion transverse to all corner strata. On the other hand, \cite[Diagram (6.11)]{AB2} shows that the conic degeneration is functorial under such immersions, so one concludes that
\[
(\bD\,\cM^{LK}_{xy})^{\xi} = \bD((\cM^{L,K})^{\xi}).
\]
It follows that the twists $(\cW^{LK})^\xi$ and $(\cW^{LQ})^\xi$ of the given morphisms define inverse isomorphisms between the twisted flow categories.
\end{proof}

\begin{cor}\label{cor:flow cats equal after local system twist}
If $(K,\xi_K)$ and $L$ are $\Psi$-oriented and $(K,(\xi_K)_{\bZ})$ and $L$ are quasi-isomorphic over $\bZ$,
there is a $\Psi$-oriented local system $\zeta$ over $L$ for which 
\begin{equation} \label{eqn:first consequence}
\cM^{L,{(K,\xi_K)}} \simeq \cM(f)^{\zeta} \simeq \cM^{L,(L,\zeta)}
\end{equation}
\end{cor} 

\begin{proof}  Combine Theorem \ref{thm: WELS} and Proposition \ref{prop:quis over Z means WELS}.
\end{proof}

\begin{cor}
    In the above setting, $(L,\zeta)$ and $(K,\xi_K)$ are quasi-isomorphic in $\scrF(X;\Psi)$.
\end{cor}

\begin{proof}
   We `twist on the other side' as in Lemma \ref{lem:common twist}, to deduce from \eqref{eqn:first consequence} that
   \begin{equation} \label{eqn:second consequence}
       \cM^{(L,\zeta), (K,\xi_K)} \simeq \cM^{(L,\zeta),(L,\zeta)}.
   \end{equation}
   The flow category on the right hand side has a unit morphism $e_{\zeta} \in \Omega_0^{\Psi}(\cM^{\zeta,\zeta}) = \scrF(\zeta,\zeta;\Psi)$, cf. Section \ref{sec:twist me} (see also \cite[Sections 5.1 and 7.6]{PS3}). The image of this under the equivalence on groups of right modules arising from \eqref{eqn:second consequence} defines a distinguished element $\hat{e}_{\zeta} \in \Omega_0^{\Psi}(\cM^{\zeta,\xi_K}) = \scrF(\zeta,\xi_K;\Psi)$.

   Now $\hat{e}_{\zeta}$ has the property that it $0$-th truncation defines a quasi-isomorphism over $\bZ$.  It follows that it is itself a quasi-isomorphism, by \cite[Remark 4.33]{PS3}, compare to \cite[Theorem 5.51 \& Corollary 7.13]{PS}.
   \end{proof}

This completes the proof of Theorem \ref{thm:main}.

\section{Stable Gauss maps}\label{sec: stab gaus}
    In this section we prove Theorem \ref{thm:main3}, constraining the stable Gauss maps of quasiisomorphic Lagrangians. Roughly speaking, if $K$ is $\Psi$-oriented but $L$ is not assumed to be, $\cM^{LK}$ is not necessarily an ELS over $L$, but something ``twisted'' by the stable Gauss map of $L$ (we call such things \emph{Gauss-twisted}, to differentiate them from the type of twisting occuring in Section \ref{sec:twist me}). The main technical result of this section, Proposition \ref{prop: gauss}, shows that existence of such a Gauss-twisted ELS places strong constraints on the stable Gauss map of $L$.

\subsection{Polarised tangential pairs}
    In this section, we recap from \cite{PS3} a large class of commutative tangential pairs.

    Let $\widetilde{U/O}^{or}$ be the 2-connected cover of $U/O$, and $Re: \widetilde{U/O}^{or} \to BO$ the map classifying the real part of a totally real subspace of $\bC^n$. 

    \begin{ex}\label{qi ex: tang str pol}
        Recall \cite[{Definition 3.6}]{PS3} that to any commutative $\cI$-monoid $F$ with a map $f: F \to \widetilde{U/O}^{or}$, we may associate a commutative tangential pair $\Psi^F=(\Theta^F, \Phi^F)$, defined by the following commutative diagram:
        \begin{equation}\label{eq: riwoghrdptgurdpotg}
            \xymatrix{
                \Theta^F = BO \times F
                \ar[rrr]_{Proj_1}
                \ar[d]_{Proj_1 +(Re \circ f)}
                &&&
                \Phi^F = BO
                \ar[d]_{\cdot \otimes \bC}
                \\
                BO
                \ar[rrr]_{\cdot \otimes \bC} 
                &&&
                BS_\pm U
            }
        \end{equation}
    \end{ex}
    We call such $\Psi^F$ \emph{polarised tangential pairs}. Note the induced map on homotopy fibres is exactly $f$.
    \begin{prop}[{\cite[Proposition 3.22]{PS3}}]\label{qi prop: mode baas sull}
        The assignment $F \mapsto E_{\Psi^F}$, from commutative $\cI$-monoids over $\widetilde{U/O}^{or}$ to connected commutative Thom $\cI$-monoids, defines an equivalence between the two homotopy categories.
    \end{prop}
    See \cite[{Section 3.3}]{PS3} for details on the statement, in particular a discussion on the model category structures underlying the homotopy categories here.
\subsection{Gauss-twisted evaluation local systems}\label{sec: gaus twis eval loca syst}

    Let $L$ be a compact manifold, possibly with boundary, and $g: L \to \widetilde{U/O}^{or}(N)$ some map; we assume $N \gg 0$ is large. Let $\cF \in \Flow_{/L}$ be a flow category over $L$.

    Composing the evaluation maps with $\Omega g$ gives maps $\cF_{xy} \to \Omega \widetilde{U/O}^{or}(N)$; concatenation in the flow category is then compatible with concatenation of loops on the target (rather than using its $\cI$-monoid structure). Let $or$ be the polarised tangential pair of $\widetilde{U/O}^{or}$ (cf. Example \ref{qi ex: tang str pol}); the ensuing commutative $\cI$-monoid satisfies $\bU^{or} \simeq BSO$.

    \begin{defn}
        \emph{Gauss-twist data for $\cF$} consists of stabilisation data $\{m_{xy}\}_{x,y \in \cF}$ for $\cF$, along with maps $\fg_{xy}: \cF_{xy} \to \bU^{or}(m_{xy})$, such that the following diagram commutes:
        \begin{equation}
            \xymatrix{
                \cF_{xy} \times \cF_{yz}
                \ar[r]
                \ar[d]
                &
                \cF_{xz}
                \ar[d]
                \\
                \bU^{or}(m_{xy}) \times \bU^{or}(m_{yz})
                \ar[r]
                &
                \bU^{or}(m_{xz})
            }
        \end{equation}
    \end{defn}
    \begin{lem}\label{lem: gaus twis data}
        $g$ determines a choice of Gauss-twisting data $\{m_{xy}, \fg_{xy}\}_{x,y}$ for $\cF$, uniquely up to homotopy and stabilisation.
    \end{lem}
    \begin{proof}
        Follows from an Eckmann-Hilton argument along with the equivalence $\bU^{or} \to \Omega \widetilde{U/O}^{or}$, similarly to Section \cite[Section 5.4]{PS3}.
    \end{proof}
    \begin{rmk}\label{rmk: gxy triv dim zero}
        Since $\bU^{or}$ is connected, we may assume the Gauss-twisting data $\{m_{xy}, \fg_{xy}\}$ induced from $g$ satisfies that $\fg_{xy}$ is constant at the basepoint whenever $\dim(\cF_{xy}) \leq 1$.
    \end{rmk}
    
    \begin{defn}\label{def: oerhgouergh}
        Let $E$ be a commutative Thom $\cI$-monoid. A \emph{$g$-Gauss-twisted $E$-orientation on $\cF$} consists of index data $\{\nu_{xy}\}$ and stabilisation data $\{d_{xy}\}$ as in Section \ref{sec:E-orientations recap}, as well as maps of Thom spaces:
        \begin{equation}\label{eq: gireghortbvhprifv}
            (\cF_{xy}, I^\cF_{xy})(\nu_{xy}) \to \left(\Base E(\nu_{xy}+d_{xy}) \times \bU^{or}(m_{xy}), E \oplus \bV\right)
        \end{equation}
        We require that the induced map of spaces $\cF_{xy} \to \bU^{or}(m_{xy})$ is exactly $\fg_{xy}$.

        These are required to satisfy an appropriate associativity condition, that the following diagram commutes:
        \begin{equation}
            \xymatrix{
                (\cF_{xy}, I^\cF_{xy})(\nu_{xy}) \times (\cF_{yz}, I^\cF_{yz})(\nu_{yz})
                \ar[rr]
                \ar[d]
                &&
                (\cF_{xz}, I^\cF_{xz})(\nu_{xz})
                \ar[d]
                \\
                \substack{
                    \left(\Base E(\nu_{xy}+d_{xy}) \times \bU^{or}(m_{xy}), E \oplus \bV\right) 
                    \\ 
                    \times \left(\Base E(\nu_{yz}+d_{yz}) \times \bU^{or}(m_{yz}), E \oplus \bV\right)
                }
                \ar[rr]
                &&
                \left(\Base E(\nu_{xz}+d_{xz}) \times \bU^{or}(m_{xz}), E \oplus \bV\right)
            }
        \end{equation}
    \end{defn}
    Note that this depends on a choice of twisting data $\{m_{xy}, \fg_{xy}\}_{x,y}$ induced by $g$; strictly speaking this must be part of the data but we suppress this from the notation, as justified by Lemma \ref{lem: gaus twis data}.

    One may define Gauss-twisting data for morphisms and bordisms of flow categories over $L$ in the same way, and $g$ induces a choice of this similarly. Similarly to Definition \ref{def: oerhgouergh}, one may define $g$-Gauss-twisted $E$-oriented morphisms and bordisms, to obtain a category of $g$-Gauss-twisted $E$-oriented flow categories, which we write as $\Flow^{E/g}_{/L}$; we may also define its truncations $\tau_{\leq i}\Flow^{E/g}_{/L}$ in the natural way.

    \begin{rmk}\label{rmk: gaus zero twis}
        It follows from \ref{rmk: gxy triv dim zero} that the $\tau_{\leq 0}$- version of the category, $\tau_{\leq 0}\Flow^{E/g}_{/L}$, is equivalent to $\tau_{\leq 0}\Flow^E_{/L}$, i.e. $g$-Gauss-twisting does not change the $0^{th}$ truncation.
    \end{rmk}

    \begin{rmk}\label{rmk: g lift no twis}
        Let $F \to \widetilde{U/O}^{or}$ be a map of commutative $\cI$-monoids. Suppose we are given a homotopy lift $\tilde g$ of $g$ to $F(N)$. This induces compatible homotopy lifts of the $\fg_{xy}$ to $\bU^{\Psi^F}(m_{xy})$. So if $E=E_{\Psi^F}$, we obtain a genuine $E$-orientation on $\cF$. Similarly, $\tilde g$ induces a functor from $g$-twisted $E$-oriented flow categories over $L$ to $E$-oriented flow categories over $L$.
    \end{rmk}
    \begin{defn}
        Assume $E$ is oriented. A \emph{$g$-Gauss-twisted $E$-oriented evaluation local system (ELS)} consists of a flow category $\cF$ over $L$, along with a $g$-Gauss-twisted $E$-orientation, such that $\tau_{\leq 0}{\cF}$ is isomorphic to $\cM(f)^{\xi_{\bZ}}$, for $\xi_\bZ$ some $\bZ$-local system on $L$, and $f: L \to \bR$ some Morse function (with $-\nabla f$ inwards-pointing along $\partial L$).
    \end{defn}

\subsection{Tangential pairs and the $J$ homomorphism}\label{sec: tang J}
    
    Let $C \to BSO$ be a map of connected commutative $\cI$-monoids. Let $R_C = \Thom(C \to BSO)$ be the corresponding commutative ring Thom spectrum.

    Let $BGL_1(R_C)$ be the delooping of $GL_1(R_C)$ (as commutative $\cI$-monoids, cf. \cite[Section 3.3]{PS3} and \cite{Sagave-Schlichtkrull:Diagram} for discussions on this construction), and let $BO \to BGL_1(R_C)$ be the map obtained from the $J$-homomorphism by delooping. Let $C^\circ$ be the homotopy fibre of the map $BO \to BGL_1(R_C)$; note that $C^\circ$ in fact lives over $BSO$.
    \begin{prop}\label{prop: 11.6}
        \begin{enumerate}
            \item There is a commutative diagram in the homotopy category of commutative $\cI$-monoids:
            \begin{equation}\label{eq: teoihgrpgh}
                \xymatrix{
                    C
                    \ar[d]
                    \ar[dr]
                    &
                    &
                    \\
                    C^\circ
                    \ar[r]
                    &
                    BO
                    \ar[r]
                    &
                    BGL_1(R_C)
                }
            \end{equation}
            \item For a finite CW complex $L$ and a vector bundle $V$ of rank $r$ classified by $V: L \to BO(r)$, there is a natural bijection between homotopy classes of homotopy lift of $V$ to $C^\circ$ and $R_C$-orientations on $V$ (meaning a map from the Thom spectrum $\Thom(V \to L) \to \Sigma^r R_C$ such that the composition with the inclusion of a fibre, $\Sigma^r \bS \to \Sigma^r R_C$, represents a unit in $\pi_0 R_C$, cf. \cite[Definition V.1.1]{Rud98} or \cite[Definition 2.17]{P}). 
            
            For a subcomplex $L' \subseteq L$, there is a similar bijection for those $R_C$-orientations/lifts extending given ones over $L'$.
        \end{enumerate}
    \end{prop}
    \begin{proof}
        \cite[Corollary 3.17]{Antolin-Camarena-Barthel} proves the composition $C \to BGL_1(R_C)$ is nullhomotopic (under the equivalence between $\bE_\infty$ spaces and commutative $\cI$-monoids), implying (1). (2) holds by \cite[Theorem 3.5]{Antolin-Camarena-Barthel}.
    \end{proof}

    \begin{Not}
        Proposition \ref{qi prop: mode baas sull} associates to the map $C \to C^\circ$ in $\cC\cS^\cI_{/BSO}$ from Proposition \ref{prop: 11.6} a map $F \to F^\circ$ of polarised tangential structures, well-definedly in the homotopy category. We let $\Psi=\Psi^F$ and $\Psi^\circ=\Psi^{F^\circ}$ be the polarised tangential pairs of $F$ and $F^\circ$ respectively. We fix these throughout the rest of Section \ref{sec: stab gaus}.
    \end{Not}
    \begin{rmk}\label{rmk: circ vers J} 
        Under the equivalence $\bU \simeq \Omega \widetilde{U/O} \simeq BO$ we obtain a well-defined map in the homotopy category of commutative $\cI$-monoids $BJ_C: \Omega\widetilde{U/O} \to B^2GL_1(R_C)$. 

        Then Proposition \ref{prop: 11.6} implies a map $g: L \to \Omega\widetilde{U/O}(N)$ lifts to $F^\circ(N)$ if and only if the composition $L \to B^2GL_1(R_C)(N)$ is nullhomotopic (for $N \gg 0$ large enough).
    \end{rmk}
    \begin{rmk}
        Since the functor $E$ is not an equivalence before passing to homotopy categories, $E_\Psi$ is not necessarily isomorphic to $C$. However, they are equivalent in the homotopy category, and hence give rise to the same Thom spectrum and bordism theory. 

        To avoid any complications arising from this, we implicitly replace $C$ with $E_\Psi$ (and $C^\circ$ with $E_{\Psi^\circ}$).
    \end{rmk}

    Before returning to symplectic geometry, our main technical result of this section is:
    \begin{prop}\label{prop: gauss}
        Let $L$ be a compact manifold (possibly with boundary). Assume there is a $g$-twisted $E_\Psi$-oriented ELS $\cF$ over $L$.

        Then there is a natural homotopy lift of $g$ along $F^\circ(N) \to \widetilde{U/O}^{or}(N)$.
    \end{prop}
\subsection{Proof of Proposition \ref{prop: gauss}}
    We prove Proposition \ref{prop: gauss} using a similar induction strategy to Section \ref{sec:ELS} over a handle decomposition of $L$, though the main step--in which we kill the obstruction--is different.
\subsubsection*{Set-up}
    Let $L_j$, $f_j$, $S$ and $H$ be as in Section \ref{sec: clas els}, and let $g_j = g|_{L_j}$. By the induction hypothesis we may assume Proposition \ref{prop: gauss} holds for $j=i$; as before, we assume for simplicity that $L_{i+1}$ is obtained from $L_i$ by attaching the lone $(i+1)$-handle $H$. Let $P, Q, R \in \Crit(f_{i+1})$ be as in Section \ref{sec: simp els}. Let $\cF$ be a $g_{i+1}$-Gauss-twisted $E_\Psi$-oriented ELS over $L_{i+1}$.

    The same arguments as in Section \ref{sec: simp els} show that $\cF$ is isomorphic to $\Cone(\cW: *[i] \to \cH)$ in $\Flow^{E_\Psi/g}_{/L_{i+1}}$, where $*[i]$ is sent to $P \in H$, and $\cH$ a $g_i$-Gauss-twisted $E_\Psi$-oriented ELS on $L_i$. The proof of Lemma \ref{lem:slide bott gluing} (in particular, the ``sliding'' step) does not work so well away from Morse theory; to rectify this, we want to use Theorem \ref{thm: WELS} to simplify $\cH$.

    By Remark \ref{rmk: g lift no twis}, the $g_i$-Gauss-twisted $E_\Psi$-orientation on $\cH$, along with the homotopy lift of $g_i$ to $F^\circ(N)$, together induce a (non-Gauss-twisted) $E_{\Psi^{\circ}}$-orientation on $\cH$; we write $\cH^\circ$ for this object. Similarly, the $g_{i+1}$-Gauss-twisted $E_\Psi$-orientation on $\cW$ induces a $g_{i+1}$-Gauss-twisted $E_{\Psi^\circ}$-orientation on $\cW$; we write $\cW^\circ$ for this object. Now by Theorem \ref{thm: WELS}, $\cH^\circ$ is isomorphic to $\cM(f_i)^\zeta$, for $\zeta$ some $\Psi^\circ$-MLS on $L_i$, and hence by the same argument as Lemma \ref{lem: 0-modification} we may assume they are equal. Now we may write the identity of $\cH^\circ$ as a fibre product as in (\ref{eq: reouhgeoudbg}), and so the same argument as in the proof of Lemma \ref{lem:slide bott gluing} shows we may additionally assume: 
    \begin{equation}
        \cW_{Pz}^\circ = \begin{cases}
            \pt 
            &
            \textrm{ if }z=Q
            \\
            ?
            &
            \textrm{ if }z=R
            \\
            \emptyset
            &
            \textrm{ otherwise}
        \end{cases}
    \end{equation}
    with the unspecified manifold $\cW_{PR}^\circ$ having dimension $(i-1)$. We may assume $ev(\cW^\circ_{PQ})$ is contained inside our choice of spider on $L$, and hence that, after collapsing the spider, this path becomes constant.

\subsubsection*{Identifying the obstruction}
    By the induction hypothesis, $g_i$ has a homotopy lift along $F^\circ(N) \to \widetilde{U/O}(N)$, induced by $\cH$. The obstruction to extending this to $L_{i+1}$ can be phrased in terms of the following homotopy lifting problem:
    \begin{equation}
        \xymatrix{
            L_i
            \ar[r]
            \ar[d]
            &
            L_{i+1}
            \ar[d]
            \ar@{-->}[dl]_{\exists ?}
            \\
            F^\circ(N)
            \ar[r]
            &
            \widetilde{U/O}^{or}(N)
        }
    \end{equation}
    $L_{i+1}$ is obtained by attaching $H \simeq D^{i+1}$ along $S \simeq S^i$ to $L_i$, so by excision we may replace the inclusion $L_i \to L_{i+1}$ with $S^i \to D^{i+1}$:
    \begin{equation}\label{eq: obst gaus}
        \xymatrix{
            S^i
            \ar[r]
            \ar[d]
            &
            D^{i+1}
            \ar[d]
            \ar@{-->}[dl]_{\exists ?}
            \\
            F^\circ(N)
            \ar[r]
            &
            \widetilde{U/O}^{or}(N)
        }
    \end{equation}
    \begin{rmk}
        The obstruction to such a homotopy lift lies in $\pi_{i-1} GL_1(R_C)$. 

    \end{rmk}
\subsubsection*{Killing the obstruction}
    The argument of Remark \ref{rmk: live over disc} shows we have a factorisation of the evaluation maps,  giving a homotopy commutative diagram:
    \begin{equation}
        \xymatrix{
            \cH_{QR}^\circ 
            \ar[r]
            \ar[d]
            &
            \cM(f_{i+1})_{QR} \cong S^{i-2}
            \ar[r]
            \ar[d]
            &
            \cP_{QR} L_{i+1}
            \ar[d]
            \ar[r]
            &
            \Omega F^\circ(N)
            \ar[d]
            \\
            \cW_{PR}^\circ
            \ar[r]
            &
            \cM(f_{i+1})_{PR} \cong D^{i-1}
            \ar[r]
            &
            \cP_{PR} L_{i+1}
            \ar[r]
            &
            \Omega \widetilde{U/O}^{or}(N)
        }
    \end{equation}  
    where leftmost two left vertical maps are given by concatenation with $\cW_{PQ}^\circ \cong \pt$, and the leftmost horizontal maps are both of degree 1.

    Choosing homotopy lifts along the equivalences $\Omega \widetilde{U/O}^{or} \simeq BSO$ and $\Omega F^\circ \simeq C^\circ$ from Section \ref{sec: tang J}, we obtain a homotopy commutative diagram:
    \begin{equation}\label{eq: gurhguohr0gohr}
        \xymatrix{
            \cH_{QR}^\circ
            \ar[r]
            \ar[d]
            &
            S^{i-2}
            \ar[d]
            \ar[r]
            &
            C^\circ(N)
            \ar[d]
            \\
            \cW_{PR}^\circ
            \ar[r]
            &
            D^{i-1}
            \ar[r]
            &
            BSO(N)
        }
    \end{equation}  
    Let $\gamma$ be the composition along the bottom and $\tilde\gamma_\partial$ the composition along the top, and use the same notation to denote the corresponding vector bundles pulled back from $BSO(N)$. 

    Because $\cW^\circ$ is $g_{i+1}$-Gauss-twisted $E_{\Psi^\circ}$-oriented and $E_{\Psi^\circ} \simeq C^\circ$, this provides the virtual vector bundles $T\cH^\circ_{QR} - \tilde\gamma_\partial$ and $T\cW^\circ_{PQ} - \gamma$ both with compatible $C^\circ$-orientations. 
    \begin{lem}\label{lem: gaus tech kill obst}
        Let $f: M \to M'$ be a smooth boundary-preserving map between compact oriented manifolds, and $V \to M'$ a vector bundle. Assume $f_\partial := f|_{\partial M}: \partial M \to \partial M'$ has degree 1 (hence $f$ also has degree 1). 
        
        Assume we are also given a $C^\circ$-orientation on the virtual vector bundle $f^*V-TM$.
        
        Then there is a natural $C^\circ$-orientation on $f^*V-TM'$.
    \end{lem}
    Heuristically, this says we can ``push forwards'' relative $R_C$-orientations along degree-1 maps.
    \begin{proof}[{Proof of Lemma \ref{lem: gaus tech kill obst}}]
        The Pontrjagin-Thom construction determines a map of Thom spectra:
        \begin{equation}
            f^!: M'^{V-TM'} \to M^{f^*V-TM}
        \end{equation}
        Since $f$ and $f_\partial$ are degree 1, $f^!$ is an equivalence. The $R_C$-orientation on $f^*V-TM$ is a map of spectra $M^{f^*V-TM} \to R_C$; its composition with $f^!$ defines an $R_C$-orientation on $V-TM'$.
    \end{proof}
    
    It follows from Lemma \ref{lem: gaus tech kill obst} (applied to the map $\cW^\circ_{PR} \to D^{i-1}$, as in (\ref{eq: gurhguohr0gohr}), and using the fact that $TD^{i-1}$ is trivial) that there is a homotopy lift $D^{i-1} \to C^\circ(N)$ fitting into (\ref{eq: gurhguohr0gohr}). This is not quite (\ref{eq: obst gaus}), however using the $\Omega$-$\Sigma$ adjunction (applied to $\Sigma S^{i-2} = S^{i-1}$ and $\Omega F^\circ \simeq C^\circ$), we obtain the desired homotopy lift (\ref{eq: obst gaus}); this completes the inductive step, and hence the proof of Proposition \ref{prop: gauss}.

    \begin{rmk}
        Lemma \ref{lem: gaus tech kill obst} leverages the fact that $C^\circ$-orientations on a (virtual) vector bundle have an alternative description in terms of the Thom spectrum (as opposed to lifts of a classifying map). This is the reason we must pass from $F$ to $F^\circ$ in the statement of Proposition \ref{prop: gauss}.  
    \end{rmk}
\subsection{Gauss-twisted ELSs from Floer theory}
    \begin{thm}\label{thm: gauss}

        Assume $X$ is polarised. Let $L, K \subseteq X$ be closed exact Lagrangians. Let $F$ be a commutative $\cI$-monoid, and $F \to \widetilde{U/O}^{or}$ some map. Let $\Psi$ be the polarised tangential pair of $F$, and $\Psi^\circ$ as in Section \ref{sec: tang J}. 
        
        Assume $L \cong K$ in the ordinary Fukaya category $\scrF(X; \bZ)$, and that $K$ admits a $\Psi$-orientation. 
        
        Then $L$ admits a $\Psi^\circ$-orientation.

    \end{thm}
    \begin{rmk}
        As in Remark \ref{rmk: circ vers J}, the conclusion of Theorem \ref{thm: gauss} is equivalent to the composition of the stable Gauss map of $L$ with the $J$ homomorphism for $R=R_C$, $L \to B^2GL_1(R)(N)$, being nullhomotopic (for $N \gg 0$), so it is equivalent to Theorem \ref{thm:main3}.
    \end{rmk}

    Assume we are in the setting of Theorem \ref{thm: gauss}. Consider the commutative diagram of commutative $\cI$-monoids; this is a special case of (\ref{eq: riwoghrdptgurdpotg}), and in this case it is a homotopy pullback square:
    \begin{equation}
        \xymatrix{
            BO \times \widetilde{U/O}
            \ar[r]_{Proj_1}
            \ar[d]_\oplus
            &
            \Phi = BO
            \ar[d]
            \\
            BO
            \ar[r]
            &
            BS_\pm U
        }
    \end{equation}
    $L$ naturally maps to the $N^{th}$ space in the top left. The composition with projection to the second factor gives the \emph{stable Gauss map} of $L$, $g_L: L \to \widetilde{U/O}(N')$, for some $N' \gg 0$.
    \begin{prop}\label{prop: gaus ind comm}
        The flow category over $L$, $\cM^{LK}$, is naturally $g_L$-Gauss-twisted $E_\Psi$-oriented.
    \end{prop}
    \begin{proof}
        Fix $x,y \in \cM^{LK}$. We explain the construction for a single moduli space $\cM^{LK}_{xy}$; making it compatible with breaking can be done with identical arguments to those of \cite[Section 5.4]{PS3}.

        By considering the linearisation of the Floer equation for $\cM^{LK}_{xy}$ and the given orientations for $K$ and $X$, we see that the analogue of \cite[(5.7)]{PS3} in this setting (which follows from identical arguments to \cite[Section 6]{PS2}) is a map of Thom spaces:
        \begin{equation}
            \rho: (\cM^{LK}_{xy}, I^{\cM^{LK}}_{xy})(\nu_\cM) \to \bU^{Fg_L}_{\bE}(\nu_\cM, N)
        \end{equation}
        where $\bE = \bE_{\cM^{LK}_{xy}}$ is the appropriate tuple of puncture data, and $\bU^{Fg_L}_{\bE}(\ldots)$ denotes a variation on the space of abstract discs $\bU^{\Psi^F}_{\bE}(\ldots)$, defined as follows. Let $\Psi^{unor} = (BO \to BS_\pm U)$. A 0-simplex in the underlying simplicial set $\bU^{\Psi^F}_{\bE}(\ldots)$ consists of:
        \begin{itemize}
            \item A loop $\gamma \in \Omega L$.
            \item An abstract disc $\bD \in \bU^{\Psi^{unor}}_{\bE}(\ldots)$ with 1 input and 1 output puncture (hence its underlying domain is biholomorphic to a (possibly nodal) strip).

            Let $D$ be the underlying domain, $V \to D$ the complex part, and $W \to \partial D$ the real part (so $W \subseteq V|_{\partial D}$ is a totally real subbundle). Let $\partial_\pm D$ be the two boundary components of $D$, and $W_\pm$ the restriction of $W$ to each of these components.
            \item A polarisation of $V$: $V \cong V' \otimes \bC$, for some real vector bundle $V' \to D$, such that:
            \begin{enumerate}
                \item $W_+ = V'|_{\partial D}$.
                \item $W_-$ is the totally real subbundle of $V' \otimes \bC$ whose Gauss map (after stabilising) is $\Omega g_L \circ \gamma$ (so it is equal to $V'|_{\partial D}$ if $\Omega g_L \circ \gamma$ is trivial).
            \end{enumerate}
            \item A map from $\partial_+ D \to F$
        \end{itemize}
        $k$-simplices for $k \geq 1$ are defined as families of the above data parametrised by $\Delta^k$, and the index bundle $\bV$ over the realisation is defined as in Section \ref{sec:tp and ad}.
        
        Note on bases, $\rho$ satisfies that if we write $\rho(u) = (\gamma, \bD, \ldots)$, $\gamma = ev(u)$.

        Given such a 0-simplex $(\gamma, \bD, \ldots)$ as above, up to contractible choice we may write it as a direct sum of two discs $\bD \simeq \bD_1 \oplus \bD_2$, both with domain $D$: $\bD_1$ has trivial complex part $\bC^{N}$, real part over $\partial_+ D$ trivial $\bR^{N}$ and real part over $\partial_- D$ given by $\Omega g_L \circ \gamma: \partial_- D \to \Omega L \to \Omega \widetilde{U/O}^{or}(N) \to U/O(N)$, and $\bD_2$ has the same complex part $V' \otimes \bC$ as $\bD$, the same real part $V'$ over $\partial_+ D$ and real part over $\partial_-D$ given by the polarisation $V' \subseteq V' \otimes \bC \cong V$ (this is similar to the ``separating into two pieces'' step in \cite[Section 5.4]{PS3}). 
        
        Gluing a cap on to the input of $\bD_1$ gives a 0-simplex $\bD_1'$ in $\bU^{or}(\ldots)$. $\bD_2$ lives in $\bU^{\Psi^F}_{\bE}(\ldots)$, so we may glue a cap to its input to obtain a 0-simplex in $\bU^{\Psi}$. The assignment $\bD \mapsto (\bD'_2, \bD'_1)$ defines the map (\ref{eq: gireghortbvhprifv}); the map on vector bundles may be defined similarly.
    \end{proof}
    \begin{proof}[{Proof of Theorem \ref{thm: gauss}}]
        Proposition \ref{prop: gaus ind comm} implies $\cM^{LK}$ is $g_L$-Gauss-twisted $E_\Psi$-oriented. It follows from Remark \ref{rmk: gaus zero twis} and the assumption $L \cong K$ in in $\scrF(X; \bZ)$ that $\cM^{LK}$ is a $g_L$-Gauss-twisted $E_\Psi$-oriented ELS. Proposition \ref{prop: gauss} then implies the theorem.
    \end{proof}

\section{Normal invariants}\label{sec:normal}

To fit classical notation, in this section we write $G$ for the infinite loop space $GL_1(\bS)$.

Throughout this section, the term `polarisation' refers to a (homotopy class of) \emph{stable} polarisations, where a stable polarisation of a complex vector bundle $W \to Q$ is a pair $(\kappa,s)$ with $\kappa \in KO(Q)$ and $s: \kappa \otimes \bC \to W$ a stable isomorphism. 

Recall that the set of normal invariants of a Poincar\'e complex $Q$ is in general a torsor for the group $[Q,G/O]$; when $Q$ is a compact smooth manifold, the set of normal invariants has a preferred base-point (coming from the normal map which is the identity on $Q$ covered by the identity on $TQ$), and the bijection to $[Q,G/O]$ becomes canonical.

If $Q$ is a manifold and $L \subset T^*Q$ is a closed exact Lagrangian submanifold, the fact that projection $\mathrm{proj}: L \to Q$ is a homotopy equivalence implies that there is a nullhomotopy of the composite
\[
\xymatrix{
Q \ar[rr]^{TQ-\mathrm{proj}_*TL} && BO \ar[r] & BG
}
\]
so there is a canonical lift to the homotopy fibre. This element of $[Q,G/O]$ represents the normal invariant of $L$. 

\begin{rmk}\label{rmk:unstable normal}
    The set $[Q,G/O]$ of normal invariants forms an abelian group, since $G/O$ is an infinite loop space. When $Q=S^n$ there is a map $\Theta_n/bP_{n+1} \to \pi_n G/O$ (from the set of homotopy spheres modulo those which framed bound) which is either onto or has index two image,  depending on whether the dimension is a `Kervaire invariant one' dimension. 
\end{rmk}

\begin{thm}\label{thm:normal through eta}
    If $L\subset T^*Q$ is a nearby Lagrangian submanifold, its normal invariant factors through the map $\eta_*: B(G/O) \to G/O$.
\end{thm}

In particular, the normal invariant is two-torsion.

\subsection{Topology generalities}\label{sec: topo gene}

For a closed $d$-manifold $Q$, the following two pointed sets are naturally in bijection:
\begin{enumerate}
    \item (Normal invariants) $[Q,G/O]$; think of an element as a pair $(V,h)$ comprising a virtual bundle $V: Q \to BO$ and a nullhomotopy of $BJ \circ h: Q \stackrel{V}{\longrightarrow} BO \to BG$. The group structure comes from the infinite loop space structure on $G/O$. The identity element is the constant map to the basepoint.
    \item  (Normal maps, tangential flavour) $\cN\cM(Q)$; an element is a tuple $(M,a,f,V,\xi)$ with $M$ a closed $d$-manifold, $f: M \to Q$ a degree one map, $V \to Q$ a vector bundle such that $\xi: TM \oplus \bR^a \to f^*V$ is a map of bundles. We take these up to equivalence under stabilisation and relative framed bordism. 
    The identity is given by $(Q,0,\id,TQ,\id_{TQ})$. If $f: L\to Q$ is a simple homotopy equivalence, and $f^{-1}: Q \to L$ a choice of homotopy inverse, the normal invariant of $L$ is $(L,0,f, (f^{-1})^*TL, \id_{TL})$.
\end{enumerate}
\begin{rmk}
    Though we will not use it, $\cN\cM(Q)$ can be endowed with the structure of a group by taking transverse intersections over $Q$, and this bijection is in fact one of groups.
\end{rmk}
    
\begin{rmk}
Recall that an $\bS$-orientation of a vector bundle $V: Q \to BO$ can be viewed equivalently as (i) a nullhomotopy of the composition of the classifying map with $BO\to BG$,  or (ii) a fibrewise homotopy trivialisation of the (stable) sphere bundle of $V$. An equivalent description of ``normal invariants"  is given by the group $\cG/\cO(Q) (\cong [Q,G/O])$ whose elements are pairs $(V,t)$ where $V\to Q$ is a vector bundle and $t: Sph(V) \to S^{k-1}\times Q$ is an $\bS$-orientation, again taken up to suitable equivalence. In this model, the identity element is $(\bR^k,\id)$ for any $k\geq 0$ and the  group operation is Whitney sum. 
\end{rmk}

\begin{rmk}
    $[Q, G/O]$ only depends on the homotopy type of $Q$, and $\cN\cM(Q)$ on the tangential homotopy type.
\end{rmk}

The natural bijection
\begin{equation}\label{eqn:relate normal invariants}
\tau: \cN\cM(Q) \to [Q,G/O]
\end{equation}
views $(M,f,a,V,\xi)$  as defining a relatively framed bordism class in $\Omega_d^{fr}(Q,V) = \pi_d^{st}(\Thom(-V))$.
    Atiyah duality says that if $V \to Q$ is a stable bundle over a compact manifold $Q$ then the Spanier-Whitehead dual $\bD(\cdot)$ of the Thom space is given by
    \[
    \bD(\Thom V) = \Thom(-TQ-V)
    \]
    So under Atiyah duality 
    \[
    \pi_d^{st}(\Thom(-V)) = \pi^0_{st}(\Thom(V-TQ))
    \]
so this is a map 
\[
s: \Thom(V-TQ) \to \bS.
\]
Since $f$ has degree one this defines a trivialisation of the stable spherical fibration of $V-TQ$, and $\tau(M,f,a,V,\xi) = (V-TQ, s)$.

\begin{lem}\label{lem:action through eta}
    There is a natural group homomorphism $[Q,B(G/O)] \to [Q,G/O]$.
\end{lem}

\begin{proof}
    Compose with the map $\eta: B(G/O) \to G/O$.
\end{proof}

There are different polarisations $V$ of $T^*Q$, with stable isomorphisms $T(T^*Q) \sim V\otimes \bC$.

\begin{lem}\label{lem:act on polarisation}
    If $V$ is a polarisation of $T^*Q$ then a choice of map $\chi \in [Q,B(G/O)]$  defines a new polarisation $V^{\chi}$.
\end{lem}

\begin{proof}
    Via the canonical map $B(G/O) \to U/O$ we get an element of $[Q,U/O]$. This then acts essentially by direct sum.
\end{proof}

\subsection{Nearby Lagrangians}\label{sec: near}
\subsubsection*{Identifying the normal invariant}
Let $\iota: L \subset T^*Q$ be a nearby Lagrangian. Then $\iota$ is a homotopy equivalence, as is $f = \pi \circ \iota: L \to Q$.

\begin{defn}
    Let $\phi_{TL}$ denote the polarisation of $T^*Q$ which writes $T(T^*Q) = V \otimes \bC$ for $V = (\iota^*)^{-1}TL \to T^*Q$. 
\end{defn}

Let $\Theta = BO \times hofib(U/O \to B^2G) = BO \times B(G/O)$. Let $\Phi = BO$. Define the map $\Theta \to \Phi$ to be the first projection $\pi_1$ and $\Theta \to BO$ to be given by $\pi_1 + (\mathrm{Re}\circ \pi_2)$, where $\pi_2$ is the map to $U/O$. Let $\Phi \to BU$ be complexification,  cf. Example \ref{qi ex: tang str pol}. 
There is then a spectral Fukaya category $\scrF(T^*Q,\Psi)$ as constructed in \cite{PS3}, with $\phi_{TL}: T^*Q \to BO$ the $\Phi$-structure on $T^*Q$.

An element of $\Omega_d^{E_{\Psi},\cO\cC}(T^*Q,\phi_{TL})$ is by definition a closed $d$-manifold $M \to T^*Q$ with some further tangential data, up to bordism (cf. Definition \ref{defn: oc-bordism}). There is a subset $\Omega_d^{E_{\Psi},\cO\cC}(T^*Q,\phi_{TL})^{|\mathrm{deg}|=1}$ of elements represented by maps which have degree $\pm 1$  after composition with $T^*Q\to Q$.

\begin{lem}
    $L$ admits a canonical $\Theta$-orientation with respect to this $\Phi$-structure on $T^*Q$. Therefore there is an associated class $[L] \in \Omega_d^{E_{\Psi},\cO\cC}(T^*Q,\phi_{TL})$, and this class has degree $\pm 1$.
\end{lem}

\begin{proof}
    We have the diagram
    \[
    \xymatrix{
     F \ar[r] \ar@{=}[d] & \Theta \ar[r] \ar@{=}[d]& \Phi\ar@{=}[d] \\
    B(G/O) \ar[r] \ar[d] & BO \times B(G/O) \ar[r]_-{\pi_1} \ar[d]_{\pi_1 + \mathrm{Re}\circ \pi_2} & BO \ar[d]_{\otimes \bC} \\
    U/O \ar[r]_{\mathrm{Re}} & BO \ar[r]_{\otimes \bC}  & BU 
    }
    \]
    If we set $L \to BO \times B(G/O)$ to be $(TL,\mathrm{const})$ then the diagram
    \[
    \xymatrix{
    L \ar[r]_{\iota} \ar[d] & T^*Q \ar[d]^{(\iota^*)^{-1}TL} \\
    BO \times B(G/O) \ar[r]_-{\pi_1} & BO
    }
    \]
    commutes. Degree $\pm 1$ is clear.
\end{proof}

Let $\Omega^{\bS}_*$ denote the bordism of manifolds with stably trivial spherical fibration. This has an underlying commutative ring spectrum $\Omega^\bS$, constructed as the Thom spectrum of the map $G/O \to BO$.

\begin{lem}
    For the tangential pair $\Psi$, the bordism group $\Omega^{E_{\Psi},\cO\cC}_*(T^*Q,\phi_{TL})$ admits a natural map $q$ to $ \Omega^{\bS}_*(\Thom(-TL))$. 
\end{lem}

\begin{proof}
    The fact that $\Omega^{E_{\Psi}}_* = \Omega^{\bS}_*$ (the bordism theory which governs manifolds with a trivialisation of the stable spherical fibration associated to the tangent bundle) is contained in \cite[Example 1.9]{PS2}. The rest is then a special case of Remark \ref{recap: OC bordism maps here}(2), given that the polarising bundle on $T^*Q$ is (the pullback of) $TL$,  together with homotopy invariance $\Omega^{\bS}_*(T^*Q,TL)\simeq \Omega^{\bS}_*(Q,TL)$.
\end{proof}

\begin{lem}\label{lem:unit splits off spherical}
    The natural unit map $\bS \to \Omega^{\bS}$ splits in the homotopy category of commutative ring spectra. 
\end{lem}
\begin{proof}
    This follows from \cite[Corollary 3.17]{Antolin-Camarena-Barthel}.
\end{proof}
\begin{rmk}
    We may describe the splitting on the level of homology theories explicitly. Fix some target cell complex $Y$ and let $f: M \to Y$ be an $\bS$-oriented $i$-manifold. Represent the $\bS$-orientation by a fibrewise equivalence $S^k \times M \to S(TM + \bR^{i-k})$. Generically perturb and take the preimage of $0$; the resulting manifold is naturally framed, and maps to $Y$.  This defines a natural transformation of cohomology theories $\Omega^{\bS} \to \bS$. To see that this splits the unit map, one can take the fibrewise spherical equivalences to be linear on the fibres.     See \cite[Section 2.5]{P} for a variant of this construction.
\end{rmk}

\begin{lem}\label{lem:split framed compatibly with eta}
    The resulting map $p: \Omega^{\bS} \to \bS$ entwines the action of $\eta$, in the sense that the following diagram commutes:
    \begin{equation}
        \xymatrix{
            BGL_1(\Omega^\bS) 
            \ar[r]_p
            \ar[d]_{\eta_*}
            &
            BGL_1(\bS)
            \ar[d]_{\eta_*}
            \\
            GL_1(\Omega^\bS)
            \ar[r]_p
            &
            GL_1(\bS)
        }
    \end{equation}
\end{lem}

\begin{proof}
    This holds for any map of commutative ring spectra.
\end{proof}

\begin{lem}\label{lem: 1235}
    There is a map 
    \begin{equation} \label{eqn:OC to normal maps}
    \chi: \Omega_d^{E_{\Psi},\cO\cC}(T^*Q,\phi_{TL})^{|\mathrm{deg}|=1}  \to \pi^0_{st}(\Thom(TL-TQ))^{|\mathrm{deg}|=1} \to [Q,G/O].
    \end{equation}
    
    This map is equivariant with respect to actions of $R_\Psi^0(Q)^\times$ on the first term and $[Q, G]$ on the second, with respect to the map of groups $R_\Psi^0(Q)^\times \to \bS^0(Q)^\times = \pi^0_{st}(Q)^\times = [Q, G]$. 
\end{lem}

\begin{proof}
The splitting $p$ induces a natural map
\begin{equation} \label{eqn:spherical framed to usual framed}
 \Omega^{\bS}_*(\Thom (-TL)) \to \Omega^{fr}_*(\Thom (TL-TQ)) = \pi_*^{st}\Thom(TL-TQ).
\end{equation}
from spherically framed bordism to framed bordism. Combine this with Atiyah duality, which is an isomorphism of modules over $\bS^0(Q)$. The equivariance follows from Lemma \ref{lem:split framed compatibly with eta}. 
\end{proof}

\begin{lem}\label{lem:get normal}
    Under the map $\chi$ of \eqref{eqn:OC to normal maps} the class $[L]$ maps to the normal invariant of $L$.
\end{lem}

\begin{proof}
    We break the map into steps:
    \begin{itemize}
        \item The $\cO\cC$-class for $L$ comes from $\Omega^{E_{\Psi};\cO\cC}(L,(\phi_{TL})|_L)$ under pushforward $\iota: L \to T^*Q$. In general the data of an $\cO\cC$-bordism class for $(X,\phi)$ is (i) a map $M \to X$ (ii) a map of Thom spaces $(M, TM) \to \Theta \times E$ and (iii) homotopy between the maps $M \to X \to \Phi$ and $M \to \Base(\Theta \times E) \to \Phi$. 
        
        Note that in defining the class $[L] \in \Omega^{E_{\Psi};\cO\cC}(L,(\phi_{TL})|_L)$ we can take the trivial $E$-orientation. In our case $E \sim \Omega F = \Omega B(G/O) = G/O$. So the map $L \to \Base(\Theta\times E) = (BO \times B(G/O)) \times G/O$ is actually trivial on the latter two factors, the $\phi_{TL}$-orientation on $T^*Q$ is classifying $TL$, and there is no data in the homotopy (iii) as both ways around the square are just $TL: L \to BO$ and the homotopy can be $\id$.
        
        \item The map to $\bS$-oriented bordism relative to $V = (\iota^{-1})^*(TL)$ is asking for a spherical trivialisation of the composite of $L \stackrel{TL-\iota^*V}{\longrightarrow} BO$ with $BO\to BG$. But the bundle being classified is $TL-TL$ which is canonically spherically oriented and framed. 
        \item We  have the framed bordism class in $\Omega^{fr}(\Thom(-TL)) = \Omega^{fr}(T^*Q;(\iota^{-1})^*TL) = \Omega^{fr}(Q,(\iota^{-1})^*TL)$ 
       represented by $L \to Q$ and the canonical framing of $TL$ relative to $\iota^*(\iota^{-1})^*TL$. But this is the normal invariant of $L$ according to its definition in $\cN\cM(Q)$ even before applying  $\tau$.
    \end{itemize}
\end{proof}

\begin{rmk}\label{rmk:woolly}
    We say that a class in $\Omega_d^{E_{\Psi},\cO\cC}(T^*Q,\phi_{TL})$ differs from $[L]$ by `something in the image of $\eta$' if it lies in the  orbit of $[L]$ under the map $[Q, BG] \xrightarrow{\eta} [Q, G] \to R^0(Q)^\times$ (and the natural module action of the latter).  The equivariance statement in Lemma \ref{lem: 1235} implies that any such class is taken by $\chi$ to something which differs from the normal invariant of $L$ by an element of $[Q,G/O]$ which factors through $\eta$ in the usual sense.
\end{rmk}

\subsection*{Zero-section}
Up to this point the only brane we have used in $\scrF(T^*Q;\Psi)$ is the brane associated to $L$.  

\begin{lem}
    The zero-section $Q$ defines an object of $\scrF(T^*Q;\Psi)$.
\end{lem}

\begin{proof}
We require a $\Theta$-orientation of $Q$, so a map $Q \to BO \times B(G/O)$. The first co-ordinate $Q \to BO$ is fixed by the need for commutativity of the square 
    \[
    \xymatrix{
    Q \ar[r] \ar[d] & T^*Q \ar[d]_{TL} \\
    BO \times B(G/O) \ar[r]_-{\pi_1} & BO 
    }
    \]
    whilst the need for a non-trivial function in the second co-ordinate arises because the map
    \[
    \xymatrix{
        Q \ar[r] & BO\times B(G/O) \ar[rr]_-{\pi_1 + \mathrm{Re}\circ \pi_2} && BO
    }
    \]
    must classify $TQ$ to be a $\Theta$-orientation.

    We use Theorem \ref{thm:main3}. The stable Gauss map of $L$ relative $\phi_{TL}$ is null; $L$ and $Q$ are quasi-isomorphic over $\bZ$; Theorem \ref{thm:main3} shows the stable Gauss map of $Q$ (defined relative to the polarisation $\phi_{TL}$) becomes null on composition with $B^2J$; and we view the nullhomotopy of $Q \to \widetilde{U/O} \to B^2GL_1(\bS)$ as a lift to the homotopy fibre, denoted $\Gamma_L: Q \to B(G/O)$. We then define the $\Theta$-orientation on $Q$ by taking the map
    \[
    (TL, \Gamma_L): Q \to BO \times B(G/O)
    \]
    and this satisfies our requirements.
\end{proof}

    Each polarisation $V$ on $T^*Q$, with $T(T^*Q) = V\otimes \bC$, defines a Fukaya category $\scrF(T^*Q;\phi_V)$ where the ambient $\Phi$-orientation comes from $V$.

\begin{lem}\label{lem: specific action on polarisation}
    The polarisation $\phi_{TL}$ is the $\Gamma_L$-twist of the `natural' polarisation $\phi_{TQ}$.
\end{lem}

\begin{proof}
 $\Gamma_L$ is a lift of the (canonically nullhomotopic) map $L \to U/O \to B^2G$ to the homotopy fibre $B(G/O)$. It acts on polarisations via $B(G/O) \to U/O$, so the result follows from the diagram
\[
\xymatrix{
B(G/O) \ar[r] \ar[d]_{\eta} & \widetilde{U/O} \ar[r] \ar[d]_{\eta} & B^2G \ar[d]_{\eta} \\ G/O \ar[r] & BO \ar[r] & BG
}
\]
and the fact that, viewing $\widetilde{U/O} = B^2 O$, the central map is stably classifying $TL$, which follows from Proposition \ref{prop: uo eta} (in the special case in which $F=B(G/O)$). 

\end{proof}

\begin{lem}\label{lem:easy for TQ choice}
    For the polarisation $\phi_{TQ}$, the zero-section $Q$ defines a brane in $\scrF(T^*Q;\phi_{TQ})$ with $\Theta$-orientation $(TQ,0): Q \to BO \times B(G/O)$. The class of this brane in $\Omega^{E_{\Psi};\cO\cC}(T^*Q,TQ)$ maps to the constant map at the base-point under $\chi_{TQ}: \Omega^{E_{\Psi};\cO\cC}(T^*Q,TQ) \to [Q,G/O]$. 
\end{lem}

\begin{proof}
This is straightforward from the definitions.
\end{proof}

\begin{lem}\label{lem:compare OC groups}
    Suppose we have two polarisations $\{V,V^{\lambda}\}$ of $T^*Q$ with $V^{\lambda}$ a twist of $V$ associated to a map $\lambda: Q \to B(G/O)$. Then there is an induced isomorphism
    \[
    \Omega^{E_{\Psi};\cO\cC}(T^*Q,V)  \longrightarrow \Omega^{E_{\Psi};\cO\cC}(T^*Q,V^\lambda).
    \]
\end{lem}

\begin{proof}
    Take an element in the $V$-oriented $\cO\cC$-bordism group, represented by $M \to T^*Q$ with a homotopy commutative diagram
    \[
    \xymatrix{
    M \ar[rr] \ar[d] && T^*Q \ar[d]_V \\ \Base(\Theta\times E) = BO \times B(G/O) \times G/O \ar[rr]_-{\pi_1} && BO
    }
    \]
    By definition $V^{\lambda} = V + \mathrm{Re}(\lambda)$ (which is still a polarisation since $\mathrm{Re}(-) \otimes \bC$ is stably trivial). Note that $\lambda$ induces a map (still denoted) $\lambda: M \to B(G/O)$. 
    Now act on the left downwards arrow by $(+\mathrm{Re}(\lambda), +\lambda, 0)$. One can similarly act on the entire homotopy $M \times [0,1] \to BO$ witnessing the square's commutativity by $\oplus \lambda$. Since $V$ is a twist of $V^{\lambda}$ by $\lambda^{-1}$, using the group structure on $[Q,B(G/O)]$, the resulting map is an isomorphism.
\end{proof}

A similar construction can be used to turn any $\Theta$-oriented brane with respect to $V$ into one with respect to $V^\lambda$. This gives a map from the set of objects of $\scrF(X; \phi_V)$ to those of $\scrF(X; \phi_{V^\lambda})$, we expect (but do not carry out) that this can be upgraded to a functor.

\begin{lem}\label{lem: change of polarisation relates OC module actions}
The isomorphism of Lemma \ref{lem:compare OC groups} entwines the actions of $R^0_{\Psi}(Q)^{\times}$.    
\end{lem}

\begin{proof}
    This follows from \cite[Lemma 5.17]{PS3}. Indeed, a class $\alpha \in R^0_{\Psi}(Q)$ acts on a bordism class $M\to T^*Q$ via replacing $M \to T^*Q$  by the preimage $M' \to T^*Q$ under zero of a map $M \times S^j \to \Thom(E_{\Psi}(j))$ for $j\gg 0$, and noting that this preimage inherits the required tangential data via projection $M' \to M$.  This commutes with twisting the (tangential data for the) input bordism class by $\lambda$ as in Lemma \ref{lem:compare OC groups}. 
\end{proof}

\begin{lem}
     $\chi_{TL}([Q])$ belongs to the image of $\eta$.
\end{lem}

\begin{proof}
For the polarisation $\phi_{TQ}$ the associated class in $\Omega^{E_{\Psi},\cO\cC}(T^*Q;TQ)$ is taken by $\chi_{TQ}$ to the constant map at the base-point.  By Lemma \ref{lem: change of polarisation relates OC module actions}, and the equivariance in Lemma \ref{lem: 1235} for each of the $\phi_{TL}$ and $\phi_{TQ}$ polarisations, it then follows that the image of $[Q]$ under $\chi_{TL}$ also belongs to the image of $\eta$ in the sense of Remark \ref{rmk:woolly}. 
\end{proof}

\begin{proof}[Proof of Theorem \ref{thm:normal through eta}.]
Theorem \ref{thm:main} says there is a local system $\xi: Q \to BGL_1(\bS)$ such that $L \sim (Q,\xi)$ in $\scrF(T^*Q;\Psi)$. Therefore $\chi_{TL}([L]) = \chi_{TL}([Q,\xi])$ and, by Theorem \ref{thm: gour main theorem}, $\chi_{TL}([Q,\xi])$ equals $\chi_{TL}([Q])$ capped with a class in $R_\Psi^0(Q)^\times$ that is in the image of $\eta$. But Lemmas \ref{lem:easy for TQ choice} and \ref{lem:compare OC groups} say that $\chi_{TL}([Q])$ is equal to $\chi_{TQ}([Q])$ capped with another class in the image of $\eta$ and so, since $\chi_{TQ}([Q])$ is the basepoint, $\chi_{TL}([L])$ is in the image of $\eta$. Combining with Lemma \ref{lem:get normal}, we find that the normal invariant of $L$ is in the image of $\eta$. 
\end{proof}

The arguments of Section \ref{sec: near} are not completely restricted to cotangent bundles. 
\begin{prop}\label{prop: fin}
    Assume that $L, K \subseteq X$ are closed exact Lagrangians in a Liouville domain, such that:
    \begin{enumerate}
        \item $X$ is polarised, by some Spin vector bundle $V_L \to X$. 
        \item $L, K$ define objects in $\scrF(X; \bZ)$, and are isomorphic to each other in this category.
        \item With respect to $V_L$, the stable Gauss map of $L$ is nullhomotopic.
        \item There is a retraction $r: X \to L$ of $L$, such that the composition $r|_K: K \hookrightarrow X \xrightarrow{r} L$ is a homotopy equivalence.
        \item There is a stable equivalence of vector bundles $V_L \cong r^*TL$.
    \end{enumerate}
    Then the normal invariant of the map $L \to K$ factors through $\eta: B(G/O) \to G/O$.
\end{prop}
In the proof of Theorem \ref{thm: norm}, we used two different polarisations on $T^*Q$; these were easy to produce since $L \simeq Q \simeq T^*Q$ were all homotopy equivalent. Much of the work here is to prove there exists a second well-behaved polarisation besides $V_L$; to do this, we apply Theorems \ref{thm:main3} and \ref{thm:main}.
\begin{proof}[Sketch proof]
    We claim the stable Gauss map of $K$, $K\to U/O$, lifts to $B^2Spin$. It suffices to show that its restriction to the $3$-skeleton $K_3$ of $K$ is nullhomotopic. By Theorem \ref{thm:main3}, the composition $K \to U/O \to B^2G$ is nullhomotopic, and hence lifts to the homotopy fibre $B(G/O)$. The low-degree homotopy groups $\pi_{\leq 4}B(G/O)$ vanish, so the map $K_3 \to B(G/O)$ is nullhomotopic and hence so is $K_3 \to U/O$. 
    
    It follows that $L$ and $K$ define objects in a spectral Fukaya category $(BO \times B^2Spin, BO)$ as in Example \ref{qi ex: tang str pol}; from Remark \ref{recap: OC bordism maps here}(2) we see the open-closed map for this Fukaya category maps to the Spin bordism groups $\Omega^{Spin}_*(X)$ of $X$ relative to $V_L$. 
    
    By Theorem \ref{thm:main}, there is some spectral local system $\xi_L$ on $L$ such that $(L, \xi_L)$ and $K$ are isomorphic in this spectral Fukaya category. A similar moduli space argument to \cite[Proposition 7.27]{PS} identifies endomorphisms in this spectral Fukaya category with the Spin bordism groups of the Lagrangian relative to the tangent bundle; one may incorporate monodromy local systems as in Lemma \ref{lem:22}. Since $(L, \xi_L)$ and $K$ are isomorphic, their endomorphisms in this spectral Fukaya category are isomorphic. The same moduli space argument as \cite[Lemma 7.15]{PS3} (cf. also \cite[Proposition 7.27 \& Remark 1.15]{PS}) shows this isomorphism lies over the bordism groups of $X$. To summarise, there is a commutative diagram:
    \begin{equation}
        \xymatrix{
            \Omega^{Spin}_*(X, V_L)
            &
            \\
            \Omega^{Spin}_*(L, TL)
            \ar[u]
            &
            \Omega^{Spin}_*(K, TK)
            \ar[ul]
            \ar[l]_\cong
        }
    \end{equation}
    with the horizontal arrow coming from the isomorphism in the spectral Fukaya category.
    
    Applying the Conner-Floyd isomorphism $\Omega^{Spin}_*(\cdot) \otimes_{\Omega^{Spin}_*} KO_* \cong KO_*(\cdot)$ \cite{Hopkins-Hovey} along with Atiyah duality and the Thom isomorphism theorem (for real $K$-theory relative to a Spin vector bundle) implies that there is a commutative triangle:
    \begin{equation}\label{erpghru0gvhousbdfv}
        \xymatrix{
            KO^*(X)
            \ar[d]
            \ar[dr]
            &
            \\
            KO^*(L)
            \ar[r]_\cong
            &
            KO^*(K)
        }
    \end{equation}
    where the two maps down are given by pullback along the inclusion.

    Consider the stable Gauss map of $K$, $K \to U/O$. We view this as a class $g_K \in KO^1(K)$. By (\ref{erpghru0gvhousbdfv}) and the fact $X$ retracts to $L$, this is the restriction of some class $\tilde g_K \in KO^1(X)$, corresponding to some other polarisation $V_K$ of $X$ such that the stable Gauss map of $K$ with respect to $V_K$ is nullhomotopic. By Theorem \ref{thm:main3}, the restrictions of the polarisations $V_K$ and $V_L$ to $L$ differ by a map $L \to B(G/O)$.

    From this, we may apply exactly the same arguments as those earlier in this section, with $K$ in place of the zero-section $Q$ and $V_K$ in place of $V_Q$, to obtain the desired conclusion, by comparing $r_*[L]$ with $r_*[K]$ in the appropriate bordism group. The final condition of the statement ensures $r$ extends to a map of Thom spaces, so we may split bordism of $L$ relative to $TL$ from bordism of $X$ relative to $V_L$.
\end{proof}
\begin{rmk}\label{rmk: riewgnepifgnrip}
    Consider a plumbing $X = T^*Q_1 \#_C T^*Q_2$ of simply connected cotangent bundles along a submanifold $C$ of codimension at least $3$.  Assume $Q_1$ is Spin and $Q_2$ is stably framed. Suppose $L \subset X $ is a simply-connected compact Lagrangian which is known to be quasi-isomorphic to $Q_1$. Then the projection $L \to Q_1$, coming from projecting $X$ to the compact core and collapsing the other factors, is a homotopy equivalence, since it has degree one and induces an isomorphism in homology over any field. 
    
    Proposition \ref{prop: fin} then applies to constrain the normal invariant of $L \to K$. 
\end{rmk}
\begin{rmk}
    The paper \cite{AS:plumbings} shows that, under the hypotheses of Remark \ref{rmk: riewgnepifgnrip}, any $L$ is a twisted complex on the core components of the plumbing. It seems reasonable to believe that the only bordism classes represented by exact Lagrangians are ones in the span of those represented by the core components.
\end{rmk}

\bibliographystyle{myamsalpha}
\bibliography{Refs.bib}{} 

@article {Abbaspour-Laudenbach,
    AUTHOR = {Abbaspour, Hossein and Laudenbach, Fran\c cois},
     TITLE = {Morse complexes and multiplicative structures},
   JOURNAL = {Math. Z.},
  FJOURNAL = {Mathematische Zeitschrift},
    VOLUME = {300},
      YEAR = {2022},
    NUMBER = {3},
     PAGES = {2637--2678},
      ISSN = {0025-5874,1432-1823},
   MRCLASS = {57R19},
  MRNUMBER = {4381216},
       DOI = {10.1007/s00209-021-02863-y},
       URL = {https://doi.org/10.1007/s00209-021-02863-y},
}

@incollection {Lashof,
    AUTHOR = {Lashof, Richard},
     TITLE = {Stable {$G$}-smoothing},
 BOOKTITLE = {Algebraic topology, {W}aterloo, 1978 ({P}roc. {C}onf., {U}niv.
              {W}aterloo, {W}aterloo, {O}nt., 1978)},
    SERIES = {Lecture Notes in Math.},
    VOLUME = {741},
     PAGES = {283--306},
 PUBLISHER = {Springer, Berlin},
      YEAR = {1979},
      ISBN = {3-540-09545-4},
   MRCLASS = {57R10 (57S10)},
  MRNUMBER = {557173},
MRREVIEWER = {Michael\ C.\ Bix},
}

@incollection {Hsiang-Shaneson,
    AUTHOR = {Hsiang, Wu-chung and Shaneson, Julius L.},
     TITLE = {Fake tori},
 BOOKTITLE = {Topology of {M}anifolds ({P}roc. {I}nst., {U}niv. of
              {G}eorgia, {A}thens, {G}a., 1969)},
     PAGES = {18--51},
 PUBLISHER = {Markham Publishing Co., Chicago, IL},
      YEAR = {1970},
   MRCLASS = {57.01},
  MRNUMBER = {281211},
MRREVIEWER = {B.\ G.\ Casler},
}

@article {Hsiang-Sharpe,
    AUTHOR = {Hsiang, W. C. and Sharpe, R. W.},
     TITLE = {Parametrized surgery and isotopy},
   JOURNAL = {Pacific J. Math.},
  FJOURNAL = {Pacific Journal of Mathematics},
    VOLUME = {67},
      YEAR = {1976},
    NUMBER = {2},
     PAGES = {401--459},
      ISSN = {0030-8730,1945-5844},
   MRCLASS = {57E05 (57D60 58D05)},
  MRNUMBER = {494165},
MRREVIEWER = {D.\ Burghelea},
       URL = {http://projecteuclid.org/euclid.pjm/1102817501},
}

@misc{Rezchikov,
    title={Integral {A}rnol'd conjecture},
author={Semon Rezchikov},
year={2022},
eprint={2209.11165},
archivePrefix={arXiv},
primaryClass={math.SG},
url={https://arxiv.org/abs/2209.11165},
note={Preprint, available at arXiv:2209.11165}
}

@misc{AMS,
    title={Complex cobordism, {H}amiltonian loops and global {K}uranishi charts},
    author={Mohammed Abouzaid and Mark McLean and Ivan Smith},
    year={2021},
    eprint={2110.14320},
    archivePrefix={arXiv},
      primaryClass={math.SG},
      url={https://arxiv.org/abs/2110.14320}, 
    note={Preprint, available at arXiv:2110.14320}
}

@misc{CP,
      title={On the parametrised {W}hitehead torsion of families of nearby Lagrangian submanifolds}, 
      author={Sylvain Courte and Noah Porcelli},
      year={2025},
      eprint={2506.06110},
      archivePrefix={arXiv},
      primaryClass={math.SG},
      url={https://arxiv.org/abs/2506.06110}, 
    note={Preprint, available at arXiv:2506.06110}
}

@misc{PS3,
      title={Open-closed maps and spectral local systems}, 
      author={Noah Porcelli and Ivan Smith},
      year={2025},
      archivePrefix={arXiv},
      primaryClass={math.SG},
    key={PS:oc},
    note={Preprint, available at arXiv:2509.21483},
}

@incollection {Hatcher:concordance,
    AUTHOR = {Hatcher, A. E.},
     TITLE = {Concordance spaces, higher simple-homotopy theory, and
              applications},
 BOOKTITLE = {Algebraic and geometric topology ({P}roc. {S}ympos. {P}ure
              {M}ath., {S}tanford {U}niv., {S}tanford, {C}alif., 1976),
              {P}art 1},
    SERIES = {Proc. Sympos. Pure Math.},
    VOLUME = {XXXII},
     PAGES = {3--21},
 PUBLISHER = {Amer. Math. Soc., Providence, RI},
      YEAR = {1978},
      ISBN = {0-8218-1432-X},
   MRCLASS = {57R52},
  MRNUMBER = {520490},
MRREVIEWER = {Gerald\ A.\ Anderson},
}

@book {Luck-Macko,
    AUTHOR = {L\"uck, Wolfgang and Macko, Tibor},
     TITLE = {Surgery theory---foundations},
    SERIES = {Grundlehren der mathematischen Wissenschaften [Fundamental
              Principles of Mathematical Sciences]},
    VOLUME = {362},
      NOTE = {With contributions by Diarmuid Crowley},
 PUBLISHER = {Springer, Cham},
      YEAR = {[2024] \copyright 2024},
     PAGES = {xv+956},
      ISBN = {978-3-031-56333-1; 978-3-031-56334-8},
   MRCLASS = {57-02 (57R65)},
  MRNUMBER = {4789553},
       DOI = {10.1007/978-3-031-56334-8},
       URL = {https://doi.org/10.1007/978-3-031-56334-8},
}

@misc{Cip,
      title={Revisiting the Cohen-Jones-Segal construction in Morse-Bott theory}, 
      author={Ciprian Mircea Bonciocat},
      year={2025},
      eprint={2409.11278},
      archivePrefix={arXiv},
      primaryClass={math.AT},
      url={https://arxiv.org/abs/2409.11278}, 
    note={Preprint, available at arXiv:2409.11278}
}

@unpublished{AB,
    author = {Abouzaid, Mohammed and Blumberg, Andrew},
    title = {Arnol'd conjecture and {M}orava {$K$}-theory} ,
    note = {Preprint, available at arXiv:2103.01507},
}

@unpublished{PS,
    author = {Porcelli, Noah and Smith, Ivan },
    title = {Bordism of flow modules and exact {L}agrangians} ,
    year = {2024},
    note = {Preprint, available at arXiv:2401.11766},
}

@article {Abouzaid:homotopy,
    AUTHOR = {Abouzaid, Mohammed},
     TITLE = {Nearby {L}agrangians with vanishing {M}aslov class are
              homotopy equivalent},
   JOURNAL = {Invent. Math.},
  FJOURNAL = {Inventiones Mathematicae},
    VOLUME = {189},
      YEAR = {2012},
    NUMBER = {2},
     PAGES = {251--313},
      ISSN = {0020-9910,1432-1297},
   MRCLASS = {53D12 (53D37 53D40)},
  MRNUMBER = {2947545},
MRREVIEWER = {Michael\ J.\ Usher},
       DOI = {10.1007/s00222-011-0365-0},
       URL = {https://doi.org/10.1007/s00222-011-0365-0},
}

@article {AS:plumbings,
    AUTHOR = {Abouzaid, Mohammed and Smith, Ivan},
     TITLE = {Exact {L}agrangians in plumbings},
   JOURNAL = {Geom. Funct. Anal.},
  FJOURNAL = {Geometric and Functional Analysis},
    VOLUME = {22},
      YEAR = {2012},
    NUMBER = {4},
     PAGES = {785--831},
      ISSN = {1016-443X,1420-8970},
   MRCLASS = {53D37 (53D12)},
  MRNUMBER = {2984118},
MRREVIEWER = {Rafael\ Santamar\'{\i}a},
       DOI = {10.1007/s00039-012-0162-y},
       URL = {https://doi.org/10.1007/s00039-012-0162-y},
}

@article {AK,
    AUTHOR = {Abouzaid, Mohammed and Kragh, Thomas},
     TITLE = {On the immersion classes of nearby {L}agrangians},
   JOURNAL = {J. Topol.},
  FJOURNAL = {Journal of Topology},
    VOLUME = {9},
      YEAR = {2016},
    NUMBER = {1},
     PAGES = {232--244},
      ISSN = {1753-8416,1753-8424},
   MRCLASS = {53D40 (53D12 55P42)},
  MRNUMBER = {3465849},
MRREVIEWER = {Jian\ Xun\ Hu},
       DOI = {10.1112/jtopol/jtv041},
       URL = {https://doi.org/10.1112/jtopol/jtv041},
}

@unpublished{ACGK,
    author = {Abouzaid, Mohammed and Courte, Sylvain and Guillermou, Stephane and Kragh, Thomas},
    title = {Twisted generating functions and the nearby {L}agrangian conjecture},
    note = {Preprint, available at arXiv:2011.13178},
}

@unpublished{BDHO,
    author = {Barraud, Jean-Francois and Damian, Mihai and Humili\'ere, Vincent and Oancea, Alexandru},
    title = {Morse homology with {DG} coefficients},
    note = {Preprint, available at arXiv:2308.06104},
}

@unpublished{BDHO2,
    author = {Barraud, Jean-Francois and Damian, Mihai and Humili\'ere, Vincent and Oancea, Alexandru},
    title = {Floer homology with {DG} coefficients: applications to cotangent bundles},
    note = {Preprint, available at arXiv:2404.07953}
}

@article {Barraud-Cornea,
    AUTHOR = {Barraud, Jean-Fran\c{c}ois and Cornea, Octav},
     TITLE = {Lagrangian intersections and the {S}erre spectral sequence},
   JOURNAL = {Ann. of Math. (2)},
  FJOURNAL = {Annals of Mathematics. Second Series},
    VOLUME = {166},
      YEAR = {2007},
    NUMBER = {3},
     PAGES = {657--722},
      ISSN = {0003-486X,1939-8980},
   MRCLASS = {53D40 (55T10)},
  MRNUMBER = {2373371},
MRREVIEWER = {Timothy\ Perutz},
       DOI = {10.4007/annals.2007.166.657},
       URL = {https://doi.org/10.4007/annals.2007.166.657},
}

@article {Bousfield,
    AUTHOR = {Bousfield, A. K.},
     TITLE = {Localization and periodicity in unstable homotopy theory},
   JOURNAL = {J. Amer. Math. Soc.},
  FJOURNAL = {Journal of the American Mathematical Society},
    VOLUME = {7},
      YEAR = {1994},
    NUMBER = {4},
     PAGES = {831--873},
      ISSN = {0894-0347,1088-6834},
   MRCLASS = {55P60 (55N15 55N20 55Q05 55U10)},
  MRNUMBER = {1257059},
MRREVIEWER = {N.\ J.\ Kuhn},
       DOI = {10.2307/2152734},
       URL = {https://doi.org/10.2307/2152734},
}

@article {Jack:Tiny,
    AUTHOR = {Smith, Jack},
     TITLE = {Hamiltonian isotopies of relatively exact {L}agrangians are
              orientation-preserving},
   JOURNAL = {Adv. Geom.},
  FJOURNAL = {Advances in Geometry},
    VOLUME = {24},
      YEAR = {2024},
    NUMBER = {4},
     PAGES = {505--506},
      ISSN = {1615-715X,1615-7168},
   MRCLASS = {53D12 (53D05)},
  MRNUMBER = {4813077},
MRREVIEWER = {Shangrong\ Deng},
       DOI = {10.1515/advgeom-2024-0021},
       URL = {https://doi.org/10.1515/advgeom-2024-0021},
}

@article {Abouzaid:based_loops,
    AUTHOR = {Abouzaid, Mohammed},
     TITLE = {On the wrapped {F}ukaya category and based loops},
   JOURNAL = {J. Symplectic Geom.},
  FJOURNAL = {The Journal of Symplectic Geometry},
    VOLUME = {10},
      YEAR = {2012},
    NUMBER = {1},
     PAGES = {27--79},
      ISSN = {1527-5256,1540-2347},
   MRCLASS = {53D37 (53D12)},
  MRNUMBER = {2904032},
MRREVIEWER = {Janko\ Latschev},
       DOI = {10.4310/jsg.2012.v10.n1.a3},
       URL = {https://doi.org/10.4310/jsg.2012.v10.n1.a3},
}

@article {Abouzaid-Kragh:Simple,
    AUTHOR = {Abouzaid, Mohammed and Kragh, Thomas},
     TITLE = {Simple homotopy equivalence of nearby {L}agrangians},
   JOURNAL = {Acta Math.},
  FJOURNAL = {Acta Mathematica},
    VOLUME = {220},
      YEAR = {2018},
    NUMBER = {2},
     PAGES = {207--237},
      ISSN = {0001-5962,1871-2509},
   MRCLASS = {53D12 (53D40 57R17 57R58)},
  MRNUMBER = {3849284},
MRREVIEWER = {Darko\ Milinkovi\'c},
       DOI = {10.4310/ACTA.2018.v220.n2.a1},
       URL = {https://doi.org/10.4310/ACTA.2018.v220.n2.a1},
}

@incollection {Hovey,
    AUTHOR = {Hovey, Mark},
     TITLE = {Bousfield localization functors and {H}opkins' chromatic
              splitting conjecture},
 BOOKTITLE = {The \v Cech centennial ({B}oston, {MA}, 1993)},
    SERIES = {Contemp. Math.},
    VOLUME = {181},
     PAGES = {225--250},
 PUBLISHER = {Amer. Math. Soc., Providence, RI},
      YEAR = {1995},
      ISBN = {0-8218-0296-8},
   MRCLASS = {55P42 (55N20 55N22 55P60)},
  MRNUMBER = {1320994},
       DOI = {10.1090/conm/181/02036},
       URL = {https://doi.org/10.1090/conm/181/02036},
}

@article {Kuhn,
    AUTHOR = {Kuhn, Nicholas J.},
     TITLE = {Localization of {A}ndr\'e-{Q}uillen-{G}oodwillie towers, and
              the periodic homology of infinite loopspaces},
   JOURNAL = {Adv. Math.},
  FJOURNAL = {Advances in Mathematics},
    VOLUME = {201},
      YEAR = {2006},
    NUMBER = {2},
     PAGES = {318--378},
      ISSN = {0001-8708,1090-2082},
   MRCLASS = {55N22 (18G55 55N20 55P43 55P47 55P60)},
  MRNUMBER = {2211532},
MRREVIEWER = {J.\ P. C. Greenlees},
       DOI = {10.1016/j.aim.2005.02.005},
       URL = {https://doi.org/10.1016/j.aim.2005.02.005},
}

@article {Harvey-Lawson,
    AUTHOR = {Harvey, F. Reese and Lawson, Jr., H. Blaine},
     TITLE = {Finite volume flows and {M}orse theory},
   JOURNAL = {Ann. of Math. (2)},
  FJOURNAL = {Annals of Mathematics. Second Series},
    VOLUME = {153},
      YEAR = {2001},
    NUMBER = {1},
     PAGES = {1--25},
      ISSN = {0003-486X,1939-8980},
   MRCLASS = {58E05 (53C65 57R70 58A25)},
  MRNUMBER = {1826410},
MRREVIEWER = {David\ E.\ Hurtubise},
       DOI = {10.2307/2661371},
       URL = {https://doi.org/10.2307/2661371},
}

@incollection {CJS,
    AUTHOR = {Cohen, R. L. and Jones, J. D. S. and Segal, G. B.},
     TITLE = {Floer's infinite-dimensional {M}orse theory and homotopy
              theory},
 BOOKTITLE = {The {F}loer memorial volume},
    SERIES = {Progr. Math.},
    VOLUME = {133},
     PAGES = {297--325},
 PUBLISHER = {Birkh\"{a}user, Basel},
      YEAR = {1995},
      ISBN = {3-7643-5044-X},
   MRCLASS = {55N35 (57R70 58E05)},
  MRNUMBER = {1362832},
MRREVIEWER = {Dave\ Auckly},
}

@book {Sullivan:Galois,
    AUTHOR = {Sullivan, Dennis P.},
     TITLE = {Geometric topology: localization, periodicity and {G}alois
              symmetry},
    SERIES = {$K$-Monographs in Mathematics},
    VOLUME = {8},
      NOTE = {The 1970 MIT notes,
              Edited and with a preface by Andrew Ranicki},
 PUBLISHER = {Springer, Dordrecht},
      YEAR = {2005},
     PAGES = {xiv+283},
      ISBN = {978-1-4020-3511-1; 1-4020-3511-X},
   MRCLASS = {55-02 (55P60 55R25 55R40 57N16)},
  MRNUMBER = {2162361},
MRREVIEWER = {John\ McCleary},
}

@misc{Blakey,
      title={Floer homotopy theory and degenerate {L}agrangian intersections}, 
      author={Kenneth Blakey},
      year={2024},
      eprint={2410.11478},
      archivePrefix={arXiv},
      primaryClass={math.SG},
      url={https://arxiv.org/abs/2410.11478}, 
    note={Preprint, available at arXiv:2410.11478},
}

@book {Rud98,
	AUTHOR = {Rudyak, Yuli B.},
	TITLE = {On {T}hom spectra, orientability, and cobordism},
	SERIES = {Springer Monographs in Mathematics},
	NOTE = {With a foreword by Haynes Miller},
	PUBLISHER = {Springer-Verlag, Berlin},
	YEAR = {1998},
	PAGES = {xii+587},
	ISBN = {3-540-62043-5},
	MRCLASS = {55-02 (55N22 55P42 57-02)},
	MRNUMBER = {1627486},
	MRREVIEWER = {Donald M. Davis},}

@article {Sagave-Schlichtkrull:Diagram,
    AUTHOR = {Sagave, Steffen and Schlichtkrull, Christian},
     TITLE = {Diagram spaces and symmetric spectra},
   JOURNAL = {Adv. Math.},
  FJOURNAL = {Advances in Mathematics},
    VOLUME = {231},
      YEAR = {2012},
    NUMBER = {3-4},
     PAGES = {2116--2193},
      ISSN = {0001-8708,1090-2082},
   MRCLASS = {55P43 (55P48)},
  MRNUMBER = {2964635},
MRREVIEWER = {Rekha\ Santhanam},
       DOI = {10.1016/j.aim.2012.07.013},
       URL = {https://doi.org/10.1016/j.aim.2012.07.013},
}

@article {Hu-Lalonge-Leclercq,
    AUTHOR = {Hu, Shengda and Lalonde, Fran\c cois and Leclercq, R\'emi},
     TITLE = {Homological {L}agrangian monodromy},
   JOURNAL = {Geom. Topol.},
  FJOURNAL = {Geometry \& Topology},
    VOLUME = {15},
      YEAR = {2011},
    NUMBER = {3},
     PAGES = {1617--1650},
      ISSN = {1465-3060,1364-0380},
   MRCLASS = {53D35 (53D40)},
  MRNUMBER = {2851073},
MRREVIEWER = {Mark\ Alan\ Branson},
       DOI = {10.2140/gt.2011.15.1617},
       URL = {https://doi.org/10.2140/gt.2011.15.1617},
}

@misc{ADP,
      title={Spectral equivalence of nearby {L}agrangians}, 
      author={Johan Asplund and Yash Deshmukh and Alex Pieloch},
      year={2024},
      eprint={2411.08841},
      archivePrefix={arXiv},
      primaryClass={math.SG},
      url={https://arxiv.org/abs/2411.08841}, 
    note={Preprint, available at arXiv:2411.08841},
}

@article {FO3:smoothness,
    AUTHOR = {Fukaya, Kenji and Oh, Yong-Geun and Ohta, Hiroshi and Ono,
              Kaoru},
     TITLE = {Exponential decay estimates and smoothness of the moduli space
              of pseudoholomorphic curves},
   JOURNAL = {Mem. Amer. Math. Soc.},
  FJOURNAL = {Memoirs of the American Mathematical Society},
    VOLUME = {299},
      YEAR = {2024},
    NUMBER = {1500},
     PAGES = {v+139},
      ISSN = {0065-9266,1947-6221},
      ISBN = {978-1-4704-7106-4; 978-1-4704-7878-0},
   MRCLASS = {53D35 (35B40 53D40 53D45 58D27)},
  MRNUMBER = {4790260},
       DOI = {10.1090/memo/1500},
       URL = {https://doi.org/10.1090/memo/1500},
    note = {Available at arXiv:1603.07026},
}

@misc{Bai-Xu,
    title={Arnol'd conjecture over the integers},
author={Bai, Shaoyun and Xu,Guangbo},
year={2022},
eprint={2209.08599},
archivePrefix={arXiv},
primaryClass={math.SG},
url={https://arxiv.org/abs/2209.08599}, 
}

@misc{Burklund:Big,
      title={How Big are the Stable Homotopy Groups of Spheres?}, 
      author={Robert Burklund and Andrew Senger},
      year={2022},
      eprint={2203.00670},
      archivePrefix={arXiv},
      primaryClass={math.AT},
      url={https://arxiv.org/abs/2203.00670}, 
note={Preprint, available at arXiv:2203.00670},
}

@misc{Telescope,
      title={$K$-theoretic counterexamples to Ravenel's telescope conjecture}, 
      author={Robert Burklund and Jeremy Hahn and Ishan Levy and Tomer M. Schlank},
      year={2023},
      eprint={2310.17459},
      archivePrefix={arXiv},
      primaryClass={math.AT},
      url={https://arxiv.org/abs/2310.17459}, 
note={Preprint, available at arXiv:2310.17459},
}

@misc{Xin,
title={Micolocal sheaf categories and the {$J$}-homomorphism},
author={Jin,Xin},
year={2020},
eprint={2004.14270},
primaryClass={math.SG},
url={https://arxiv.org/abs/2004.14270}, 
}

@article {Mak-Wu,
    AUTHOR = {Mak, Cheuk Yu and Wu, Weiwei},
     TITLE = {Dehn twist exact sequences through {L}agrangian cobordism},
   JOURNAL = {Compos. Math.},
  FJOURNAL = {Compositio Mathematica},
    VOLUME = {154},
      YEAR = {2018},
    NUMBER = {12},
     PAGES = {2485--2533},
      ISSN = {0010-437X,1570-5846},
   MRCLASS = {53D40 (53D12 53D37 57R90)},
  MRNUMBER = {3873526},
MRREVIEWER = {Georgios\ Dimitroglou Rizell},
       DOI = {10.1112/s0010437x18007479},
       URL = {https://doi.org/10.1112/s0010437x18007479},
}

@article{Antolin-Camarena-Barthel,
   title={A simple universal property of {T}hom ring spectra},
   volume={12},
   ISSN={1753-8424},
   url={http://dx.doi.org/10.1112/topo.12084},
   DOI={10.1112/topo.12084},
   number={1},
   journal={Journal of Topology},
   publisher={Wiley},
   author={Antolín‐Camarena, Omar and Barthel, Tobias},
   year={2018},
   month=oct, pages={56–78} }

@misc{Nadler-Shende,
      title={Sheaf quantization in {W}einstein symplectic manifolds}, 
      author={David Nadler and Vivek Shende},
      year={2022},
      eprint={2007.10154},
      archivePrefix={arXiv},
      primaryClass={math.SG},
      url={https://arxiv.org/abs/2007.10154}, 
}

@incollection {Joyce,
    AUTHOR = {Joyce, Dominic},
     TITLE = {On manifolds with corners},
 BOOKTITLE = {Advances in geometric analysis},
    SERIES = {Adv. Lect. Math. (ALM)},
    VOLUME = {21},
     PAGES = {225--258},
 PUBLISHER = {Int. Press, Somerville, MA},
      YEAR = {2012},
      ISBN = {978-1-57146-248-0},
   MRCLASS = {58A05},
  MRNUMBER = {3077259},
MRREVIEWER = {Hirokazu\ Nishimura},
}

@article {Kragh,
    AUTHOR = {Kragh, Thomas},
     TITLE = {Parametrized ring-spectra and the nearby {L}agrangian
              conjecture},
      NOTE = {With an appendix by Mohammed Abouzaid},
   JOURNAL = {Geom. Topol.},
  FJOURNAL = {Geometry \& Topology},
    VOLUME = {17},
      YEAR = {2013},
    NUMBER = {2},
     PAGES = {639--731},
      ISSN = {1465-3060,1364-0380},
   MRCLASS = {53D12 (53D40 55P42 55T10)},
  MRNUMBER = {3070514},
MRREVIEWER = {Michael\ J.\ Usher},
       DOI = {10.2140/gt.2013.17.639},
       URL = {https://doi.org/10.2140/gt.2013.17.639},
}

@unpublished{Large,
    author = {Large, Tim},
    title = {Spectral {F}ukaya categories for {L}iouville manifolds},
    note = {MIT PhD Thesis, 2021, available at http://dspace.mit.edu/handle/1721.1/139233}
}

@misc{AAGCKarxiv,
      title={Normal invariant of nearby Lagrangians via twisted derivative}, 
      author={Mohammed Abouzaid and Daniel Alvarez-Gavela and Sylvain Courte and Thomas Kragh},
      year={2025},
      eprint={2505.12515},
      archivePrefix={arXiv},
      primaryClass={math.SG},
      url={https://arxiv.org/abs/2505.12515}, 
    note={Preprint, available at arXiv:2505.12515},
}

@article {Schlichtkrull,
    AUTHOR = {Schlichtkrull, Christian},
     TITLE = {Units of ring spectra and their traces in algebraic
              {$K$}-theory},
   JOURNAL = {Geom. Topol.},
  FJOURNAL = {Geometry and Topology},
    VOLUME = {8},
      YEAR = {2004},
     PAGES = {645--673},
      ISSN = {1465-3060,1364-0380},
   MRCLASS = {19D55 (19D10 55P43 55P48)},
  MRNUMBER = {2057776},
MRREVIEWER = {Mark\ Hovey},
       DOI = {10.2140/gt.2004.8.645},
       URL = {https://doi.org/10.2140/gt.2004.8.645},
}

@article {Hopkins-Hovey,
    AUTHOR = {Hopkins, Michael J. and Hovey, Mark A.},
     TITLE = {Spin cobordism determines real {$K$}-theory},
   JOURNAL = {Math. Z.},
  FJOURNAL = {Mathematische Zeitschrift},
    VOLUME = {210},
      YEAR = {1992},
    NUMBER = {2},
     PAGES = {181--196},
      ISSN = {0025-5874,1432-1823},
   MRCLASS = {55N22 (19L41)},
  MRNUMBER = {1166518},
MRREVIEWER = {P.\ S.\ Landweber},
       DOI = {10.1007/BF02571790},
       URL = {https://doi.org/10.1007/BF02571790},
}

@book {Seidel:book,
    AUTHOR = {Seidel, Paul},
     TITLE = {Fukaya categories and {P}icard-{L}efschetz theory},
    SERIES = {Zurich Lectures in Advanced Mathematics},
 PUBLISHER = {European Mathematical Society (EMS), Z\"{u}rich},
      YEAR = {2008},
     PAGES = {viii+326},
      ISBN = {978-3-03719-063-0},
   MRCLASS = {53D40 (16E45 32Q65 53D12)},
  MRNUMBER = {2441780},
MRREVIEWER = {Timothy\ Perutz},
       DOI = {10.4171/063},
       URL = {https://doi.org/10.4171/063},
}

@misc{HarrisR,
      title={Projective twists in $A_{\infty}$-categories}, 
      author={Harris, Richard},
      year={2011},
      eprint={1111.0538},
      archivePrefix={arXiv},
      primaryClass={math.SG},
      url={https://arxiv.org/abs/1111.0538}, 
note={Preprint, available at arXiv:1111.0538},
}

@article {Huybrechts-Thomas,
    AUTHOR = {Huybrechts, Daniel and Thomas, Richard},
     TITLE = {{$\Bbb P$}-objects and autoequivalences of derived categories},
   JOURNAL = {Math. Res. Lett.},
  FJOURNAL = {Mathematical Research Letters},
    VOLUME = {13},
      YEAR = {2006},
    NUMBER = {1},
     PAGES = {87--98},
      ISSN = {1073-2780},
   MRCLASS = {14F05 (18E30)},
  MRNUMBER = {2200048},
MRREVIEWER = {Bal\'azs\ Szendr\H oi},
       DOI = {10.4310/MRL.2006.v13.n1.a7},
       URL = {https://doi.org/10.4310/MRL.2006.v13.n1.a7},
}

@article {Kasilingam-CP,
    AUTHOR = {Kasilingam, Ramesh},
     TITLE = {Classification of smooth structures on a homotopy complex
              projective space},
   JOURNAL = {Proc. Indian Acad. Sci. Math. Sci.},
  FJOURNAL = {Indian Academy of Sciences. Proceedings. Mathematical
              Sciences},
    VOLUME = {126},
      YEAR = {2016},
    NUMBER = {2},
     PAGES = {277--281},
      ISSN = {0253-4142,0973-7685},
   MRCLASS = {57R55 (57R50)},
  MRNUMBER = {3489166},
MRREVIEWER = {Anar\ Akhmedov},
       DOI = {10.1007/s12044-016-0269-4},
       URL = {https://doi.org/10.1007/s12044-016-0269-4},
}

@article{Basu-Kasilingam:cpn,
   title={Inertia groups of high-dimensional complex projective spaces},
   volume={18},
   ISSN={1472-2747},
   url={http://dx.doi.org/10.2140/agt.2018.18.387},
   DOI={10.2140/agt.2018.18.387},
   number={1},
   journal={Algebraic and Geometric Topology},
   publisher={Mathematical Sciences Publishers},
   author={Basu, Samik and Kasilingam, Ramesh},
   year={2018},
   month=jan, pages={387–408} }

@article {Basu-Kasilingam,
    AUTHOR = {Basu, Samik and Kasilingam, Ramesh},
     TITLE = {Inertia groups and smooth structures on quaternionic
              projective spaces},
   JOURNAL = {Forum Math.},
  FJOURNAL = {Forum Mathematicum},
    VOLUME = {34},
      YEAR = {2022},
    NUMBER = {2},
     PAGES = {369--383},
      ISSN = {0933-7741,1435-5337},
   MRCLASS = {57R60 (55P25 55P42 57R55)},
  MRNUMBER = {4388343},
       DOI = {10.1515/forum-2020-0125},
       URL = {https://doi.org/10.1515/forum-2020-0125},
}

@book {Kirby-Siebenmann,
    AUTHOR = {Kirby, Robion C. and Siebenmann, Laurence C.},
     TITLE = {Foundational essays on topological manifolds, smoothings, and
              triangulations},
    SERIES = {Annals of Mathematics Studies},
    VOLUME = {No. 88},
      NOTE = {With notes by John Milnor and Michael Atiyah},
 PUBLISHER = {Princeton University Press, Princeton, NJ; University of Tokyo
              Press, Tokyo},
      YEAR = {1977},
     PAGES = {vii+355},
   MRCLASS = {57-02 (57AXX)},
  MRNUMBER = {645390},
MRREVIEWER = {Ronald\ J.\ Stern},
}

@article {Torricelli,
    AUTHOR = {Torricelli, Brunella Charlotte},
     TITLE = {Projective twists and the {H}opf correspondence},
   JOURNAL = {Algebr. Geom. Topol.},
  FJOURNAL = {Algebraic \& Geometric Topology},
    VOLUME = {24},
      YEAR = {2024},
    NUMBER = {8},
     PAGES = {4139--4200},
      ISSN = {1472-2747,1472-2739},
   MRCLASS = {53D35 (53D37)},
  MRNUMBER = {4843728},
       DOI = {10.2140/agt.2024.24.4139},
       URL = {https://doi.org/10.2140/agt.2024.24.4139},
}

@article {Hsiang,
    AUTHOR = {Hsiang, Wu-chung},
     TITLE = {A note on free differentiable actions of {$S\sp{1}$} and
              {$S\sp{3}$} on homotopy spheres},
   JOURNAL = {Ann. of Math. (2)},
  FJOURNAL = {Annals of Mathematics. Second Series},
    VOLUME = {83},
      YEAR = {1966},
     PAGES = {266--272},
      ISSN = {0003-486X},
   MRCLASS = {57.47},
  MRNUMBER = {192506},
MRREVIEWER = {F.\ Hirzebruch},
       DOI = {10.2307/1970431},
       URL = {https://doi.org/10.2307/1970431},
}

@article {Crowley,
    AUTHOR = {Crowley, Diarmuid},
     TITLE = {The smooth structure set of {$S^p\times S^q$}},
   JOURNAL = {Geom. Dedicata},
  FJOURNAL = {Geometriae Dedicata},
    VOLUME = {148},
      YEAR = {2010},
     PAGES = {15--33},
      ISSN = {0046-5755,1572-9168},
   MRCLASS = {57R55 (57R65)},
  MRNUMBER = {2721618},
MRREVIEWER = {Yasuhiko\ Kitada},
       DOI = {10.1007/s10711-010-9513-8},
       URL = {https://doi.org/10.1007/s10711-010-9513-8},
}

@article {Adams,
    AUTHOR = {Adams, J. F.},
     TITLE = {On the groups {$J(X)$}. {IV}},
   JOURNAL = {Topology},
  FJOURNAL = {Topology. An International Journal of Mathematics},
    VOLUME = {5},
      YEAR = {1966},
     PAGES = {21--71},
      ISSN = {0040-9383},
   MRCLASS = {55.40},
  MRNUMBER = {198470},
MRREVIEWER = {E.\ Dyer},
       DOI = {10.1016/0040-9383(66)90004-8},
       URL = {https://doi.org/10.1016/0040-9383(66)90004-8},
}

@misc{Kasilingam-cpn58,
      title={Smooth structures on $\mathbb{C}P^{m}$ for $5\leq m\leq 8$}, 
      author={Ramesh Kasilingam},
      year={2023},
      eprint={1701.07592},
      archivePrefix={arXiv},
      primaryClass={math.GT},
      url={https://arxiv.org/abs/1701.07592}, 
    note={Preprint, available at arXiv:1701.07592},
}

@article {Kasilingam-inertia,
    AUTHOR = {Ramesh, Kasilingam},
     TITLE = {Inertia groups and smooth structures of {$(n-1)$}-connected
              {$2n$}-manifolds},
   JOURNAL = {Osaka J. Math.},
  FJOURNAL = {Osaka Journal of Mathematics},
    VOLUME = {53},
      YEAR = {2016},
    NUMBER = {2},
     PAGES = {309--319},
      ISSN = {0030-6126},
   MRCLASS = {57R55 (57R50 57R60 57R65)},
  MRNUMBER = {3492800},
MRREVIEWER = {Sergey\ M.\ Finashin},
       URL = {http://projecteuclid.org/euclid.ojm/1461781789},
}

@article {Kervaire-Milnor,
    AUTHOR = {Kervaire, Michel A. and Milnor, John W.},
     TITLE = {Groups of homotopy spheres. {I}},
   JOURNAL = {Ann. of Math. (2)},
  FJOURNAL = {Annals of Mathematics. Second Series},
    VOLUME = {77},
      YEAR = {1963},
     PAGES = {504--537},
      ISSN = {0003-486X},
   MRCLASS = {57.10},
  MRNUMBER = {148075},
MRREVIEWER = {J.\ F.\ Adams},
       DOI = {10.2307/1970128},
       URL = {https://doi.org/10.2307/1970128},
}

@inproceedings {Wehrheim,
    AUTHOR = {Wehrheim, Katrin},
     TITLE = {Smooth structures on {M}orse trajectory spaces, featuring
              finite ends and associative gluing},
 BOOKTITLE = {Proceedings of the {F}reedman {F}est},
    SERIES = {Geom. Topol. Monogr.},
    VOLUME = {18},
     PAGES = {369--450},
 PUBLISHER = {Geom. Topol. Publ., Coventry},
      YEAR = {2012},
   MRCLASS = {57R55 (37D15)},
  MRNUMBER = {3084244},
MRREVIEWER = {Marco\ Mazzucchelli},
       DOI = {10.2140/gtm.2012.18.369},
       URL = {https://doi.org/10.2140/gtm.2012.18.369},
}

@article {P,
    AUTHOR = {Porcelli, Noah W.},
     TITLE = {Families of relatively exact {L}agrangians, free loop spaces
              and generalised homology},
   JOURNAL = {Selecta Math. (N.S.)},
  FJOURNAL = {Selecta Mathematica. New Series},
    VOLUME = {30},
      YEAR = {2024},
    NUMBER = {2},
     PAGES = {Paper No. 21},
      ISSN = {1022-1824,1420-9020},
   MRCLASS = {53D12 (55N20)},
  MRNUMBER = {4700425},
       DOI = {10.1007/s00029-023-00910-6},
       URL = {https://doi.org/10.1007/s00029-023-00910-6},
}

@misc{AB2,
      title={Foundation of {F}loer homotopy theory {I}: {F}low categories}, 
      author={Mohammed Abouzaid and Andrew J. Blumberg},
      year={2024},
      eprint={2404.03193},
      archivePrefix={arXiv},
      primaryClass={math.SG},
      url={https://arxiv.org/abs/2404.03193}, 
    note={Preprint, available at arXiv:2404.03193},
}

@article {IWX,
    AUTHOR = {Isaksen, Daniel C. and Wang, Guozhen and Xu, Zhouli},
     TITLE = {Stable homotopy groups of spheres: from dimension 0 to 90},
   JOURNAL = {Publ. Math. Inst. Hautes \'Etudes Sci.},
  FJOURNAL = {Publications Math\'ematiques. Institut de Hautes \'Etudes
              Scientifiques},
    VOLUME = {137},
      YEAR = {2023},
     PAGES = {107--243},
      ISSN = {0073-8301,1618-1913},
   MRCLASS = {55T15 (14F42)},
  MRNUMBER = {4588596},
MRREVIEWER = {Guchuan\ Li},
       DOI = {10.1007/s10240-023-00139-1},
       URL = {https://doi.org/10.1007/s10240-023-00139-1},
}

@book {Audin,
    AUTHOR = {Audin, Mich\`ele},
     TITLE = {Cobordismes d'immersions lagrangiennes et legendriennes},
    SERIES = {Travaux en Cours [Works in Progress]},
    VOLUME = {20},
 PUBLISHER = {Hermann, Paris},
      YEAR = {1987},
     PAGES = {xvi+203},
      ISBN = {2-7056-6056-9},
   MRCLASS = {58F05 (57R90 58G15)},
  MRNUMBER = {903652},
MRREVIEWER = {Alan\ Weinstein},
}

@misc{R-F,
      title={On cobordism groups of {L}agrangian immersions}, 
      author={Dominique Rathel-Fournier},
      year={2024},
      eprint={2409.14651},
      archivePrefix={arXiv},
      primaryClass={math.SG},
      url={
         https://arxiv.org/abs/2409.14651}, 
    note={Preprint, available at arXiv:2409.14651},
}

@misc{PS2,
      title={Spectral {F}loer theory and tangential structures}, 
      author={Noah Porcelli and Ivan Smith},
      year={2024},
      eprint={2411.03257},
      archivePrefix={arXiv},
      primaryClass={math.SG},
      url={https://arxiv.org/abs/2411.03257}, 
note={Preprint, available at arXiv:2411.03257, to appear in \textit{Trans. Amer. Math. Soc.}},
}

@article {Ekholm-Smith,
    AUTHOR = {Ekholm, Tobias and Smith, Ivan},
     TITLE = {Exact {L}agrangian immersions with a single double point},
   JOURNAL = {J. Amer. Math. Soc.},
  FJOURNAL = {Journal of the American Mathematical Society},
    VOLUME = {29},
      YEAR = {2016},
    NUMBER = {1},
     PAGES = {1--59},
      ISSN = {0894-0347,1088-6834},
   MRCLASS = {53D35 (14J70 53D12 53D40)},
  MRNUMBER = {3402693},
MRREVIEWER = {Jelena\ Kati\'c},
       DOI = {10.1090/S0894-0347-2015-00825-6},
       URL = {https://doi.org/10.1090/S0894-0347-2015-00825-6},
}

\end{document}